\documentclass[12pt]{amsart}
\usepackage{amscd,amssymb,amsthm,amsmath,csquotes,graphicx,array,wasysym,comment,epigraph,fancyvrb,mathrsfs,textcomp}
\usepackage[matrix,arrow,curve]{xy}
\usepackage[margin=2cm]{geometry}
\usepackage{calc}
\usepackage{enumitem,xcolor,longtable}
\usepackage[colorlinks, linktocpage, citecolor = blue, linkcolor = blue,backref]{hyperref}
\usepackage[normalem]{ulem}

\makeatletter\@addtoreset{equation}{section}
\makeatother

\newcommand{\R}{\ensuremath{\mathbb{R}}}
\newcommand{\C}{\ensuremath{\mathbb{C}}}
\newcommand{\F}{\ensuremath{\mathbb{F}}}
\newcommand{\p}{\ensuremath{\mathbb{P}}}

\newcommand{\Z}{\ensuremath{\mathbb{Z}}}
\newcommand{\AAA}{\ensuremath{\mathbb{A}}}
\DeclareMathOperator{\Bir}{Bir}
\DeclareMathOperator{\Aut}{Aut}
\DeclareMathOperator{\Pic}{Pic}
\DeclareMathOperator{\Gal}{Gal}
\DeclareMathOperator{\Witt}{W}
\DeclareMathOperator{\PGL}{PGL}
\DeclareMathOperator{\Units}{U}

\DeclareMathOperator{\disc}{disc}
\DeclareMathOperator{\Milnor}{K}
\DeclareMathOperator{\rk}{rk}

\DeclareMathOperator{\Sym}{S}
\DeclareMathOperator{\SF}{\delta}
\DeclareMathOperator{\Spec}{Spec}

\DeclareMathOperator{\QForms}{M}

\newcommand{\kk}{\Bbbk}

\newcommand{\FFF}{\mathbb F}
\newcommand{\Sph}{\mathbb S}

\renewcommand{\epsilon}{\varepsilon}

\theoremstyle{definition}
\newtheorem{example}[equation]{Example}
\newtheorem*{example*}{Example}
\newtheorem{definition}[equation]{Definition}
\newtheorem{definition-example}[equation]{Definition-Example}
\newtheorem{definition-proposition}[equation]{Definition-Proposition}

\newtheorem{theorem}[equation]{Theorem}
\newtheorem{lemma}[equation]{Lemma}
\newtheorem{lemma-def}[equation]{Lemma-Definition}
\newtheorem{corollary}[equation]{Corollary}
\newtheorem{proposition}[equation]{Proposition}

\newtheorem*{conjecture*}{Conjecture}
\newtheorem*{maintheorem*}{Main Theorem}
\newtheorem*{corollary*}{Corollary}
\newtheorem*{maincorollary*}{Main Corollary}
\newtheorem*{maincorollary2*}{Second Main Corollary}
\newtheorem*{secondmaincorollary*}{Second Main Corollary}
\newtheorem{question}[equation]{Question}
\newtheorem*{question*}{Question}

\newtheorem*{problem*}{Problem}
\newtheorem*{theorem*}{Theorem}

\theoremstyle{definition}
\newtheorem{remark}[equation]{Remark}
\newtheorem*{remark*}{Remark}
\newtheorem*{convention*}{Convention}
\newtheorem*{conventions*}{Conventions}

\makeatletter\@addtoreset{equation}{section} \makeatother

\author[I. Cheltsov, F. Mangolte, E. Yasinsky, S. Zimmermann]{Ivan Cheltsov, Fr\'ed\'eric Mangolte, Egor Yasinsky, Susanna Zimmermann}

\title{Birational involutions of the real projective plane}

\address{\emph{Ivan Cheltsov}
\newline
\textnormal{University of Edinburgh,  Edinburgh, Scotland}
\newline
\textnormal{\texttt{i.cheltsov@ed.ac.uk}}}

\address{\emph{Fr\'ed\'eric Mangolte}
\newline
\textnormal{Aix Marseille Univ, CNRS, I2M, Marseille, France}
\newline
\textnormal{\texttt{frederic.mangolte@univ-amu.fr}}}

\address{\emph{Egor Yasinsky}
\newline
\textnormal{Institut de Math\'{e}matiques de Bordeaux, UMR 5251 CNRS, Universit\'{e} de Bordeaux,	33405 Talence cedex, France}
\newline
\textnormal{\texttt{yasinskyegor@gmail.com}}}

\address{\emph{Susanna Zimmermann}
\newline
\textnormal{Université Paris-Saclay, Institut de mathématique, CNRS, UMR 8628, F-91405 Orsay, France}
\newline
\textnormal{\texttt{susannamaria.zimmermann@gmail.com}}}


\begin{document}

\begin{abstract}
We classify birational involutions of the real projective plane up to conjugation.
In contrast with an analogous classification over the complex numbers (due to E. Bertini, G. Castelnuovo, F. Enriques, L. Bayle and A. Beauville),
which includes $4$ different classes of involutions, we discover $12$ different classes over the reals,
and provide many examples when the fixed curve of an involution does not determine its conjugacy class in the real plane Cremona group.
\end{abstract}

\maketitle

\tableofcontents

\section{Introduction}
\label{section:intro}

\subsection{Bertini, Bayle and Beauville}
\label{subsection:BBB}
Classification of finite subgroups in the plane Cremona group $\Bir(\p_\C^2)$, i.e. the group of birational automorphisms of the complex projective plane, is a~very classical problem which goes back to the works of E. Bertini and the Italian school of algebraic geometry.
Bertini was interested in the description of involutions inside this group and discovered three types of them,
which are known today as de Jonqui\`{e}res, Geiser and Bertini involutions.
However, Bertini's classification was known to be incomplete and it required more than a~century to get a~really precise statement.
Using an equivariant version of Mori theory in dimension two, L. Bayle and A. Beauville obtained the following elegant classification:

\begin{theorem}[{\cite{BayleBeauville}}]
\label{theorem:Bayle-Beauville}
Every birational involution in $\Bir(\p_\C^2)$ is conjugate to one and only one of the following involutions:
\begin{enumerate}
\item A linear involution acting on $\p_\C^2$.
\item A de Jonqui\`{e}res involution of genus $g\geqslant 1$, i.e. an involution given in affine coordinates by
$$
\chi: (x,y)\mapsto \left ( x, \frac{f(x)}{y} \right )
$$
with $f(x)$ being a~polynomial of degree $d=2g+1$ 
with no multiple factors.
\item A Bertini involution.
\item A Geiser involution.
\end{enumerate}
\end{theorem}

Except for the case (1), all these involutions have moduli \cite{BayleBeauville}.
Namely, conjugacy classes of de Jonqui\`{e}res involutions of genus $g\geqslant 1$ are parametrized by hyperelliptic curves of genus~$g\geqslant 1$.
Conjugacy classes of Geiser involutions are parametrized by non-hyperelliptic curves of genus~$3$,
and conjugacy classes of Bertini involutions are parametrized by non-hyperelliptic curves of genus~$4$ whose canonical model lies on a~singular quadric.

The goal of this article is to demonstrate that classification of birational involutions of the projective plane over
 the field $\R$ of real numbers is a~much more subtle problem.

\subsection{The real plane Cremona group}

In general, the Cremona group of the real projective plane enjoys many interesting properties
which are quite different from the properties of its complex counterpart.
Recall that the famous theorem of Noether and Castelnuovo states that $\Bir(\p_\C^2)$ is generated by its subgroup $\Aut(\p_\C^2)\cong\mathrm{PGL}_3(\C)$ and the standard Cremona involution
$$
\sigma_0: [x:y:z]\dashrightarrow [yz: xz: xy].
$$
This is not enough to generate $\Bir(\p_\R^2)$: both linear transformations and the Cremona involution have real indeterminacy points,
hence the same is true for any composition of these maps, while the Apollonius' circle inversion
$$
\sigma_1: [x:y:z]\dashrightarrow [y^2+z^2: xy: xz].
$$
is not defined at $[0:1:i]$, $[0:1:-i]$, $[1:0:0]$.
Another example of a~real birational map which cannot be a~composition of linear transformations and $\sigma_0$ can be obtained as follows:
start with three pairs of complex conjugate points on $\p_\R^2$ in general position,
blow them up to get a~real cubic surface and then blow down the strict transforms of the conics passing through five of the six points.
Such birational self-maps of $\p_\R^2$ are called \emph{standard quintic transformations}.
In fact, the group $\Bir(\p_\R^2)$ is generated by $\Aut(\p_\R^2)$, $\sigma_0$, $\sigma_1$ and standard quintic transformations \cite{BlancMangolteGenerators}.

As the group $\Bir(\p_\C^2)$, the real Cremona group   $\Bir(\p_\R^2)$ is also  generated by involutions \cite{ZimmermannPhD}.
On the other hand,  the abelianization of the group $\Bir(\p_\R^2)$ is the infinite sum $\bigoplus_{(0,1]}\mathbb{Z}/2\mathbb{Z}$ \cite{ZimmermannAbel},
which is absolutely not the case over the complex numbers: the abelianization of $\Bir(\p_\C^2)$ is trivial.

In this paper, we are interested in classification of finite subgroups of $\Bir(\p_\R^2)$. Their study was initiated in \cite{YasinskyOdd,Robayo}
and continued in \cite{Yasinsky}.
In those papers, a~big part of the classification was obtained,
including complete classification of finite subgroups of $\Bir_\R(\p_\R^2)$ of odd order.
However, the case of involutions remained unsolved.
The goal of this paper is to fill this gap.

\subsection{Regularisation}
\label{subsection:regularisation}
Recall the general approach to classification of finite subgroups in Cremona groups (for more details see \cite{DolgachevIskovskikh,BlancLinearisation}).
Let $G$ be a~finite subgroup of $\Bir(\p_\R^2)$.
Then there exists an $\R$-rational smooth projective surface $X$, an injective homomorphism $\iota: G\rightarrow \Aut(X)$ and a~birational $G$-equivariant $\R$-map $\psi: X\dasharrow\p^2_\R$,
such that $G=\psi\circ\iota (G)\circ{\psi}^{-1}$. One says that $G$ is \emph{regularised} on $X$.
Conversely, for an $\R$-rational $G$-surface $X$ a~birational map $\psi: X\dasharrow\p^2_\R$ yields an injective homomorphism
$G \hookrightarrow\Bir(\p^2_\R)$ given by $g\mapsto \psi\circ g\circ \psi^{-1}$.

Two finite subgroups of $\Bir(\p^2_\R)$ are conjugate if and only if the corresponding smooth $\R$-rational $G$-surfaces are
$G$-birationally equivalent, so there is a~natural bijection between the conjugacy classes of finite subgroups $G\subset\Bir(\p^2_\R)$
and classes of $G$-birational equivalence of smooth $\R$-rational $G$-surfaces.

Moreover, applying a~$G$-equivariant version of the Minimal Model Program, we may further assume that $X$ admits a~structure of a~\emph{$G$-Mori fibre space}.
This simply means that one of the following two cases hold:
\begin{enumerate}
\item either $\Pic(X)^G\simeq\mathbb{Z}^2$ and there exists a~$G$-equivariant conic bundle $X\to\mathbb{P}^1_{\mathbb{R}}$,
\item or $\Pic(X)^G\simeq\mathbb{Z}$ and $X$ is a~smooth del Pezzo surface.
\end{enumerate}

Therefore, the classification of finite subgroups of $\Bir(\p^2_\R)$ is equivalent to $G$-birational classification of such two-dimensional $\mathbb{R}$-rational $G$-Mori fibre spaces.
In the second case, a~fairly detailed description of such pairs $(X,G)$ has been given in \cite{Yasinsky}.

\subsection{Fixed curves of birational involutions}
\label{subsection:fixed-curves}
For a~birational involution $\iota\in\Bir(\p^2_\R)$ and its regularisation $\tau\in\Aut(X)$ on some smooth $\R$-rational surface $X$,
let us denote by $F(\tau)$ the~union of all geometrically irrational \emph{real} curves in the~surface $X$ that are pointwise fixed by the~involution~$\tau$.
If there are no such curves, we let $F(\tau)=\varnothing$.
In fact, a~non-empty $F(\tau)$ consists of a~unique geometrically irreducible smooth real curve in $X$, see Lemma~\ref{lemma:fixed-curve-unique}.
Moreover, up to an isomorphism, the~fixed curve $F(\tau)$ does not depend on the~choice of the~regularisation,
so~that this real curve depends only on the~conjugacy class of the~involution~$\iota\in\Bir(\mathbb{P}^2_{\mathbb{R}})$.
Thus, we say that $F(\iota)=F(\tau)$ is the~\emph{fixed curve} of the~involutions $\iota$ and $\tau$.

One can always consider $\iota$ as a~birational involution of the~\emph{complex} projective plane.
Then, as an element of the~complex Cremona group $\mathrm{Bir}(\mathbb{P}^2_{\mathbb{C}})$,
the involution $\iota$ is uniquely determined by its fixed curve up to a~conjugation \cite{BayleBeauville}.
In particular, $F(\iota)=\varnothing$ if and only if $\iota$ is conjugated in $\Bir(\p_\C^2)$ to the linear involution $[x:y:z]\mapsto[x:y:-z]$.
In the~real Cremona group $\mathrm{Bir}(\mathbb{P}^2_{\mathbb{R}})$ this is no longer true:
we cannot always recover the conjugacy class of an involution in the~real Cremona group from its fixed curve.

\begin{example}[{\cite[Example 6.1]{Tre16}, \cite[Example 1.4]{Yasinsky}}]
\label{example:Trepalin}
Let $Z_n$ be the~surface in $\p_\R^2\times\p_\R^1$ that is given by
$$
x^2\prod_{k=1}^{2n}(t_0^2+k^2t_1^2)+y^2t_0^{4n}+z^2t_1^{4n}=0,
$$
where $([x:y:z],[t_0:t_1])$ are coordinates on $\p_\R^2\times\p_\R^1$. Then $Z_n$ is rational over $\R$. Let $\tau_n\in\Aut(Z_n)$ be the~involution induced by the~map $[t_0:t_1]\mapsto [-t_0:t_1]$ and let $\iota_n$ be the corresponding birational involution of $\p_\R^2$. Note that $F(\tau_n)=\varnothing$.
On the~other hand, the~involutions $\iota_i$ and $\iota_j$ are non-conjugate in $\Bir(\p_\R^2)$ for $i\ne j$.
See Section~\ref{section:Trepalin-involutions}, for generalizations.
\end{example}

In this paper, we classify birational involutions in $\Bir(\p_\R^2)$,
and construct uncountably many non-conjugate birational involution in $\Bir(\p_\R^2)$ that all have
the same fixed curve: see Main Corollary below for more precise statement.

\subsection{Classification}
\label{subsection:classification}
The main result of this paper, Main Theorem below,
shows that classification of involutions in the~real plane Cremona group is  quite sophisticated and differs drastically from the~one over $\C$.
To state it, we will use the~language of $\R$-rational $G$-surfaces introduced in Section~\ref{subsection:regularisation},
and notions of exceptional/non-exceptional $G$-conic bundles,
see Sections~\ref{section:exceptional-conic-bundles} and \ref{section:non-exceptional} for precise definitions.
So, our approach goes along the~lines of the~approach in \cite{DolgachevIskovskikh,BayleBeauville}.

\begin{maintheorem*}\label{maintheorem}
Let $\iota$ be an involution in $\Bir(\p^2_\R)$.
Then $\iota$ admits a~regularisation $\tau\in\Aut(S)$
such that $S$ is a smooth real projective rational surface,
where $S$ and the subgroup $G=\langle\tau\rangle$ belong to one of the~following classes, which we consider as classes of birational involutions in $\Bir(\p^2_\R)$.
\begin{itemize}
\item[($\mathfrak{L}$)] The surface $S$ is $\mathbb{P}^2_\R$, $\tau$ is the linear involution $[x:y:z]\mapsto[x:y:-z]$, and $F(\tau)=\varnothing$.

\item[($\mathfrak{Q}$)] 
The surface $S$ is a quadric surface in $\p^3_\R$, $\tau$ is an automorphism, $F(\tau)=\varnothing$ and $S(\R)^{\tau}=\varnothing$. See Section \ref{section: dP8}.

\item[($\mathfrak{T}_{4n}$)] The surface $S$ admits a $G$-equivariant morphism $S\to\mathbb{P}^1_\R$
that is a conic bundle with $4n\geqslant 4$ singular fibres, $\Pic(S)^G\simeq\mathbb{Z}^2$,
the involution $\tau$ is a $0$-twisted Trepalin involution, and $F(\tau)=\varnothing$. See Section~ \ref{subsection:Trepalin-involution-1}.

\item[($\mathfrak{T}_{4n+2}^\prime$)] The surface $S$ admits a $G$-equivariant morphism $S\to\mathbb{P}^1_\R$
that is a conic bundle with~\mbox{$4n+2\geqslant 6$} singular fibres, $\Pic(S)^G\simeq\mathbb{Z}^2$,
$\tau$ is a~$1$-twisted Trepalin involution, and $F(\tau)=\varnothing$. See Section~ \ref{subsection:Trepalin-involution-2}.

\item[($\mathfrak{T}_{4n}^{\prime\prime}$)] The surface $S$ admits a $G$-equivariant morphism $S\to\mathbb{P}^1$
that is a conic bundle with $4n\geqslant 4$ singular fibres, $\Pic(S)^G\simeq\mathbb{Z}^2$,
$\tau$ is a~$2$-twisted Trepalin involution, and $F(\tau)=\varnothing$. See Section~ \ref{subsection:Trepalin-involution-3}.

\item[($\mathfrak{B}_4$)] The surface $S$ is a del Pezzo surface of degree $1$,
$\Pic(S)^G\simeq\mathbb{Z}$,
the involution $\tau$ is the~Bertini involution of the~surface $S$, and $F(\tau)$ is a~non-hyperelliptic curve of genus $4$. See Section~\ref{section:Bertini}.
		
\item[($\mathfrak{G}_3$)] The surface $S$ is a~del Pezzo surface of degree $2$, $\Pic(S)^G\simeq\mathbb{Z}$,
the involution $\tau$ is the~Geiser involution of the surface $S$, and $F(\tau)$ is a~non-hyperelliptic curve of genus $3$. See Section~\ref{section:Geiser}.
		
\item[($\mathfrak{K}_1$)] The surface $S$ is a~del Pezzo surface of degree $2$, $\Pic(S)^G\simeq\mathbb{Z}$,
the involution $\tau$ is a~Kowalevskaya involution of the surface $S$, and $F(\tau)$ is a genus 1 curve. See Section~\ref{section:Geiser}.
				
\item[($\mathfrak{dJ}_{g}$)] The surface $S$ admits a $G$-equivariant morphism $S\to\mathbb{P}^1_\R$
that is a $G$-exceptional conic bundle with $2g+2$ singular fibres, $\Pic(S)^G\simeq\mathbb{Z}^2$,
the involution $\tau$ is a de~Jonqui\`{e}res involution, and $F(\tau)$ is a~hyperelliptic curve $C$ of genus $g\geqslant 1$. We give the definition after Proposition~ \ref{proposition:exceptional-conic-bundles}.

\item[($\mathfrak{I}_{g}$)] The surface $S$ admits a $G$-equivariant morphism $S\to\mathbb{P}^1_\R$
that is a non-$G$-exceptional conic bundle with $2g+2$ singular fibres, $\Pic(S)^G\simeq\mathbb{Z}^2$,
the involution $\tau$ is a $0$-twisted Iskovskikh involution,
and $F(\tau)$ is a~hyperelliptic curve $C$ of genus $g\geqslant 1$. See Section~\ref{subsection:Iskovskikh-involutions}.

\item[($\mathfrak{I}_{g}^\prime$)] The surface $S$ admits a $\langle\tau\rangle$-equivariant morphism $S\to\mathbb{P}^1_\R$
that is a non-$G$-exceptional conic bundle with $2g+3$ singular fibres, $\Pic(S)^G\simeq\mathbb{Z}^2$,
the~involution $\tau$ is a~$1$-twisted Iskovskikh involution,
and $F(\tau)$ is a~hyperelliptic curve $C$ of genus $g\geqslant 1$. See Section~\ref{subsection:Iskovskikh-involutions}.
		
\item[($\mathfrak{I}_{g}^{\prime\prime}$)] The surface $S$ admits a $G$-equivariant morphism $S\to\mathbb{P}^1_\R$
that is a non-$G$-exceptional  conic bundle with $2g+4$ singular fibres, $\Pic(S)^G\simeq\mathbb{Z}^2$,
the involution $\tau$ is a~$2$-twisted Iskovskikh involution,
and $F(\tau)$ is a~hyperelliptic curve $C$ of genus $g\geqslant 1$. See Section~\ref{subsection:Iskovskikh-involutions}.
\end{itemize}

Moreover, two birational involutions contained in two distinct classes
$\mathfrak{L}$, $\mathfrak{Q}$, $\mathfrak{T}_{4n}$, $\mathfrak{T}_{4n+2}^\prime$, $\mathfrak{T}_{4n}^{\prime\prime}$,
$\mathfrak{B}_4$, $\mathfrak{G}_3$, $\mathfrak{K}_1$, $\mathfrak{dJ}_{g}$,
$\mathfrak{I}_{g}$, $\mathfrak{I}_{g}^\prime$, $\mathfrak{I}_{g}^{\prime\prime}$ are not conjugate in $\Bir(\p^2_\R)$
with the only possible exception: involutions from $\mathfrak{dJ}_{1}$ can be conjugate to involutions from $\mathfrak{I}_{1}$ (see Example~\ref{example:non-exceptional-exceptional}). In fact, we do not know whether these two classes coincide.
\end{maintheorem*}

We refer the~reader to Sections~\ref{section:Bertini}, \ref{section:Geiser},
\ref{section:exceptional-conic-bundles}, \ref{section:non-exceptional},
for a~more explicit description of the~birational involutions in the~classes
$\mathfrak{L}$, $\mathfrak{Q}$, $\mathfrak{T}_{4n}$, $\mathfrak{T}_{4n+2}^\prime$, $\mathfrak{T}_{4n}^{\prime\prime}$, $\mathfrak{B}_4$, $\mathfrak{G}_3$, $\mathfrak{K}_1$, $\mathfrak{dJ}_{g}$,
$\mathfrak{I}_{g}$, $\mathfrak{I}_{g}^\prime$, $\mathfrak{I}_{g}^{\prime\prime}$. In what follows, the involutions in the class $\mathfrak{L}$ (i.e. those which are conjugate to $[x:y:z]\mapsto [x:y:-z]$ in $\Bir(\p_{\R}^2)$) will be called \emph{linearizable}. In Section~ \ref{subsection:conjugation}, we also explain how to determine whether two given birational involutions contained in one class among
$\mathfrak{T}_{4n}$ with $n\geqslant 2$, $\mathfrak{T}_{4n+2}^\prime$, $\mathfrak{T}_{4n}^{\prime\prime}$ with $n\geqslant 2$,
$\mathfrak{B}_4$, $\mathfrak{G}_3$, $\mathfrak{K}_1$, $\mathfrak{dJ}_{g}$ with $g\geqslant 2$,
$\mathfrak{I}_{g}$ with $g\geqslant 2$, $\mathfrak{I}_{g}^\prime$, $\mathfrak{I}_{g}^{\prime\prime}$ are conjugate or not.
Using this, in Section \ref{subsection:main-corollary} we obtain

\begin{maincorollary*}\label{maincorollary}
Let $C$ be a~real smooth projective hyperelliptic curve of genus $g\geqslant 2$ such that the~locus $C(\R)$ consists of at least $2$ connected components.
Then $\Bir(\p_\R^2)$ contains uncountably many non-conjugate involutions that all fix a~curve isomorphic to the~curve $C$. Besides, the real plane Cremona group $\Bir(\p_\R^2)$ contains uncountably many non-conjugate involutions that fix no geometrically irrational curves.
\end{maincorollary*}

Recently there have been new developments in the study of embedding of finite groups into Cremona groups. In particular, new obstructions and invariants were introduced in the works of Kontsevich-Pestun-Tschinkel \cite{KontsevichPestunTschinkel}, Kresch-Tschinkel \cite{KreschTschinkel1,KreschTschinkel2} and Hassett-Kresch-Tschinkel \cite{HassettKreschTschinkel}; see also references therein. It would be very interesting to apply these new techniques to conjugacy problems in Cremona groups over non-closed fields, including the case of real numbers, considered in this paper.   

\subsection*{Acknowledgement}
The authors would like to thank Ilia Itenberg, Viatcheslav Kharlamov for very fruitful discussion of real plane curves, and Alexander Merkurjev, Alexander Vishik for discussions on quadratic forms and Witt's theory. The authors also thank Igor Dolgachev for providing useful references and Andrey Trepalin for the help with the proof of Lemma \ref{lem: KS2=4}.
The third author would like to thank the~University of Basel for the~wonderful years spent there and especially J\'er\'emy Blanc for hospitality and numerous discussions (also related to this paper). 
The last author would like to thank the Laboratoire Angevin de Recherche en Mathématiques, University of Angers, for the support she received during the initial stages of this project. 
This project was initiated at the~conference ``Algebraic Geometry in Auckland" and we thank Sione Ma'u and the~University of Auckland for the~stimulating atmosphere. This project was completed during first author's stay at l'Institut des Hautes \'{E}tudes Scientifiques. We thank the institute for the hospitality. Finally, we would like to thank the anonymous referees for their detailed reports and extremely helpful comments on an earlier version of this paper.

During this project, Susanna Zimmermann was supported by the ANR Project FIBALGA ANR-18-CE40-0003-01, the Projet PEPS 2019 ``JC/JC'', the Project \'Etoiles montantes of the R\'egion Pays de la Loire, the project ANR-22-ERCS-0013-01 and by the Centre Henri Lebesgue.  
Ivan Cheltsov was supported by the Leverhulme Trust grant RPG-2021-229. 
Fr\'ed\'eric Mangolte was partially supported by the ANR grant ENUMGEOM ANR-18-CE40-0009.

\section{Equivariant Minimal Model Program}
\label{section:MMP}

In this section, we collect some preliminary information about real algebraic surfaces, admitting a~structure of a~$G$-Mori fibre space.
Namely, let $S$ be a~real geometrically rational smooth projective surface such that $S(\R)\ne\varnothing$.
Then the~Galois group $\Gal(\C/\R)$ acts on $S_\C=S\times_{\Spec\R}\Spec\C$ by the~complex conjugation.

Fix a~finite subgroup $G\subset\mathrm{Aut}(S)$.
We say that $S$ is a~\emph{$G$-minimal del Pezzo surface} if $S$ is a~del Pezzo surface, i.e. the~divisor $-K_S$ is ample, and $\Pic(S)^G\simeq\mathbb{Z}$.
Similarly, we say that a~$S$ admits a~\emph{$G$-conic bundle structure}
if there exists a~flat surjective morphism $\pi\colon S\to\p_\R^1$ such that each scheme fibre of $\pi$ is isomorphic to a~reduced conic in $\p_{\R}^2$. It will be called a~\emph{$G$-minimal} conic bundle if $\Pic(S)^G\simeq\mathbb{Z}^2$.
Note that a~$G$-minimal conic bundle is not necessarily $G$-minimal in the~sense of \cite{DolgachevIskovskikh}. For example, a blow-up of a $G$-fixed real point on $\p_\R^2$ is a $G$-minimal conic bundle $\F_1\to\p_\R^1$ (with no singular fibres), which is clearly not $G$-minimal as a $G$-surface.

As we already mentioned in Section~\ref{subsection:regularisation}, in the remaining part of the paper, we may work with these two classes of surfaces, $G$-minimal del Pezzo surfaces and $G$-minimal conic bundles.

\subsection{Sarkisov theory}
The main tool for exploring conjugacy in Cremona groups is the~\emph{Sarkisov program}.
For example, it will be especially useful for us to be able to decompose the~birational maps \eqref{equation:fibrewise-map}
into a~sequence of \emph{elementary transformations}, Sarkisov links of type II. The main reference used in this paper is a~classical treatment by Iskovskikh \cite{Iskovskikh1996},
where $G$ is assumed to be the (absolute) Galois group of the base field.
It is straightforward to transfer the whole theory to the mixed case, i.e. when both geometric and the Galois group come into the picture; see e.g. \cite[Section 2]{SchneiderZimmermann} for more details about this setting.

So, we work in the category of $G$-surfaces over $\R$.
Then any birational $G$-map between two $G$-surfaces can be decomposed into a~sequence of birational $G$-morphisms and their inverses. One can blow up $G$-orbits of real points and pairs of complex conjugate points. In this paper, we work with rational $G$-minimal del Pezzo surfaces and $G$-minimal conic bundles (in the sense defined above).
They are rational $G$-Mori fibre spaces $\pi\colon S\to C$,
where $C$ is a~point in the del Pezzo case, and $C=\mathbb{P}^1_{\mathbb{R}}$ in the conic bundle case. Let $\pi\colon S\to C$ and $\pi'\colon S'\to C'$ be two-dimensional $G$-Mori fibre spaces.
Then every $G$-birational map $f\colon S\dashrightarrow S'$
can be factorized into a~composition of \emph{Sarkisov links} of four types whose complete description can be found in \cite{Iskovskikh1996}.

\begin{lemma}\label{rem:K2=4-and-rkPic=2}
	Suppose that $\pi: S\to\p^1_\R$ is a $G$-minimal conic bundle and $\varphi\colon S\dashrightarrow S'$ a $G$-birational map to a $G$-minimal conic bundle. Assume that $K_S^2\leqslant 4$. Then one has the following:
	\begin{enumerate}
		\item If $K_S^2\leqslant 0$, then $\varphi$ is a composition of elementary transformations of conic bundles. 
		\item\label{K2=4-and-rkPic=2:2} If $K_S^2=1$ or $K_S^2=2$, then $\varphi$ is a composition of elementary transformations of conic bundles and Sarkisov links of type IV, exchanges of two conic bundle structures (note that in this case, $S$ is a del Pezzo surface and the link is represented by the Bertini and Geiser involutions, respectively).
		\item If $K_S^2=3$, then $\varphi$ is a composition of elementary transformations of conic bundles and birational transformations $\eta_1\circ\alpha\circ\eta_2$, where $\eta_2$ is a Sarkisov link of type III --- a contraction of a $G$-invariant real line to a quartic surface $T$, the map $\alpha$ is a (possibly trivial) composition of birational Bertini or Geiser involutions, and $\eta_1$ is a blow-up of a $G$-fixed point on $T$ --- a Sarkisov link of type I.
		\item\label{K2=4-and-rkPic=2:4} If $K_S^2=4$, then $\varphi$ is a composition of elementary transformations of conic bundles and Sarkisov links of type IV.
	\end{enumerate}
In particular, one has $K_{S'}^2=K_S^2$. 
\end{lemma}
\begin{proof}
		Follows from the classification of Sarkisov links \cite[Theorem 2.6]{Iskovskikh1996} and Theorem \ref{theorem:Iskovskikh}.
\end{proof}

\subsection{Real rational surfaces}\label{subsec: real rational surfaces}
As above, let $S$ be a~real geometrically rational algebraic surface such that $S(\R)\ne\varnothing$.
Recall that the~surface $S$ is said to be $\R$-rational if there exists a~birational map $S\dasharrow\p_{\R}^2$.
We have the~following  rationality criterion:

\begin{theorem}[{see e.g. \cite[Corollary VI.6.5]{Si89}}]
\label{theorem:rational-real}
A real geometrically rational smooth projective surface $S$ is $\R$-rational if and only if its real locus $S(\R)$ is non-empty and connected (in the Euclidean topology).
\end{theorem}

In what follows, we denote by $Q_{3,1}$ the smooth quadric surface
$
\{x^2+y^2+z^2=w^2\}\subset\p_\R^3,
$
and we denote by $Q_{2,2}$ the smooth quadric surface
$
\{x^2+y^2=z^2+w^2\}\subset\p_\R^3.
$
For a real del Pezzo surface $S$, we denote by $S(a,b)$ a blow-up of $S$ at $a$ real points and $b$ pairs of complex conjugate points. If $S$ is an $\R$-rational del Pezzo surface, its real locus $S(\R)$ is diffeomorphic to one of the~following manifolds:
\begin{enumerate}
\item a sphere $\Sph^2$ if $X\simeq Q_{3,1}(0,b)$;
\item a torus $\Sph^1\times\Sph^1$ if $X\simeq Q_{2,2}(0,b)$;
\item a connected sum $N^g=\#^g\R\p^2$ if $X\simeq\p_\R^2(a,b)$ where $g=a+1$ and $0\leqslant a+2b\leqslant 8$.
\end{enumerate}
For more details, see \cite{Kollar97} or \cite{Mangolte}.

\begin{proposition}[{\cite[Theorem 2.2]{Kollar97}, \cite[Theorem 4.4.14]{Mangolte}}]
\label{proposition:conic-bundle-real-loci}
Let $\pi: S\to\p_\R^1$ be a real conic bundle such that $\Pic(S)\simeq\mathbb{Z}^2$ (note that this condition implies that all singular fibres of $\pi$ are real, i.e. lie over points of $\p_\R^1(\R)$). If $\pi$ has singular fibres, then it has an even number $2m$ of them and
	\[
	S(\R)\approx\bigsqcup_{i=1}^m\Sph^2.
	\]
	Otherwise, $S(\R)$ is either a~torus $\Sph^1\times\Sph^1$ or a~Klein bottle $\#^2\mathbb{RP}^2$. Moreover, $S$ is $\R$-rational if and only if it is isomorphic to one of the~following surfaces: a~real del Pezzo surface $Q_{3,1}(0,1)$ of degree 6 (then $S(\R)\approx\Sph^2$), a~real Hirzebruch surface $\FFF_{2n+1}$ (then $S(\R)\approx\#^2\mathbb{RP}^2$) or a~real Hirzebruch surface $\FFF_{2n}$ (then $S(\R)\approx\Sph^1\times\Sph^1$).
\end{proposition}

\begin{remark}
\label{remark:conic-bundle-intervals}
If the~conic bundle $\pi\colon S\to\p^1_\R$ in Proposition~\ref{proposition:conic-bundle-real-loci}
has at least one singular fibre, then $\pi(S(\R))$ is a~union of intervals and their boundary points are exactly the~images of real singular fibres whose components are permuted by the~Galois group $\Gal(\C/\R)$. If $S=Q_{3,1}(0,1)$ then $\pi(S(\R))$ is an interval in $\p_\R^1(\R)$. If $S=\FFF_n$ then $\pi(S(\R))=\p_\R^1(\R)$.
\end{remark}

Finally, there is a characterization of minimal rational surfaces due to Iskovskikh:

\begin{theorem}[{\cite[\S 4]{Iskovskikh1996}}, {\cite[Proposition 2.15]{BlancMangolteGenerators}}]\label{thm: IskovskikhCrit}
	A minimal geometrically rational surface $X$ over $\R$ is $\R$-rational if and only if the following two conditions are satisfied:
	\begin{description}
		\item[(i)] $X(\R)\ne\varnothing$;
		\item[(ii)] $d=K_X^2\geqslant 6$.
	\end{description}
\end{theorem}

\subsection{Fixed curves of birational maps}
In the~remaining part of this section, let us assume that $S$ is a~smooth projective $\R$-rational surface and $\tau\in\Aut(S)$ is an involution.

\begin{lemma}
\label{lemma:fixed-curves-smooth}
Let $C_1,\ldots,C_n$ be irreducible curves on $S$ which are pointwise fixed by $\tau$. Then each $C_i$ is smooth, and $C_i\cap C_j=\varnothing$ for $i\ne j$.
\end{lemma}

\begin{proof}
This follows from the main result of \cite{FogartyNorman}, as finite groups in characteristic zero are linearly reductive.
\end{proof}

The following definition plays the~key role in this paper.

\begin{lemma-def}[\textbf{The fixed curve of a~birational involution} {\cite{deFernex,BlancLinearisation}}]
Consider a~birational involution $\iota\in \Bir(\p^2_\R)$ and its regularisation $\tau\in\Aut(S)$.
Denote by $F(\tau)$ the~union of all geometrically non-rational real curves on the~surface $S$ which are pointwise fixed by $\tau$.
If there are no such curves, we let $F(\tau)=\varnothing$. Then
\begin{itemize}
\item either $F(\tau)=\varnothing$,
\item or $F(\tau)$ consists of a~unique geometrically irreducible smooth real curve in $S$ by Lemma~\ref{lemma:fixed-curve-unique} below.
\end{itemize}
Moreover, the~fixed curve $F(\tau)$ does not depend on the~choice of the~regularisation,	so~that this real curve depends only on the~conjugacy class of the~birational involution~$\iota\in\Bir(\p^2_\R)$.
Thus, we say that $F(\tau)$ is the~\emph{fixed curve} of the~involutions $\iota$ and $\tau$.
\end{lemma-def}

\begin{lemma}
\label{lemma:fixed-curve-unique}
Suppose that $F(\tau)\ne\varnothing$. Then $F(\tau)$ consists of one smooth geometrically irreducible curve. Moreover, if $\pi: S\to\p^1_\R$ is a~conic bundle then $\tau$ acts trivially on the~base $\p^1_\R$ and $F(\tau)$ is a~double section of $\pi$.
\end{lemma}

\begin{proof}
If $\tau$ acts on a~del Pezzo surface with $\Pic(S)^{\langle\tau\rangle}\simeq\mathbb{Z}$ then the~assertion follows from Lemma~\ref{lemma:fixed-curves-smooth}. Thus, we may assume that there exists a~$\tau$-equivariant conic bundle $\pi\colon S\to\p_\R^1$ such that $\Pic(S)^{\langle\tau\rangle}\simeq\mathbb{Z}^2$. Let $\mathcal{C}$ be a~curve in $S$ such that each irreducible component of the~curve $\mathcal{C}_\C$ is irrational and pointwise fixed by $\tau$.
Then $\mathcal{C}$ is smooth by Lemma~\ref{lemma:fixed-curves-smooth},
and moreover, it must be a~multi-section of the~conic bundle $\pi$, so that the~action of the~involution $\tau$ on the~base of the~conic bundle $\pi$ is trivial.
Hence, the~involution $\tau$ acts faithfully on a~general fibre of $\pi$, so it fixes two points on it.
This implies that $\mathcal{C}$ is an irreducible component of a double section of $\pi$. Since all components of $\mathcal{C}$ are geometrically irrational, it follows that $\mathcal{C}$ is irreducible and a double section.
\end{proof}

\subsection{Birational rigidity}
\label{subsection:rigidity}
To get the~conjugacy classes of involutions in $\Bir(\p_\R^2)$,
one should classify $G$-Mori fibre spaces up to $G$-birational equivalence for $G=\langle\tau\rangle$. The results recalled in this section are valid over any perfect field. In this subsection, we denote by $\kk$ a perfect field and any birational map and any surface in this subsection is defined over $\kk$, as in \cite{Iskovskikh1996}.
Some $G$-Mori fibre spaces admit very few $G$-birational maps:

\begin{theorem}[\bf{Manin--Segre}]
\label{theorem:Manin-Segre}
Let $S$ be a~smooth del Pezzo surface, and let $G$ be a~finite subgroup in $\Aut(S)$ such that $\Pic(S)^G\simeq\mathbb{Z}$.
If $K_S^2\leqslant 3$, then $S$ is the~only $G$-Mori fibre space that is $G$-equivariantly birational to $S$.
If $K_S^2=1$, then any $G$-equivariant map $S\dashrightarrow S$ is an automorphism.
\end{theorem}

\begin{proof}
Any $G$-birational map to another $G$-Mori fibre space $S'$ decomposes into $G$-equivariant Sarkisov links and isomorphisms (see e.g. \cite[Appendix]{Co95} or \cite{Iskovskikh1996}).
For any link $\chi\colon S\dashrightarrow S'$, there exists a~commutative diagram
$$
\xymatrix{
&S''\ar@{->}[dl]_{\alpha}\ar@{->}[dr]^{\alpha'}&\\
S\ar@{-->}[rr]^{\chi}&& S'}
$$
where $\alpha,\alpha'$ are birational morphisms and $S',S''$ are del Pezzo surfaces as well. This yields the~claim if $K_S^2=1$.
If $K_S^2\in\{2,3\}$, then up to automorphisms of $S$, any such link is a~Bertini or Geiser involution \cite[Theorem 2.6]{Iskovskikh1996},
i.e. there exists a~biregular involution $\sigma\in\Aut(S^{\prime\prime})$ and an automorphism $\delta\in\Aut(S)$ such that
$\chi=\alpha'\circ\sigma\circ\alpha^{-1}\circ\delta$.
Thus, in particular, we have $S'\simeq S$. Moreover, since $\sigma$ centralizes $G$, we conclude that $S'$ is $G$-equivariantly isomorphic to $S$, cf. \cite{LucasMirko}.
\end{proof}

The assertion of this result brings us to

\begin{definition}[{\cite[Definition 3.1.1]{CS}}]
\label{definition:rigid}
Let $S$ be a~del Pezzo surface, and let $G$ be a~finite subgroup in $\Aut(S)$ such that~\mbox{$\Pic(S)^G\simeq\mathbb{Z}$}. One says that $S$ is $G$-birationally rigid if it is the~only $G$-Mori fibre space that is \mbox{$G$-equivariantly birational} to $S$.
If in addition, we have $\Bir^G(S)=\Aut^G(S)$,
then we say that $S$ is $G$-birationally super-rigid.
\end{definition}

Here $\Bir^G(S)$ is the~subgroup in $\Bir(S)$ consisting of all $G$-birational selfmaps of the~surface $S$,
which is isomorphic to the~normalizer of the~subgroup $G$ in $\Bir(S)$.

Recall that a \emph{weak} del Pezzo surface is a nonsingular surface $S$ with $-K_S$ nef and big.

\begin{theorem}[{\cite{Iskovskikh1967}}]
\label{theorem:Iskovskikh}
Let $S$ be a~smooth projective surface, and let $G$ be a~finite subgroup in $\Aut(S)$.
Suppose that there is a~$G$-conic bundle $\pi\colon S\to\mathbb{P}^1_{\kk}$ such that $\Pic(S)^G\simeq\mathbb{Z}^2$,
and $K_S^2\leqslant 0$. Then the following two assertions hold:
\begin{itemize}
\item[(i)] $S$ is not $G$-birational to a smooth weak del Pezzo surface;
\item[(ii)] for every~$G$-birational map $S\dasharrow\widehat{S}$ defined over $\kk$ such that
there is a~$G$-conic bundle $\widehat{\pi}\colon\widehat{S}\to\p^1_\kk$,
there exists a~$G$-equivariant commutative diagram
\begin{equation}\label{equation:fibrewise-map}
\xymatrix{
S\ar@{->}[d]_{\pi}\ar@{-->}[rr]&& \widehat{S}\ar@{->}[d]^{\widehat{\pi}}\\
\p^1_{\kk}\ar@{->}[rr]^\upsilon&&\p^1_{\kk}}
\end{equation}
for some isomorphism $\upsilon\colon\p^1_{\kk}\to\p^1_{\kk}$.
\end{itemize}
\end{theorem}

\begin{proof}
We only prove the~assertion (i), since the~proof of the~assertion (ii) is similar~\cite{Iskovskikh1967}.
Suppose that there exists a~$G$-equivariant birational map $\chi\colon S\dashrightarrow X$, where $X$ is a~smooth weak del Pezzo surface.
Let us seek for a~contradiction.

Let $\mathcal{M}_X=|-nK_X|$ for $n\gg 0$, let $\mathcal{M}_S$ be its proper transform on the~surface $S$.
Now,~choose $\lambda\in\mathbb{Q}_{>0}$ such that $K_S+\lambda \mathcal{M}_S\sim_{\mathbb{Q}}\pi^*(D)$
for some $\mathbb{Q}$-divisor $D$ on~$\mathbb{P}^1$. Such $\lambda$ does exist, because $\Pic(S)^G\simeq\mathbb{Z}^2$. Then
$$
0\leqslant \lambda^2 M_1\cdot M_2=\Big(\pi^*(D)-K_S\Big)^2=-2K_S\cdot\pi^*(D)+K_S^2=4\mathrm{deg}\big(D\big)+K_S^2\leqslant4\mathrm{deg}\big(D\big)
$$
for two general curves $M_1$ and $M_2$ in $\mathcal{M}_S$.
Hence, we see that $\mathrm{deg}(D)\geqslant 0$.

Observe that there exists a $G$-equivariant commutative diagram
$$
\xymatrix{
&Y\ar@{->}[dl]_{\alpha}\ar@{->}[dr]^{\beta}&\\
S\ar@{-->}[rr]^{\chi}&& X}
$$
where $Y$ is a~smooth surface, and $\alpha$ and $\beta$ are $G$-equivariant birational morphisms. 
Let $\mathcal{M}_Y$ be the~proper transform of the~linear system $\mathcal{M}_X$ on the~surface $Y$.
Then
$$
\alpha^*\big(K_S+\lambda\mathcal{M}_S\big)+\sum_{i=1}^k a_iE_i\sim_{\mathbb{Q}} K_Y+\lambda\mathcal{M}_Y\sim_{\mathbb{Q}}\beta^*\big(K_X+\lambda\mathcal{M}_X\big)+\sum_{i=1}^m b_iF_i
$$
for some rational numbers $a_1,\ldots,a_k,b_1,\ldots,b_m$, where each $E_i$ is an $\alpha$-exceptional curve that is irreducible over $\mathbb{R}$,
and each $F_i$ is a~$\beta$-exceptional curve that is irreducible over $\mathbb{R}$.
Moreover, since $\mathcal{M}_X$ is base point free, we see that $b_i>0$.
Thus, we have
\begin{equation}
\label{equation:conic-bundle-del-Pezzo}
(\pi\circ\alpha)^*\big(D\big)+\sum_{i=1}^k a_iE_i\sim_{\mathbb{Q}}(\lambda n-1)\beta^*\big(-K_X\big)+\sum_{i=1}^m b_iF_i.
\end{equation}
If $\lambda>\frac{1}{n}$, then the~divisor in the~right hand side of \eqref{equation:conic-bundle-del-Pezzo} is big,
while the~divisor in the~left hand side of \eqref{equation:conic-bundle-del-Pezzo} is not big.
Hence, we conclude that $\lambda\leqslant\frac{1}{n}$.

Suppose that the~singularities of the~(mobile) log pair $(S,\lambda\mathcal{M}_S)$ are not canonical.
Then it follows from \cite[Proof of Theorem 5.4 for surfaces]{Co95} that there is a~$G$-equivariant commutative diagram
$$
\xymatrix{
S\ar@{->}[d]_{\pi}\ar@{-->}[rr]^{\zeta}&& \widehat{S}\ar@{->}[d]^{\widehat{\pi}}\\
\mathbb{P}^1_\kk\ar@{=}[rr]&&\mathbb{P}^1_\kk}
$$
such that $\zeta$ is a~birational map, the~surface $\widehat{S}$ is smooth, $\widehat{\pi}$ is a~conic bundle,
$\Pic(\widehat{S})^G\simeq\mathbb{Z}^2$, and the~log pair $(\widehat{S},\lambda\mathcal{M}_{\widehat{S}})$
has at most canonical singularities, where $\mathcal{M}_{\widehat{S}}$ is the~proper transform of the~linear system $\mathcal{M}_S$
on the~surface $\widehat{S}$. Note also that $K_{\widehat{S}}^2=K_S^2\leqslant 0$,
because $\delta$ can be decomposed into elementary transformations, see \cite{Iskovskikh1996} or \cite[Appendix]{Co95}.
Therefore, we can swap the~surfaces $S$ and $\widehat{S}$ and replace the~birational map $\chi$ by $\chi\circ\zeta^{-1}$.
Then we may assume that the~log pair $(S,\lambda\mathcal{M}_S)$ has canonical singularities.

Since $(S,\lambda\mathcal{M}_S)$ has canonical singularities, all the~numbers $a_1,\ldots,a_n$ are non-negative.
Since $\mathrm{deg}(D)\geqslant 0$, in the~left hand side of \eqref{equation:conic-bundle-del-Pezzo}, we have a~pseudoeffective divisor,
which implies that $\lambda=\frac{1}{n}$ and $D\sim_{\mathbb{Q}} 0$.
Now, since the~intersection form of the~curves $E_1,\ldots,E_k$ and the~intersection form of the~curves $F_1,\ldots,F_m$ are negative definite, it follows from \eqref{equation:conic-bundle-del-Pezzo} and \cite[Lemma 2.19]{KollarEtAl} that
$$
\sum_{i=1}^k a_iE_i=\sum_{i=1}^{m}b_iF_i.
$$
Therefore, since all $b_i$ are positive, all $\beta$-exceptional curves are $\alpha$-exceptional. Hence $\chi^{-1}$ is a birational morphism, which implies that $K_S^2\geqslant K_X^2\geqslant 1$. Since we assume that $K_S^2\leqslant 0$, this is a contradiction.
\end{proof}

\section{Involutions on quadric surfaces}\label{section: dP8}

This short section is devoted to the study of involutions on real del Pezzo surfaces of degree 8, i.e. to the class $\mathfrak{Q}$ of the Main Theorem. Recall that in Section \ref{subsec: real rational surfaces} we introduced two quadric surfaces in $\p_\R^3$:
\[
Q_{3,1}\colon x^2+y^2+z^2=w^2\ \ \text{and}\ \ Q_{2,2}\colon x^2+y^2=z^2+w^2.
\]

\begin{proposition}\label{prop: K2=8}
	Let $S$ be a real $\R$-rational $G$-Mori fibre space of dimension~2 with $K_S^2=8$, and $\tau$ be the generator of $G$. Then either $\tau$ is linearisable, i.e. is conjugate to a linear involution of $\p_\R^2$ in $\Bir(\p_\R^2)$, or $S$ is $G$-birational to one of the following pairs:
		\begin{enumerate}
			\item $S\simeq Q_{3,1}$, $\tau$ is the antipodal involution $[w:x:y:z]\mapsto [-w:x:y:z]$. 
			\item $S\simeq Q_{2,2}\simeq\p_\R^1\times\p_\R^1$ and $\tau$ acts by $([x:y],[s:t])\mapsto ([x:y]:[t:-s])$. 
		\end{enumerate}
Moreover, these two are non-$G$-birationally isomorphic pairs.
\end{proposition}

\begin{proof}
	If $S$ is a del Pezzo surface with $\Pic(S)^G\simeq\Z$ then $S\simeq Q_{3,1}$ or $S\simeq Q_{2,2}$. If $S$ is a $G$-minimal conic bundle, then $S\simeq\FFF_n$ for some $n\geqslant 0$. In this case, using elementary transformations one can construct a $G$-equivariant commutative diagram
	\[
	\xymatrix{
		S\ar@{->}[d]_{\pi}\ar@{-->}[rr]^{\chi}&&\mathbb{F}_m\ar@{->}[d]^{\eta}\\
		\mathbb{P}^1_\R\ar@{=}[rr]&&\mathbb{P}^1_\R}
	\]
	such that $\chi$ is a birational map, $\eta$ is a natural projection, and $m\in\{0,1\}$. If $m=1$, then $\tau$ is conjugate to a linear involution of $\p_\R^2$. If $m=0$, then $\FFF_m\simeq Q_{2,2}\simeq\p_\R^1\times\p_\R^1$. So, it remains to study the involutions on $Q_{3,1}$ and $Q_{2,2}$. Then $\tau$ extends to a regular involution of $\p^3_\R$ preserving the associated quadratic form. This yields the following possibilities:
	
	\begingroup
	\renewcommand*{\arraystretch}{1.3}
	\begin{longtable}{|c|c|c|c|}
		\hline
		$S$ & $\tau$ & $S^\tau$ & $S^\tau(\R)$ \\ \hline\hline	
		$Q_{3,1}$ & $[w: x: y: z]\mapsto [w: -x:-y:z]$ & $4$ points & $2$ points  \\ \hline	
		$Q_{3,1}$ & $[w: x: y: z]\mapsto [w: -x:y:z]$ & \{ $y^2+z^2=w^2$\} & $\Sph^1$  \\ \hline	
		$Q_{3,1}$ & $[w: x: y: z]\mapsto [-w:x:y:z]$ & \{ $x^2+y^2+z^2=0$ \} & $\varnothing$  \\ \hline	
		$Q_{2,2}$ & switching the factors & the diagonal & $\Sph^1$  \\ \hline
		$Q_{2,2}$ & fibrewise (see below) & 4 points & 0, 2 or 4 points  \\ \hline
	\end{longtable}
	\endgroup  
	Here and in what follows, $S^\tau$ denotes the fixed locus of $\tau$ on $S$. If $\tau$ fixes a real point in $S=Q_{3,1}$ or $S=Q_{2,2}$, then $\tau$ is linearizable via the projection from this point. If $\tau$ does not fix a real point in $Q_{3,1}$, then $\tau$ is not linearizable, as follows from the classification of Sarkisov links \cite{Iskovskikh1996}. Alternatively, one can show that $\tau$ is not linearizable by noticing that existence of $G$-fixed real smooth points is an equivariant birational invariant when $G\simeq\Z/2$ (compare with \cite{KollarSzabo}). So, it remains to study the involutions on $Q_{2,2}$ which do not fix real points. Recall that every element of $\PGL_2(\R)$ of order 2 is conjugate to one of the following:
	\begin{enumerate}
		\item $\alpha\colon [x:y]\mapsto [y:x]$, fixing the points $[1:1]$ and $[1:-1]$; 
		\item $\alpha'\colon [x:y]\mapsto [y:-x]$, fixing the points $[1:i]$ and $[1:-i]$. 
	\end{enumerate}
	So, a fibrewise automorphism of $S=Q_{2,2}\simeq\p_\R^1\times\p_\R^1$ has a real fixed point, unless it is conjugate to $({\rm id},\alpha')$, $(\alpha,\alpha')$, or $(\alpha',\alpha')$. 
	In the first case, there are no $G$-equivariant Sarkisov links starting from $S$, so the corresponding action is not conjugate to any of the actions described above.
	Moreover, $({\rm id},\alpha')$ and $(\alpha',\alpha')$ are conjugate by the birational involution
	\[
	\left( [s:t],[x:y]\right ) \mapsto \left ([ty+sx:sy-tx],[x:y]\right ),
	\] 
	and $(\alpha,\alpha')$ and $(\alpha',\alpha')$ are conjugate by the birational involution 
	\[
	\left( [s:t],[x:y] \right)\mapsto\left( [sy:tx],[y:x]\right).
	\]
	This finishes the proof.
\end{proof}

\section{Trepalin involutions}\label{section:Trepalin-involutions}

Now, we present three constructions of birational involutions of the real projective plane
that do not fix irrational curves. These constructions are inspired by Example~\ref{example:Trepalin}.
Because of this, we will call the corresponding  involutions (\mbox{\emph{$0$-twisted}}, \mbox{\emph{$1$-twisted}}, \emph{$2$-twisted}) \emph{Trepalin involutions}.

\subsection{First construction ($0$-twisted Trepalin involutions)}
\label{subsection:Trepalin-involution-1}
Let $Y$ be the affine surface in $\mathbb{A}^3_{\R}$ given by
\begin{equation}\label{eq: Trepalin 1}
	x^2+y^2+(t-\epsilon_1)(t-\epsilon_2)\cdots(t-\epsilon_{2r})=0
\end{equation}
where $\epsilon_1<\epsilon_2<\cdots<\epsilon_{2r}$ are real numbers and $r\geqslant 1$ and $\pi_Y\colon Y\to\mathbb{A}^1_{\R}$ be the~map given by $(x,y,t)\mapsto t$. Then there exists a commutative diagram:
$$
\xymatrix{
Y\ar@{->}[d]_{\pi_Y}\ar@{^{(}->}[rr]&&X\ar@{->}[d]^{\pi_{X}}\\
\mathbb{A}^1_\R\ar@{^{(}->}[rr]&&\mathbb{P}^1_\R}
$$
where $X$ is a real smooth projective surface with $\mathrm{Pic}(X)\cong\mathbb{Z}^2$,
both $Y\hookrightarrow X$ and $\mathbb{A}^1_\R\hookrightarrow\mathbb{P}^1_\R$ are open immersions,
and $\pi_X$ is a conic bundle such that $\pi_X^{-1}([1:0])$ is a smooth conic that does not have real points.
By construction, the conic bundle $\pi_X$ has $2r$ singular $\R$-irreducible fibres, which are fibres over the points $[\epsilon_1:1],[\epsilon_2:1],\ldots,[\epsilon_{2r}:1]$.
The real locus $X(\R)$ is a disjoint union of $r\geqslant 1$ spheres (see Proposition~\ref{proposition:conic-bundle-real-loci}), so that the~surface $X$ is $\mathbb{R}$-rational if and only if $r=1$.

Now, we let $U$ be the affine surface in $\mathbb{A}^4_{\R}$ given by
\begin{equation}\label{eq: Trepalin 1 in A4}
\left\{\aligned
&x^2+y^2+(t-\epsilon_1)(t-\epsilon_2)\cdots(t-\epsilon_{2r})=0,\\
&w^2+(t-\lambda_1)(t-\lambda_2)=0,
\endaligned
\right.
\end{equation}
where $\lambda_1<\lambda_2$ are real numbers such that $\lambda_1\ne\epsilon_1<\lambda_2<\epsilon_2$,
and $x$, $y$, $w$, $t$ are coordinates on  $\mathbb{A}^4_{\R}$.
Let $\theta\colon U\to Y$ be the morphism given by $(x,y,w,t)\mapsto(x,y,t)$.
Then $\theta$ is a double cover branched over the fibres $\pi_Y^{-1}(\lambda_1)$ and $\pi_Y^{-1}(\lambda_2)$.
Moreover, we have the following commutative diagram:
\begin{equation}\label{diagr: Trepalin 1}
\xymatrix{
&&U\ar@/^{-1.2pc}/[dll]_{\theta}\ar@{^{(}->}[rr]&&S\ar@{->}[dll]^{\vartheta}\ar@{->}[dd]^{\pi_S}\\
Y\ar@{->}[d]_{\pi_Y}\ar@{^{(}->}[rr]&&X\ar@{->}[d]^{\pi_{X}}&&\\
\mathbb{A}^1_\R\ar@{^{(}->}[rr]&&\mathbb{P}^1_\R&&\mathbb{P}^1_\R\ar@{->}[ll]_{\omega}}
\end{equation}
where $S$ is a real smooth projective surface,
the morphism $\vartheta$ is a double cover branched over the~smooth fibres $\pi_X^{-1}([\lambda_1:1])$ and $\pi_X^{-1}([\lambda_2:1])$,
the morphism $\omega$ is a double cover branched over the points $[\lambda_1:1]$ and $[\lambda_2:1]$,
and $\pi_S$ is a conic bundle with $4r$ singular fibres, so $K_S^2=8-4r\leqslant 4$. Let $\tau$ be the Galois involution of the cover $\vartheta$ and $G=\langle\tau\rangle$.
Then one has \mbox{$\mathrm{rk}\,\Pic(S)^G=\mathrm{rk}\,\mathrm{Pic}(X)=2$}.
Moreover, if $\epsilon_1<\lambda_1$, then $S(\R)\approx\mathbb{S}^1\times\mathbb{S}^1$, and if  $\lambda_1<\epsilon_1$, then $S(\R)\approx\mathbb{S}^2$. 
Indeed, the real locus of $S$ is contained in $U$, then it suffices to look at its equations. If $\epsilon_1<\lambda_1$, then for each $i$ the two singular fibres $\pi_S^{-1}([w:1])$ for $w^2=\epsilon_i$ are not real. If  $\lambda_1<\epsilon_1$, there are exactly two real singular fibres which are $\pi_S^{-1}([w:1])$ for $w^2=\epsilon_1$ and they are irreducible over $\R$.
In both cases, the surface $S$ is $\mathbb{R}$-rational by Theorem \ref{theorem:rational-real}. 
Therefore, the involution $\tau$ induces a~birational involution $\iota\in\Bir(\p_\R^2)$.
The $\tau$-fixed locus  in $S$ consists of the fibres
$\pi_S^{-1}([\lambda_1:1])$ and $\pi_S^{-1}([\lambda_2:1])$,
so $F(\iota)=F(\tau)=\varnothing$.

The involution $\tau$ and the~corresponding birational involution in $\Bir(\mathbb{P}^2_{\mathbb{R}})$
will both be called \emph{$0$-twisted Trepalin involution} with $4r$ singular fibres.
The class of such birational  involutions in $\Bir(\mathbb{P}^2_{\mathbb{R}})$ will be denoted by $\mathfrak{T}_{4r}$.

\subsection{Second construction ($1$-twisted Trepalin involutions)}
\label{subsection:Trepalin-involution-2}
Let us use all assumptions and notations of Section~\ref{subsection:Trepalin-involution-1}.
Now, we let $U^\prime$ be the affine surface in $\mathbb{A}^4_{\R}$ given by
\begin{equation}\label{eq: Trepalin 2}
	\left\{\aligned
	&x^2+y^2+(t-\epsilon_1)(t-\epsilon_2)\cdots(t-\epsilon_{2r})=0,\\
	&w^2+(t-\lambda_1')(t-\lambda_2')=0,
	\endaligned
	\right.
\end{equation}
where $\lambda_1'$, $\lambda_2'$ are real numbers such that either $\epsilon_1\ne \lambda_1'<\epsilon_2$ and $\lambda_2'=\varepsilon_2$, or $\varepsilon_1<\lambda_1'<\varepsilon_2$ and $\lambda_2'=\varepsilon_3$.
Let $\theta^\prime\colon U^\prime\to Y$ be the morphism given by $(x,y,w,t)\mapsto(x,y,t)$.
Then $\theta^\prime$ is a double cover branched over the fibres $\pi_Y^{-1}(\lambda_1)$ and $\pi_Y^{-1}(\lambda_2)$.
Since $\pi_Y^{-1}(\lambda_2)$ is singular, we see that the surface $U^\prime$ has an ordinary double point at the point $(x,y,w,t)=(0,0,0,\lambda_2)$.
We have the following commutative diagram:
\begin{equation}\label{diag: Trepalin 2} 
\xymatrix{
&&&&&\widetilde{S}^\prime\ar@{->}[dl]_{\alpha}\ar@{->}[dr]^{\beta}&&&\\
&&U^\prime\ar@/^{-1.2pc}/[dll]_{\theta^\prime}\ar@{^{(}->}[rr]&&S^\prime\ar@{->}[dll]^{\vartheta^\prime}\ar@{->}[dd]^{\pi_{S^\prime}}&&\widehat{S}^\prime\ar@{->}[dd]^{\pi_{\widehat{S}^\prime}}\\
Y\ar@{->}[d]_{\pi_Y}\ar@{^{(}->}[rr]&&X\ar@{->}[d]^{\pi_{X}}&&&&\\
\mathbb{A}^1_\R\ar@{^{(}->}[rr]&&\mathbb{P}^1_\R&&\mathbb{P}^1_\R\ar@{->}[ll]_{\omega^\prime}\ar@{=}[rr]&&\mathbb{P}^1_\R}
\end{equation}
where $S^\prime$ is a real projective surface that is smooth along $S^\prime\setminus U^\prime$,
the morphism $\vartheta^\prime$ is a double cover branched over the~fibres $\pi_X^{-1}([\lambda_1:1])$ and $\pi_X^{-1}([\lambda_2':1])$,
the morphism $\omega^\prime$ is a double cover branched over the points $[\lambda_1':1]$ and $[\lambda_2':1]$,
the morphism $\alpha$ is the blow up of the singular point of the surface $S^\prime$,
and $\beta$ is the contraction of the proper transform of the fibre $\pi_{S^\prime}^{-1}([\lambda_2':1])$ to a pair of complex conjugate non-real points.
Then $\widehat{S}^\prime$ is a smooth projective surface,  $\pi_{\widehat{S}^\prime}$ is a conic bundle that has $4r-2\geqslant 2$ singular fibres, so that
$K_{\widehat{S}^\prime}^2=10-4r\leqslant 6$.

Let $\tau^\prime$ be the Galois involution of the double cover $\vartheta^\prime$. Set $G^\prime=\langle\tau^\prime\rangle$.
Then \mbox{$\mathrm{rk}\,\mathrm{Pic}(\widehat{S}^\prime)^{G^\prime}=2$}. We observe that $\widehat{S}'(\R)$ is homeomorphic to the real locus of the minimal resolution of $U'$. Therefore, we have the following possibilities:
\begin{itemize}
	\item If $\lambda_2'=\varepsilon_2$, $\varepsilon_1<\lambda_1'<\varepsilon_2$, then $\widehat{S}^\prime(\R)\approx\mathbb{S}^1\times\mathbb{S}^1$. 
	\item If $\lambda_2'=\varepsilon_2$, $\lambda_1'<\varepsilon_1$, then $\widehat{S}^\prime(\R)\approx\mathbb{S}^2$.
	\item If $\lambda_2'=\varepsilon_3$, $\varepsilon_1<\lambda_1'<\varepsilon_2$, then $\widehat{S}^\prime(\R)\approx\mathbb{S}^2$.
\end{itemize}
In all cases, the surface $\widehat{S}^\prime$ is $\mathbb{R}$-rational by Theorem \ref{theorem:rational-real}.
Thus, the involution $\tau^\prime$ induces a~birational involution $\iota^\prime\in\Bir(\p_\R^2)$. 
The~$\tau^\prime$-fixed locus consists of the fibre
$\pi_{\widehat{S}^\prime}^{-1}([\lambda_1:1])$ and a pair of complex conjugate non-real points contained in the~fibre
$\pi_{\widehat{S}^\prime}^{-1}([\varepsilon_2':1])$, which implies that $F(\iota^\prime)=F(\tau^\prime)=\varnothing$.

If $r\geqslant 2$, the involution $\tau^\prime$ and the~corresponding birational involution in $\Bir(\mathbb{P}^2_{\mathbb{R}})$
will be both called \emph{$1$-twisted Trepalin involution} with $4r-2$ singular fibres.
The class of such birational  involutions in $\Bir(\mathbb{P}^2_{\mathbb{R}})$ will be denoted by $\mathfrak{T}_{4r-2}^\prime$.

\subsection{Third construction (2-twisted Trepalin involutions)}
\label{subsection:Trepalin-involution-3}
Let us use all assumptions and notations of Section~\ref{subsection:Trepalin-involution-1}.
Now, we let $U^{\prime\prime}$ be the affine surface in $\mathbb{A}^4_{\R}$ given by
\begin{equation}\label{eq: Trepalin 3}
	\left\{\aligned
	&x^2+y^2+(t-\epsilon_1)(t-\epsilon_2)\cdots(t-\epsilon_{2r})=0,\\
	&w^2+(t-\lambda_1'')(t-\lambda_2'')=0,
	\endaligned
	\right.
\end{equation}
such that either $\lambda_1''=\varepsilon_1$ and $\lambda_2''=\varepsilon_2$, or $\lambda_1''=\varepsilon_1$ and $\lambda_2''=\varepsilon_3$. Let $\theta^{\prime\prime}\colon U^{\prime\prime}\to Y$ be the morphism given by $(x,y,w,t)\mapsto(x,y,t)$.
Then $\theta^{\prime\prime}$ is a double cover branched over the singular fibres $\pi_Y^{-1}(\lambda_1'')$ and $\pi_Y^{-1}(\lambda_2'')$,
so that the surface $U^{\prime\prime}$ has two ordinary double points: $(0,0,0,\lambda_1'')$ and $(0,0,0,\lambda_2'')$.
Now, we have the following commutative diagram:
\begin{equation}\label{diag: Trepalin 3}
\xymatrix{
&&&&&\widetilde{S}^{\prime\prime}\ar@{->}[dl]_{\gamma}\ar@{->}[dr]^{\delta}&&&\\
&&U^{\prime\prime}\ar@/^{-1.2pc}/[dll]_{\theta^{\prime\prime}}\ar@{^{(}->}[rr]&&S^{\prime\prime}\ar@{->}[dll]^{\vartheta^{\prime\prime}}\ar@{->}[dd]^{\pi_{S^{\prime\prime}}}&&\widehat{S}^{\prime\prime}\ar@{->}[dd]^{\pi_{\widehat{S}^{\prime\prime}}}\\
Y\ar@{->}[d]_{\pi_Y}\ar@{^{(}->}[rr]&&X\ar@{->}[d]^{\pi_{X}}&&&&\\
\mathbb{A}^1_\R\ar@{^{(}->}[rr]&&\mathbb{P}^1_\R&&\mathbb{P}^1_\R\ar@{->}[ll]_{\omega^{\prime\prime}}\ar@{=}[rr]&&\mathbb{P}^1_\R}
\end{equation}
where $S^{\prime\prime}$ is a real projective surface that is smooth along $S^{\prime\prime}\setminus U^{\prime\prime}$,
the morphism $\vartheta^{\prime\prime}$ is a double cover branched over the~fibres $\pi_X^{-1}([\lambda_1'':1])$ and $\pi_X^{-1}([\lambda_2'':1])$,
the morphism $\omega^{\prime\prime}$ is a double cover branched over the points $[\lambda_1'':1]$ and $[\lambda_2'':1]$,
the morphism $\gamma$ is the blow up of both singular points of the surface $S^{\prime\prime}$,
and $\delta$ is the contraction of the proper transforms of the fibres $\pi_{S^{\prime\prime}}^{-1}([\lambda_1'':1])$ and $\pi_{S^{\prime\prime}}^{-1}([\lambda_2'':1])$ to two pairs of complex conjugate non-real points.
Then $\widehat{S}^{\prime\prime}$ is a smooth projective surface,
$\pi_{\widehat{S}^{\prime\prime}}$ is a conic bundle that has $4r-4\geqslant 0$ singular fibres, so that
$K_{\widehat{S}^{\prime\prime}}^2=12-4r\leqslant 8$.

Let $\tau^{\prime\prime}$ be the Galois involution of the double cover $\vartheta^{\prime\prime}$. Set $G^\prime=\langle\tau^{\prime\prime}\rangle$.
Then \mbox{$\mathrm{rk}\,\mathrm{Pic}^{G^{\prime\prime}}(\widehat{S}^{\prime\prime})=2$}. We have the following possibilities:
\begin{itemize}
	\item If $\lambda_1''=\varepsilon_1$, $\lambda_2''=\varepsilon_2$, then $\widehat{S}''(\R)\approx\mathbb{S}^1\times\mathbb{S}^1$. 
	\item If $\lambda_1''=\varepsilon_1$, $\lambda_2''=\varepsilon_3$, then $\widehat{S}''(\R)\approx\mathbb{S}^2$.
\end{itemize}
So, the surface $\widehat{S}''$ is $\mathbb{R}$-rational.
Hence, the involution $\tau^{\prime\prime}$ induces a~birational involution $\iota^{\prime\prime}\in\Bir(\p_\R^2)$.
The~$\tau^{\prime\prime}$-fixed locus consists of two  pairs of complex conjugate non-real points contained in the~fibres
$\pi_{\widehat S^{\prime\prime}}^{-1}([\lambda_1'':1])$ and $\pi_{\widehat S^{\prime\prime}}^{-1}([\lambda_2'':1])$, respectively. 
This shows that $F(\iota^{\prime\prime})=F(\tau^{\prime\prime})=\varnothing$.

If $r\geqslant 2$, the involution $\tau^{\prime\prime}$ and the~corresponding birational involution in $\Bir(\mathbb{P}^2_{\mathbb{R}})$
will be both called \emph{$2$-twisted Trepalin involution} with $4r-4$ singular fibres.
The class of such birational  involutions in $\Bir(\mathbb{P}^2_{\mathbb{R}})$ will be denoted by $\mathfrak{T}_{4r-4}^{\prime\prime}$.

\subsection{Classification of Trepalin involutions}\label{subsec: Trepalin classification}

Let $\pi_S: S\to\p_\R^1$ be an $\R$-rational and $G$-minimal conic bundle, i.e. $\Pic(S)^G\simeq\Z^2$, where $G$ is generated by an involution $\tau$ acting non-trivially on the base $\p_\R^1$ (so in particular $F(\tau)=\varnothing$, as the fixed locus of $\tau$ is contained in finitely many fibres of $\pi_S$). Let $\iota$ be the corresponding birational involution of $\p_\R^2$. Assume that $K_S^2\leqslant 4$. The main result of this section is the following

\begin{theorem}\label{thm: Trepalin involutions}
	The involution $\iota$ is conjugate to a $t$-twisted Trepalin's involution with $t\in\{0,1,2\}$.
\end{theorem}

In the remaining part of this section, we give a proof of this theorem. 

\medskip

Since $\tau$ acts non-trivially on $\p_\R^1$, it fixes two points on it, say $p_1$ and $p_2$, which are either both real or complex conjugate. Denote the fibres over $p_1$ and $p_2$ by $F_1$ and $F_2$, respectively. Note that the fibres $F_1$ and $F_2$ are necessarily smooth, see e.g. \cite[Lemma 4.6]{Tre16}. Consider the diagram
	\[
	\xymatrix{
	S\ar@{->}[d]_{\pi_S}\ar[rr]^{\vartheta}&& V=S/\tau\ar@{->}[d]^{\pi_{V}} &&\\
	\p^1_\R\ar[rr]^{\omega}&&\p^1_\R&&
		}
	\]
where $\vartheta$ is a quotient map $S\to V=S/\tau$. Since $K_S^2\leqslant 4$, the conic bundle $\pi_V$ has at least two singular fibres above real points contained in the smooth locus of $V$. 

Moreover, we have $\Pic(V)\simeq\Z^2$, so the irreducible components of these singular fibres are exchanged by the Galois group $\Gal(\C/\R)$. In particular, $\pi_V$ has no real sections. However it has two complex sections $C_1$ and $C_2$ such that $C_1+C_2$ is defined over $\R$. Let $Z_1$ and $Z_2$ denote their preimages on $S$ under the map $\vartheta$. Then $Z_1$ and $Z_2$ are sections of $\pi_S$ such that $Z_1+Z_2$ is defined over $\R$. Changing the coordinates on $\p_\R^1$, if needed, we may assume that $\pi_V^{-1}([1:0])$ is a smooth conic with no real points (this is possible because one of the real singular fibres is a conic with no smooth real points). Furthermore, we may assume that $\omega(p_1)=[\lambda_1:1]$ and $\omega(p_2)=[\lambda_2:1]$, where $\lambda_i$ are (a priori) either both real or complex conjugate, and $\lambda_1<\lambda_2$ in the first case. In fact, we have the following

\begin{lemma}\label{lem: lambda are real}
	$\lambda_1$ and $\lambda_2$ are real. 
\end{lemma}
\begin{proof}
	Suppose that $\lambda_1$ and $\lambda_2$ are complex conjugate. If $V$ is singular, then $\tau$ fixes two pairs of complex conjugate points in $F_1\cup F_2$. So, blowing up one pair and contracting the proper transforms of the fibres $F_1$, $F_2$, we obtain a $\tau$-equivariant commutative diagram
	\[
	\xymatrix{
		S\ar@{->}[d]_{\pi_S}\ar@{-->}[rr]&& S'\ar@{->}[d]^{\pi_{S'}} &&\\
		\p^1_\R\ar@{=}[rr]&&\p^1_\R&&
	}
	\]
	such that $\tau$ fixes two fibres of $\pi_{S'}$ over $[\lambda_1:1]$, $[\lambda_2:1]$ pointwise. Hence, replacing $S$ with $S'$, we may assume that $V$ is smooth.
	
	Note that $K_V^2\leqslant 4$ and by Proposition~\ref{proposition:conic-bundle-real-loci} the real locus $V(\R)$ is a disjoint union of $m\geqslant 1$ spheres. Therefore, if $\lambda_1$ and $\lambda_2$ are complex conjugate, then the preimage of each connected component of $V(\R)$ is either empty or consists of two disjoint spheres (and the induced double cover $S(\R)\to V(\R)$ is trivial). So, $S$ is not $\R$-rational by Theorem \ref{theorem:rational-real}.
\end{proof}

\begin{lemma}\label{lem: point at infinity}
	We have 
	\[
	\pi_V(V(\R))\cup\omega(\p_\R^1(\R))\ne\p_\R^1(\R).
	\]
	In particular, we can make a further change of coordinates on $\p_\R^1$ so that the point at infinity $[1:0]$ is not contained in $\pi_V(V(\R))\cup\omega(\p_\R^1(\R))$.
\end{lemma}
\begin{proof}
	Let us view $\p_\R^1(\R)$ as a circle and $I=\pi_V(V(\R))$ as a union of closed segments $[\varepsilon_{2i-1},\varepsilon_{2i}]$, $i=1,\ldots,r$, on a circle, placed in a clockwise direction. Recall that $[1:0]$ is not contained in this union by our initial assumption, because $\pi_V^{-1}([1:0])$ is a smooth conic with no real points. Note that $\pi_V\circ\vartheta(S(\R))$ is a closed interval in $J=\omega(\p_\R^1(\R))=[\lambda_1,\lambda_2]$. 
	
	Furthermore, if $V$ is smooth then 
	\begin{equation}\label{eq: intervals}
		\pi_V\circ\vartheta(S(\R))=\omega\circ\pi_S(S(\R))=\pi_V(V(\R))\cap \omega(\p_\R^1(\R))=I\cap J.
	\end{equation}	
	So, $I$ and $J$ cannot cover the whole circle. Otherwise, the endpoints of $[\lambda_1,\lambda_2]$ belong to the interior of some $[\varepsilon_{2i-1},\varepsilon_{2i}]$. However, in that case $S(\R)$ would be disconnected or empty.
	
	Similarly, if $V$ is singular then $I\cap J$ is a union of the closed interval $\pi_V\circ\vartheta(S(\R))$ and at most 2 points, the endpoints of $J$. Thus $I\cup J$ is the whole circle $\p_\R^1(\R)$ only if $r=1$ and the intervals $[\varepsilon_1,\varepsilon_2]$ and $J$  share one common endpoint. However, in this case $K_S^2=6$ (since there are no non-real singular fibres), which is excluded by our assumption. 
\end{proof}

Note that the sections $Z_1$ and $Z_2$ of $X$ are both $G$-invariant and exchanged by the Galois group $\Gal(\C/\R)$. Arguing as in the first paragraph of the proof of \cite[Lemma 15]{BlancMangolte} , we may construct a fibrewise $G$-birational transformation $S\dashrightarrow\overline{S}$ that fits the following commutative diagram
\[
\xymatrix{
	S\ar@{->}[d]_{\pi_S}\ar@{-->}[rr]^{\chi}&& \overline{S}\ar@{->}[d]^{\pi_{\overline{S}}} &&\\
	\p^1_\R\ar@{=}[rr]&&\p^1_\R&&
}
\]
such that $\overline{S}$ is a smooth surface with $\Pic(\overline{S})^G\simeq\Z^2$, $\pi_{\overline{S}}$ is a $G$-conic bundle, and $\chi(Z_1)$ and $\chi(Z_2)$ are two disjoint sections of this conic bundle. Thus, replacing $S$ with $\overline{S}$, we may assume additionally that $Z_1$ and $Z_2$ are disjoint. 

Note that the set $\{Z_i\cap F_j\colon i,j=1,2\}$ consists of four $G$-fixed points. Without loss of generality, we may assume that the fixed locus of $G$ is one of the following:

\begin{enumerate}
	\item[(0)] The fibres $F_1$ and $F_2$ are fixed pointwise by $G$. 
	\item[(1)] The fibre $F_1$ is pointwise fixed by $G$, and $Z_1\cap F_2$, $Z_2\cap F_2$ are the only $G$-fixed points on $F_2$.
	\item[(2)] The points $\{Z_i\cap F_j\colon i,j=1,2\}$ are the only points fixed by $G$. 
\end{enumerate}

We will show that these 3 cases correspond to $0$-twisted, $1$-twisted and $2$-twisted Trepalin involutions, respectively.

\begin{lemma}\label{lem: Trepalin 1}
	Assume that $F_1$ and $F_2$ are fixed pointwise by $G$. Then $\tau$ is conjugate to a $0$-twisted Trepalin involution.
\end{lemma}
\begin{proof}
	First note that $V$ is a smooth surface, equipped with the structure of $\R$-minimal conic bundle with $2r$ singular fibres, so $K_V^2=8-2r$ and $K_S^2=8-4r$. Denote by $F$ the general fibre of $\pi_V$. Then one has $C_1+C_2\sim -K_V+rF$, $C_1^2=C_2^2=-r$ and $-K_V\cdot C_i=2-r$. 
	
	One can show that the linear system $|-K_V+(r-2)F|$ is base-point free and gives a birational morphism $\rho: V\to Z$ contracting $C_1$ and $C_2$ such that $Z$ is a hypersurface in $\p_\R(1,1,r,r)$ given by the equation
	\[
	x^2+y^2+f_{2r}(t,s)=0,
	\] 
	where $t,s,x,y$ are coordinates on $\p_\R(1,1,r,r)$ of weights $1,1,r,r$, respectively, $f_{2r}$ is a real polynomial of degree $2r$, and $\pi_V\circ\rho^{-1}\colon Z\dashrightarrow\p_\R^1$ is given by $[t:s:x:y]\mapsto [t:s]$. So, we identify $[t:s]$ with coordinates on the base of $\pi_V$. Since $V$ is smooth and $\R$-minimal over $\p_\R^1$, the polynomial $f_{2r}(t,s)$ has $2r$ distinct real roots in $\p_\R^1$. By our assumption, $f_{2r}([1:0])>0$. 
	
	Let $[\varepsilon_i:1]$, $\varepsilon_i<\varepsilon_{i+1}$, $i=1,\ldots, 2r$, denote the roots of $f_{2r}(t,s)$, i.e.
	\[
	f_{2r}(t,s)=(t-\varepsilon_1s)\cdot\ldots\cdot (t-\varepsilon_{2r}s).
	\]
	Set $Y=V\setminus\big(\pi_V^{-1}([1:0])\cup C_1\cup C_2 \big )$. Then $Y$ is given by the equation (\ref{eq: Trepalin 1}) in $\AAA_\R^3$. We now regard $t$ as the affine coordinate on $\AAA_\R^1$ and denote by $\pi_Y: Y\to\AAA_\R^1$ the corresponding conic bundle. Let $U=\vartheta^{-1}(Y)\subset S$ and let $\theta\colon U\to Y$ be the induced double cover. Then $U$ can be given by the equation
	\begin{equation}
		\left\{\aligned
		&x^2+y^2+(t-\epsilon_1)(t-\epsilon_2)\cdots(t-\epsilon_{2r})=0,\\
		&w^2\pm(t-\lambda_1)(t-\lambda_2)=0,
		\endaligned
		\right.
	\end{equation} 
	in $\AAA_\R^4$ and $\theta$ is given by $(x,y,t,w)\mapsto (x,y,t)$. Observe that $Y(\R)=V(\R)$, $S(\R)=U(\R)$, 
	\[
	\pi_Y(Y(\R))=\bigsqcup_{i=1}^{r}[\varepsilon_{2i-1},\varepsilon_{2i}].
	\]
	By Lemma \ref{lem: lambda are real}, one has $\lambda_i\in\R$. By Lemma \ref{lem: point at infinity}, one has $[1:0]\notin\omega(\p_\R^1(\R))$, so $U$ is in fact given by (\ref{eq: Trepalin 1 in A4}). By (\ref{eq: intervals}), one has
	\[
	\pi_Y\circ\vartheta(U(\R))=[\lambda_1,\lambda_2]\cap \left (\bigsqcup_{i=1}^{r}[\varepsilon_{2i-1},\varepsilon_{2i}] \right ).
	\]
	Since $U(\R)$ is connected, $[\lambda_1,\lambda_2]$ can intersect only one of the $[\varepsilon_{2i-1},\varepsilon_{2i}]$. Making a further change of coordinates on $\p_\R^1$, if needed, we may assume that this is $[\varepsilon_1,\varepsilon_2]$. Moreover, $\lambda_1,\lambda_2\notin\{\varepsilon_1,\varepsilon_2\}$, as the fibres $F_1$ and $F_2$ are smooth. If $[\varepsilon_1,\varepsilon_2]$ is strictly contained in $[\lambda_1,\lambda_2]$ then $S(\R)=U(\R)\approx\Sph^2\sqcup\Sph^2$, a contradiction. Hence up to an affine change of coordinate $t$, we may assume that $\lambda_2<\varepsilon_2$. Observe that we obtained the commutative diagram (\ref{diagr: Trepalin 1}) with $X=V$. This shows that $\tau$ is a $0$-twisted Trepalin involution constructed in Section \ref{subsection:Trepalin-involution-1}.
	
\end{proof}

\begin{lemma}\label{lem: Trepalin 2}
	Assume that $F_1$ is pointwise fixed by $G$, and $Z_1\cap F_2$, $Z_2\cap F_2$ are the only $G$-fixed points on $F_2$. Then $\tau$ is conjugate to a $1$-twisted Trepalin involution.
\end{lemma}
\begin{proof}
	Recall that $\lambda_1,\lambda_2\in\R$. The surface $V$ has two singular points of type $A_1$ which are exchanged by the Galois involution. We have the following commutative diagram
	\[
	\xymatrix{
		&& & \widetilde{S}\ar[d]^{\widetilde{\vartheta}}\ar[ddlll]_{\beta}\ar[ddrrr]^{\alpha}\\
		&& & \widehat{V}\ar[dl]\ar[dr] &&\\
		S\ar@{->}[d]_{\pi_S}\ar[rr]^{\vartheta}&& V\ar@{->}[d]^{\pi_{V}} && V'\ar[d]^{\pi_{V'}} && S'\ar[d]^{\pi_{S'}}\ar[ll]_{\vartheta'}\\
		\p^1_\R\ar[rr]^{\omega}&&\p^1_\R\ar@{=}[rr]&& \p_\R^1&& \p_\R^1\ar@{->}[ll]_{\omega'}
	}
	\]
	where $\beta$ is a blow-up of two points $Z_1\cap F_1$ and $Z_1\cap F_2$, $\alpha$ is a contraction of the proper transform of $F_2$ to a singular point of $S'$ of type $A_1$, $\widehat{V}\to V$ is a minimal resolution of the singularities of $V$ (i.e. the blow up of $\vartheta(Z_1\cap F_2)$ and $\vartheta(Z_2\cap F_2)$), $\widehat{V}\to V'$ is a contraction of proper transform of $\vartheta(F_2)$ to a smooth point of $V'$, $\widetilde{\vartheta}$ is a double cover ramified in proper transform of $F_1$ and $\beta$-exceptional curves, and $\vartheta'$ is a double cover ramified in the smooth fibre $\pi_{V'}^{-1}([\lambda_1:1])$ and the singular fibre $\pi_{V'}^{-1}([\lambda_2:1])$; finally, $\omega'$ is a double cover branched at $[\lambda_1:1]$ and $[\lambda_2: 1]$. Observe that $V'$ is a smooth real surface, $\Pic(V')\simeq\Z^2$ and $\pi_{V'}$ is a conic bundle. Recall that the fibre at infinity $\pi_{V'}^{-1}([1:0])\simeq \pi_{V}^{-1}([1:0])$ is a smooth real conic with no real points.
	
	Denote by $C_1'$ and $C_2'$ the proper transforms of the (complex) curves $C_1$, $C_2$ on $V'$. Then $C_1'$ and $C_2'$ are two disjoint sections of the conic bundle $\pi_{V'}$ such that $C_1'+C_2'$ is defined over $\R$. Let $Y=V'\setminus (C_1'\cup C_2'\cup \pi_{V'}^{-1}([1:0]))$, and let $\pi_Y: Y\to\AAA_\R^1$ be the morphism induced by $\pi_{V'}$. Arguing as in the proof of Lemma \ref{lem: Trepalin 1}, we see that $Y$ can be given in $\AAA_\R^3$ by the equation (\ref{eq: Trepalin 2}) for some real numbers $\epsilon_1<\epsilon_2<\cdots<\epsilon_{2r}$, and $\pi_Y\colon Y\to\mathbb{A}^1_{\R}$ is the~map given by $(x,y,t)\mapsto t$. Note that $\pi_{V'}$ has $2r$ singular fibres over the points $[\varepsilon_i:1]$, and $\pi_S$ has $2(2r-1)=4r-2$ singular fibres. As we assume $K_S^2\leqslant 4$, we get $r\geqslant 2$.
	
	By construction, $\lambda_2\in\{\varepsilon_1,\ldots,\varepsilon_{2r}\}$. Set $U'={\vartheta'}^{-1}(Y)\subset S'$ and let $\theta'\colon U'\to Y$ be the induced double cover. Then $\theta'$ is branched over the smooth fibre $\pi_Y^{-1}(\lambda_1)$ and the singular fibre $\pi_Y^{-1}(\lambda_2)$. Since $[1:0]\notin\omega'(\p_\R^1(\R))=\omega(\p_\R^1(\R))$ by Lemma \ref{lem: point at infinity}, $U'$ may be given by the equation (\ref{eq: Trepalin 2}) with $\lambda_1'=\lambda_1$ and $\lambda_2'=\lambda_2$. Now, arguing as in the end of the proof of Lemma \ref{lem: Trepalin 1}, we can change the coordinates on $\p_\R^1$ so that $\lambda_2\in\{\varepsilon_2,\varepsilon_3\}$. Since $S(\R)$ is non-empty and connected, we see that	either $\epsilon_1\ne \lambda_1<\epsilon_2$ and $\lambda_2=\varepsilon_2$, or $\varepsilon_1<\lambda_1<\varepsilon_2$ and $\lambda_2=\varepsilon_3$. More precisely, we have the following possibilities: 
	\begin{itemize}
		\item $\lambda_2=\varepsilon_2$, $\varepsilon_1<\lambda_1<\varepsilon_2$, and $F_1(\R)\ne\varnothing$, $F_2(\R)\ne\varnothing$. 
		\item $\lambda_2=\varepsilon_2$, $\lambda_1<\varepsilon_1$ and $F_1(\R)=\varnothing$, $F_2(\R)\ne\varnothing$. 		
		\item $\lambda_2=\varepsilon_3$, $\varepsilon_1<\lambda_1<\varepsilon_2$ and $F_1(\R)\ne\varnothing$, $F_2(\R)=\varnothing$.
	\end{itemize}
	
	We obtained the commutative diagram (\ref{diag: Trepalin 2}) with $X=V'$, $\widehat{S}'=S$ and $\widetilde{S}'=\widetilde{S}$. This shows that $\tau$ is a $1$-twisted Trepalin involution constructed in Section~ \ref{subsection:Trepalin-involution-2}.
\end{proof}

\begin{lemma}\label{lem: Trepalin 3}
	Suppose that points $\{Z_i\cap F_j\colon i,j=1,2\}$ are the only points fixed by $G$. Then $\tau$ is conjugate to a $2$-twisted Trepalin involution.
\end{lemma}
\begin{proof}
The proof is analogous to the proof of Lemma~\ref{lem: Trepalin 2}, only we now blow up the four $A_1$ singularities of $V$. The same way, we obtain a $G$-birational map from $S$ to a smooth surface $U''\subset\AAA^4_\R$ given by equation (\ref{eq: Trepalin 3}) with $\lambda_1''=\lambda_1$ and $\lambda_2''=\lambda_2$. Arguing as in the end of the proof of Lemma \ref{lem: Trepalin 2}, we can change the coordinates on $\p_\R^1$ so that $\lambda_1=\varepsilon_1$ and  $\lambda_2\in\{\varepsilon_2,\varepsilon_3\}$. If $(\lambda_1,\lambda_2)=(\varepsilon_1,\varepsilon_2)$ then $F_1(\R)\ne\varnothing$, $F_2(\R)\ne\varnothing$. On the other hand, if $(\lambda_1,\lambda_2)=(\varepsilon_1,\varepsilon_3)$ then $F_1(\R)\ne\varnothing$, $F_2(\R)=\varnothing$. 
Observe that we obtained the commutative diagram (\ref{diag: Trepalin 3}) with $X=V''$, $\widehat{S}''=S$ and $\widetilde{S}''=\widetilde{S}$. This shows that $\tau$ is a $2$-twisted Trepalin involution constructed in Section~ \ref{subsection:Trepalin-involution-3}.
\end{proof}
 
\section{Bertini involutions}
\label{section:Bertini}

Let $S$ be a~real del Pezzo surface of degree $1$ that is rational over $\R$. In our treatment we follow \cite[§6.6]{Kollar97}. Recall that $S$ is a~hypersurface in $\p_\R(1,1,2,3)$ of degree $6$,
and the~natural projection to $S\to\p_\R(1,1,2)$ gives a~double cover
$\pi\colon S\to Q$, where $Q$ is the geometrically irreducible quadric cone in $\p^3_\R$.
The Galois involution of this double cover and the~corresponding birational involution in $\Bir(\mathbb{P}^2_{\mathbb{R}})$ will be both called the \emph{Bertini involution}.
The class of such birational  involutions in $\Bir(\mathbb{P}^2_{\mathbb{R}})$ will be denoted by $\mathfrak{B}_{4}$.

Now let us briefly discuss how the Bertini involutions are related to the topology of $S(\R)$. This will not be used in the classification, but shows that the fixed curve defines the involution uniquely in some cases. 
The real locus $S(\R)$ is not empty and smooth, thus the~regular part of $Q(\mathbb{R})$ is not empty and we can choose suitable affine coordinates $(x, y, z)$ on $\mathbb{A}^3$ to write an equation of $Q\cap \mathbb{A}^3_\R$ in the~form
$
x^2+y^2=1,
$
so that $Q(\mathbb{R})$ is a~cylinder with a~singular point at infinity.
The surface $S$ is the~double cover of the~quadric cone branched at the~vertex of the~cone and at a~smooth curve $C_6$ of genus $4$
which is the~intersection of a~cubic surface with the~cone. The equations of $S\cap \mathbb{A}^4$ are of the~form
$$
x^2+y^2-1=u^2+f_3(x,y,z)=0
$$
where $f_3$ is a~cubic inhomogeneous polynomial.

From Theorem~\ref{theorem:rational-real}, the~real locus $S(\mathbb{R})$ is not empty and connected.
Using \cite[§4]{Wa87} or \cite[§6.6]{Kollar97}, we get that $C_6(\R)$ is the~union of one simple closed loop homotopic
to a~plane section of the~cone called a~\emph{big circle}, which we denote by $\Omega_0$, and $t\in\{0,\dots, 4\}$ non-nested ovals (simple closed loops which are null homotopic)
on the~same side of the~big circle.

We may assume that $f_3(x,y,z)$ is positive inside the~ovals in $Q(\mathbb{R})$, if there are any. Let $S_{+}$ be the double cover of $Q$ such that the equation of $S_{+}\cap \mathbb{A}^4$ is of the~form
$
x^2+y^2-1=u^2+ f_3(x,y,z)=0
$. Similarly, let $S_{-}$ be the double cover of $Q$ such that the equation of $S_{-}\cap \mathbb{A}^4$ is of the~form
$
x^2+y^2-1=u^2-f_3(x,y,z)=0
$. 
If there are no ovals, the~case is symmetrical and we can denote by $S_+$ one of the~double cover and by $S_-$ the~other one.
The surface $S$ is either $S_+$ or $S_-$.
The topology of the~loci $S_\pm(\R)$  is described in Table~\ref{table.dp1}.

\begin{table}[ht]
\renewcommand{\arraystretch}{1.5}
\begin{tabular}{|c||c|c|c|c|c|}
\hline
$C_6(\mathbb{R})$ & \quad $\Omega_0$  \quad\quad& $\Omega_0\sqcup 1$ oval & $\Omega_0\sqcup 2$ ovals  & $\Omega_0\sqcup 3$ ovals  & $\Omega_0\sqcup 4$ ovals  \\
 \hline
$S_+(\R)$  & $\mathbb{RP}^2$ & $\mathbb{RP}^2 \sqcup \Sph^2$ & $\mathbb{RP}^2 \sqcup_{i=1}^{2}\Sph^2$ & $\mathbb{RP}^2 \sqcup_{i=1}^{3}\Sph^2$ & $\mathbb{RP}^2 \sqcup_{i=1}^{4}\Sph^2$ \\
\hline
$S_-(\R)$  & $\mathbb{RP}^2$ & $\#^3\mathbb{RP}^2$ & $\#^5\mathbb{RP}^2$ & $\#^7\mathbb{RP}^2$ & $\#^9\mathbb{RP}^2$ \\
  \hline
\end{tabular}
\bigskip
\caption{Real del Pezzo surfaces of degree $1$. The real locus $C_6(\mathbb{R})$ is formed by one big circle $\Omega_0$ and $t\in\{0,\dots, 4\}$ ovals.}
\label{table.dp1}
\end{table}

Here, we denote by $\#^k\mathbb{RP}^2$ the~connected sum of $k$ copies of $\mathbb{RP}^2$, for example $\#^2\mathbb{RP}^2$ is the~Klein bottle.
Note that there are four other possible real loci for $S_\pm(\R)$, which all are non connected, see Table \ref{table.dp1}.

\begin{lemma}
\label{lemma:Bertini-curve}
If $C_6(\mathbb{R})$ is non-connected, then $S$ is uniquely determined by $C_6$.
\end{lemma}

\begin{proof}
If $C_6(\mathbb{R})$ has at least one oval, then it follows from Table~\ref{table.dp1} that only one of the covers $S_+$ or $S_-$ is rational over $\R$, so $S$ is uniquely determined by the real curve $C_6$.
\end{proof}

If the real locus $C_6(\mathbb{R})$ is connected, then it is formed by one big circle in $Q$.
In this case, both surfaces $S_+$ and $S_-$ are rational over $\mathbb{R}$; their real loci are diffeomorphic to $\mathbb{RP}^2$. These two real surfaces may be isomorphic in some cases. Let us give two examples.

\begin{example}
\label{example:Geiser-no-oval-isomorphic}
Let $S_{+}$ and $S_{-}$ be smooth hypersurfaces in $\p_\R(1,1,2,3)$ that are given by
$$
w^2=\pm\big(z^3+xy^5+yx^5\big).
$$
Then  topologically both $S_{+}(\R)$ and $S_{-}(\R)$ are $\mathbb{RP}^2$.
Moreover, one has $S_+\cong S_-$, and this isomorphism is given by $[x:y:z:w]\mapsto[-x:y:-z:w]$.
\end{example}

\begin{example}
\label{example:Geiser-no-oval-non-isomorphic}
Let $S_{+}$ and $S_{-}$ be smooth hypersurfaces in $\p_\R(1,1,2,3)$ that are given by
$$
w^2=\pm\big(z^3+f_4(x,y)z+f_6(x,y)\big),
$$
where $f_4$ and $f_6$ are homogeneous polynomials of degree $4$ and $6$, respectively.
Suppose that $f_4(x,y)>0$ for every $(x,y)\in\R^2\setminus(0,0)$.
Then topologically both $S_{+}(\R)$ and $S_{-}(\R)$ are $\mathbb{RP}^2$.
On the~other hand, for sufficiently general choice of the~polynomials $f_4$ and $f_6$,
the surfaces $S_+$ and $S_-$ are not isomorphic over $\R$.
\end{example}

Now, let $\tau$ be an involution in $\mathrm{Aut}(S)$, let $G=\langle\tau\rangle$,
and let~$C$ be the~union of all curves in $S$ that are pointwise fixed by $\tau$.
Then the curve $C$ is smooth by Lemma~\ref{lemma:fixed-curves-smooth} and either $F(\tau)=\varnothing$
or $F(\tau)$ is the unique irrational component of the curve $C$ by Lemma~\ref{lemma:fixed-curve-unique}.
Moreover, if $\tau$ is the~Bertini involution then
$
F(\tau)=C=C_6
$
and $\Pic(S)^G\cong\mathbb{Z}$. Vice versa, we have the~following result:

\begin{proposition}[{\cite[Proposition 9.2]{Yasinsky}}]
\label{proposition:Bertini}
If $\Pic(S)^G\simeq\mathbb{Z}$, then $\tau$ is the Bertini involution. 
\end{proposition}

\begin{corollary}
\label{corollary:Bertini-curve}
For a birational map $\rho\colon S\dasharrow\mathbb{P}^2_{\mathbb{R}}$, let $\iota=\rho\circ\tau\circ\rho^{-1}\in\Bir(\mathbb{P}^2_{\mathbb{R}})$.
If $\Pic(S)^G\simeq\mathbb{Z}$, then  $S$ is uniquely determined by the~conjugacy class of the birational involution $\iota$.
\end{corollary}
\begin{proof}
By Theorem~\ref{theorem:Manin-Segre}, if $\Pic(S)^G\simeq\mathbb{Z}$, then $S$ is $G$-birationally super-rigid. Applying Proposition~\ref{proposition:Bertini} to $S$ and $\tau$ yields the result.
\end{proof}

\section{Geiser and Kowalevskaya involutions}
\label{section:Geiser}

Let $S$ be a~real smooth projective del Pezzo surface of degree $K_S^2=2$ that is rational over $\R$,
and let~$\pi\colon S\to\p_\R^2$ be the~anticanonical double cover.
Then the Galois involution of this double cover and  the~corresponding birational involution in $\Bir(\mathbb{P}^2_{\mathbb{R}})$ will be both called the \emph{Geiser involution}.
The class of such birational involutions in $\Bir(\mathbb{P}^2_{\mathbb{R}})$ will be denoted by $\mathfrak{G}_{3}$.

Let $C_4$ be the quartic curve in $\p_\R^2$ that is the ramification curve of the double cover $\pi$.
Then $C_4$ is a~real curve of genus $3$ whose real locus $C_4(\mathbb{R})$ is a~collection of ovals in $\p_\R^2(\R)$.
A priori, we may have $C_4(\mathbb{R})=\varnothing$, but we will see later that this is impossible since $S$ is rational (over $\mathbb{R})$.

We may choose projective coordinates on $\p_\R^2$ such that all ovals of $C_4$ are contained in the~chart $z\ne 0$.
Indeed, it follows from \cite{Zeuthen} or \cite[Exercise 6.5.]{Kollar97} that $\p_\R^2$ contains a~real bi-tangent of the quartic $C_4$. See also \cite{MangolteRaffalli} for a more general result about multitangents. Thus, perturbing this bi-tangent a~little bit, we obtain a~real line in $\p_\R^2$ that does not intersect ovals of $C_4$.
Now, we can choose projective coordinates such that this line is given by $z=0$,
so that all ovals of $C_4$ are contained in the~chart $z\ne 0$,
which we identify with $\mathbb{R}^2$ with affine coordinates $(x,y)=[x:y:1]$.
Thus, in the following (except for the proof of Proposition~\ref{prop:Kowalevskays-involution}) we will assume that
all ovals of $C_4$ are contained in the~chart $z\ne 0$,
i.e. the intersection $\{z=0\}\cap C_4(\mathbb{R})$ is empty.

Choose quartic polynomial $f_4(x,y,z)$
such that $C_4$ is given by $f_4(x,y,z)=0$ and $f_4(x,y,1)$ is negative on the~unbounded part of the plane $\mathbb{R}^2$. Let $S_+$ and $S_-$ be two real quartic surfaces in $\p_\R(1,1,1,2)$ that are given by
$$
\pm w^2=f_4(x,y,z),
$$
respectively.
Here, we also consider $x$, $y$ and $z$ as coordinates on $\p_\R(1,1,1,2)$ of weight $1$, and $w$ is a~coordinate of weight $2$.
Note that $\p_\R(1,1,1,2)(\mathbb{R})$ is disconnected, and both surfaces $S_+$ and $S_-$ are isomorphic over $\C$.
Furthermore, our surface $S$ is either $S_+$ or $S_-$.
The topology of $S_+(\mathbb{R})$ and $S_-(\mathbb{R})$ is described in Table~\ref{table.dp2},  see \cite[Proposition 6.2]{Kollar97}.

\begin{table}[ht]
\renewcommand{\arraystretch}{1.5}\hspace*{-1cm}
\begin{tabular}{|c||c|c|c|c|c|c|}
\hline
 $C_4(\mathbb{R})$  & $\varnothing$ & \quad$1$ oval\quad\quad & $2$ non-nested ovals & $2$ nested ovals & \quad $3$ ovals \quad\quad & \quad $4$ ovals \quad\quad\\
\hline
\hline
$S_+(\mathbb{R})$  & $\varnothing$ & $\Sph^2$& $\Sph^2\sqcup \Sph^2$& $\Sph^1\times \Sph^1$& $\sqcup_{i=1}^{3}\Sph^2$& $\sqcup_{i=1}^{4}\Sph^2$\\
\hline
$S_-(\mathbb{R})$  & $\mathbb{RP}^2\sqcup\mathbb{RP}^2$ &{$\#^2\mathbb{RP}^2$}& {$\#^4\mathbb{RP}^2$} & {$\Sph^2\sqcup \#^2\mathbb{RP}^2$}& {$\#^6\mathbb{RP}^2$}& {$\#^8\mathbb{RP}^2$}\\
\hline
\end{tabular}
\bigskip
\caption{Real del Pezzo surfaces of degree $2$.}
\label{table.dp2}
\end{table}

Since $S$ is assumed to be $\R$-rational, the~real locus $S(\mathbb{R})$ is not empty and connected. In~particular, we see from Table~\ref{table.dp2} that the real locus $C_4(\mathbb{R})$ is not empty.

\begin{remark}
\label{corollary:Geiser-curve}
Table~\ref{table.dp2} implies the following:
if $C_4(\mathbb{R})$ is non-connected, then $S$ is uniquely determined by $C_4$,
and if $C_4(\mathbb{R})$ is connected, then $S$ is uniquely determined by $C_4$ and the orientability of the~real locus $S(\R)$.
\end{remark}

Now let $\tau$ be an involution in $\mathrm{Aut}(S)$, let $G=\langle\tau\rangle$,
and let~$C$ be the~union of all curves in $S$ that are pointwise fixed by $\tau$.
By Lemma~\ref{lemma:fixed-curves-smooth} and Lemma~\ref{lemma:fixed-curve-unique}, the~curve $C$ is smooth (and maybe empty) and
either $F(\tau)=\varnothing$, or $F(\tau)$ is the unique irrational component of the curve $C$.
Moreover, if our involution $\tau$ is the~Geiser involution,
then $\Pic(S)^G\cong\mathbb{Z}$ and $F(\tau)=C=C_4$.
However, contrary to the case of complex del Pezzo surface of degree 2, if $\Pic(S)^G\simeq\mathbb{Z}$,
we cannot conclude that $\tau$ is the~Geiser involution. This phenomenon was first observed in \cite[Example 8.2]{Yasinsky}. 
The reason for this is the following remarkable (not well known) fact that was essentially discovered by Sophie Kowalevskaya in \cite{Kowalevskaya}.

\begin{proposition}[{\cite[\S9]{DolgachevDuncan}}]
\label{proposition:Kowalevskays}
Let $B$ be a~smooth quartic curve in $\p^2$. Then four bitangents of the curve $B$ meet at one point $O$
if and only if there exists a~biregular involution $\kappa\in\mathrm{Aut}(\p^2)$ which leaves $B$ invariant
and has the point $O\notin B$ as an isolated fixed point.
\end{proposition}

\begin{remark}
	Sophie Kowalevskayas's original motivation for the study of plane quartics was a~problem of reduction	of abelian integrals to elliptic integrals. Namely, in \cite{Kowalevskaya} she proves the following remarkable statement. Let $y=f(x)$ be an algebraic function which satisfies a~quartic equation $F(x,y)=0$. Let $B$ denote the quartic curve given by $F(x,y)=0$. Then there exists an abelian integral of the first kind
	\[
	\int \Phi(x,f(x)) dx
	\]
	which can be reduced to an elliptic integral using a~change of variables of degree 2 if and only if $B$ has four bitangents meeting
	at one point.
\end{remark}

Using Proposition~\ref{proposition:Kowalevskays} and its proof, we get the following refinement of \cite[Proposition~8.1]{Yasinsky}.

\begin{proposition}
\label{prop:Kowalevskays-involution}
Suppose $\Pic(S)^G\simeq\mathbb{Z}$, the involution $\tau$ is not the involution of the~double cover $\pi\colon S\to\p^2_\R$, i.e. $\tau$ is not the Geiser involution, and $F(\tau)\ne\varnothing$.
Then one can choose coordinates on $\p_\R^2$ and $\p_\R(1,1,1,2)$ such that $\tau$ is given by
$[x:y:z:w]\mapsto[x:-y:z:w]$, one has $S=S_{+}$ and $f_4=-(y^2+ax^2+bxz+cz^2)^2\pm xz(x-z)(x-sz)$ for some real numbers $a,b,c,s$ such that $a>0$, $s>1$, the~curve $C_4$ is smooth,
and either $C_4(\mathbb{R})$ consists of one oval, or the~locus $C_4(\mathbb{R})$ consists of two nested ovals in $\p_\R^2(\R)$.
In particular, the~surface $S$ is given in $\p_\R(1,1,1,2)$ by
$$
w^2=-(y^2+ax^2+bxz+cz^2)^2\pm xz(x-z)(x-sz).
$$
Moreover, if the real locus $C_4(\mathbb{R})$ consists of two nested ovals, then $S(\R)\approx \Sph^1\times \Sph^1$.
Similarly, if the locus $C_4(\mathbb{R})$ consists of one oval, then $S(\R)\approx \Sph^2$.
One has $F(\tau)=C$, where $C$ is the smooth genus 1 curve
that is cut out on $S$ by $y=0$.
\end{proposition}

\begin{proof}
Since $\tau$ is not a~Geiser involution and the~double cover $\pi$ is $G$-equivariant,
the~involution $\tau$ induces a~non-trivial involution $\kappa\in\mathrm{Aut}(\p_\R^2)$ which leaves the curve $C_4$ invariant.
Hence, choosing appropriate coordinates on $\p_\R^2$, we may assume that $\kappa$ acts on $\p_\R^2$ by $[x:y:z]\mapsto [x:-y:z]$.
Note that we abuse our previous choice of coordinates on $\p_\R^2$ that guarantees that the~equation $f_4(x,y,0)=0$ has no real solutions with $(x,y)\ne(0,0)$.
However, we will see later that this condition can also be preserved under the new choice of coordinates on $\p_\R^2$.

Since $C_4$ is $G$-invariant, we may assume that
$f_4(x,y,z)=\pm(y^2+g_2(x,z))^2+g_4(x,z)$
for some real homogeneous polynomials $g_2(x,z)$ and $g_4(x,z)$ of degree $2$ and $4$, respectively.
Since $F(\tau)\ne\varnothing$, we conclude that $\tau$ acts on $S$ as follows: $[x:y:z:w]\mapsto[x:-y:z:w]$.
Then $F(\tau)=C$, where $C$ is the smooth genus 1 curve cut out on $S$ by $y=0$.

Suppose that $C_4(\mathbb{R})$ contains two non-nested ovals. 
By \cite[Exercise 6.5]{Kollar97}, $C_4$ has at least 8 real bitangents. Then there exists a~real line $\ell\subset\p_\R^2$ that is tangent to both of them,
so that $\ell$ is a~real bi-tangent of the curve $C_4$ that is contained in the locus where $f_4(x,y,z)$ is non-positive.
Using Table~\ref{table.dp2}, we see that $S=S_-$, since $S$ is rational by assumption.
Hence, we have $\pi^{*}(\ell)=\ell_1+\ell_2$, where $\ell_1$ and $\ell_2$ are two distinct $(-1)$-curves in $S$ that are both defined over $\mathbb{R}$.
Moreover, we also have $\tau(\ell_1)\ne\ell_2$, because $\tau$ acts on $S$ as $[x:y:z:w]\mapsto[x:-y:z:w]$.
Thus, either $\tau(\ell_1)=\ell_1$ or $\tau(\ell_1)\cdot\ell_1=1$ or $\tau(\ell_1)\cap\ell_1=\varnothing$,
which implies that $(\ell_1+\tau(\ell_1))^2\leqslant 0$.
This contradicts the~condition $\Pic(S)^G\simeq\mathbb{Z}$, because the divisor $\ell_1+\tau(\ell_1)$ is $G$-invariant.

Hence, using Table~\ref{table.dp2}, we conclude that either $C_4(\mathbb{R})$ consists of one oval, or $C_4(\mathbb{R})$ consists of two nested ovals.

We claim that $g_4(x,z)$ has four real roots $[x:z]\in\p_\R^1$.
Indeed, suppose that $g_4(x,z)$ has a~root $[1:\xi]\in\p_\R^1$ such that $\xi\not\in\mathbb{R}$. Thus, if $S$ is given by $w^2=(y^2+g_2(x,y))^2+g_4(x,z)$ or by $-w^2=-(y^2+g_2(x,y))^2+g_4(x,z)$, then the curves
$$
\left\{\aligned
&z=\xi x,\\
&w=\big(y^2+g_2(x,z)\big)
\endaligned
\right.
$$
and
$$
\left\{\aligned
&z=\bar{\xi} x,\\
&w=\big(y^2+g_2(x,z)\big)
\endaligned
\right.
$$
are $(-1)$-curves in $S$ that intersect transversally by $1$ point, namely $[0:1:0:1]$. 
Since the sum of these two curves is a~$G$-invariant divisor on $S$ that is defined over $\mathbb{R}$,
we immediately obtain a~contradiction with the~condition $\Pic(S)^G\simeq\mathbb{Z}$, because this divisor has self-intersection zero.
Similarly, if $S$ is given by $w^2=-(y^2+g_2(x,y))^2+g_4(x,z)$ or by $-w^2=(y^2+g_2(x,y))^2+g_4(x,z)$, then the curves
$$
\left\{\aligned
&z=\xi x,\\
&w=i\big(y^2+g_2(x,z)\big)	
\endaligned
\right.
$$
and
$$
\left\{\aligned
&z=\bar{\xi} x,\\
&w=-i\big(y^2+g_2(x,z)\big)	
\endaligned
\right.
$$
are disjoint $(-1)$-curves in $S$.
This also contradicts $\Pic(S)^G\simeq\mathbb{Z}$,
because the sum of these two curves is a~$G$-invariant divisor on $S$ that is defined over $\mathbb{R}$,
and its square is $-2$.

Thus, we see that the~polynomial $g_4(x,z)$ has four real roots in $\p_\R^1$.
Keeping in mind that $C_4$ is smooth, we see that these roots are distinct.
Therefore, changing the coordinates $x$ and $z$,
we may assume that $g_4=\pm xz(x-z)(x-sz)$
for some real number $s>1$. Then
$$
f_4(x,y,z)=\pm\big(y^2+g_2(x,z)\big)^2\pm xz(x-z)(x-sz),
$$
so that the surface $S$ is given in $\p_\R(1,1,1,2)$ by
\begin{equation}
\label{equation:dP2-Kowalevskays}
\pm w^2=\pm\big(y^2+g_2(x,z)\big)^2\pm xz(x-z)(x-sz),
\end{equation}
where a~priori all $\pm$ in these two equations are independent of each other.

We claim that the first two $\pm$ in the~equation \eqref{equation:dP2-Kowalevskays} are actually not the same,
i.e. either we have $S=S_+$ and the~surface $S$ is given by $w^2=-(y^2+g_2(x,z))^2\pm xz(x-z)(x-sz)$,
or $S=S_-$ and $S$ is given by $-w^2=(y^2+g_2(x,z))^2\pm xz(x-z)(x-sz)$.
Indeed, if this is not the case, then the curve
$\{x=w-(y^2+g_2(x,z))=0\}\subset\p_\R(1,1,1,2)$
is a~real $G$-invariant $(-1)$-curve in $S$, which contradicts to $\Pic(S)^G\cong\mathbb{Z}$.
Thus, we see that
\begin{itemize}
\item either $S=S_+$ and $f_4=-(y^2+g_2(x,z))^2\pm xz(x-z)(x-sz)$,
\item or $S=S_-$ and $f_4=(y^2+g_2(x,z))^2\pm xz(x-z)(x-sz)$.
\end{itemize}
Now, we claim that at least two numbers among $g_2(1,0)$, $g_2(0,1)$, $g_2(1,1)$, $g_2(s,1)$ are positive. 
Indeed, if all these four numbers are non-positive, $C_4(\R)$ has points in two non consecutive sectors bordered by the lines $\{x=0\}$, $\{z=0\}$, $\{x=z\}$, $\{x=sz\}$ hence the locus $C_4(\R)$ contains at least two non-nested ovals,
which is impossible, since we already proved that $C_4(\R)$ is either a~single oval or a~two nested ovals. 
By the same reason, the three numbers $g_2(0,1)$, $g_2(1,1)$ and $g_2(s,1)$ cannot be all non-positive.
On the other hand, the group $\Aut(\p_\R^1)$ contains a~subgroup isomorphic to $(\mathbb{Z}/2)^2$ that transitively
permutes the points $[1:0]$, $[0:1]$, $[1:1]$ and $[s:1]$.
Hence, we may change the coordinates $x$ and $z$ further such that $g_2(1,0)>0$.
This simply means that $g_2(x,z)=ax^2+bxz+cz^2$ for some real numbers such that $a>0$. 

In particular, we see that the line $\{z=0\}$ does not contain points in $C_4(\mathbb{R})$,
i.e. all ovals of $C_4$ are contained in the~chart $z\ne 0$ as we assumed earlier in this section. 
Since at least one number among $g_2(0,1)$, $g_2(1,1)$ or $g_2(s,1)$ is positive, at least one of the lines $\{x=0\}$, $\{x=z\}$, $\{x=sz\}$ does not contain points in $C_4(\mathbb{R})$,
so that the real points of this line are contained in the~unbounded part of the complement $\mathbb{R}^2\setminus C_4(\mathbb{R})$,
where $\mathbb{R}^2$ is the subset in $\p_\R^2$ given by $z\ne 0$.
We can choose $f_4(x,y,z)$ to be negative on the unbounded part of $\mathbb{R}^2\setminus C_4(\mathbb{R})$, as explained in the beginning of the section, so $f_4(x,y,z)$ is negative at every real point of this line.
In particular, one of the numbers $f_4(0,0,1)$,  $f_4(1,0,1)$, $f_4(s,0,1)$ is $<0$,
so $f_4=-\big(y^2+ax^2+bxz+cz^2\big)^2\pm xz(x-z)(x-sz)$. It follows that $S=S_+$ and $S$ is given by $w^2=-(y^2+ax^2+bxz+cz^2)^2\pm xz(x-z)(x-sz)$.
\end{proof}

\begin{remark}
\label{remark:Kowalevskays-involution}
Suppose that $\tau$ is given by $[w:x:y:z]\mapsto[w:x:-y:z]$,
and $S$ is a~smooth surface in $\p_\R(1,1,1,2)$ that is given by
$$
w^2=-(y^2+ax^2+bxz+cz^2)^2\pm xz(x-z)(x-sz)
$$
for some real $a,b,c,s$ such that $s\ne 0, 1$.
Then $\Pic(S)^G\cong\mathbb{Z}$. Indeed, let $Y=S/G$. Then $Y$ is a~hypersurface in $\p_\R(1,1,2,2)$ that is given by
$$
w^2=-(u+ax^2+bxz+cz^2)^2\pm xz(x-z)(x-sz).
$$
where we consider $x$ and $z$ as coordinates on $\p_\R(1,1,2,2)$ of weight $1$,
and $u$ and $w$ are coordinates of weight $2$. Then $Y$ is a~so-called Iskovskikh surface.
Namely, the surface $Y$ is a~del Pezzo surface of degree $4$ that is singular at the points $[0:0:1:i]$ and $[0:0:1:-i]$,
its singularities at these points are ordinary double points, and there exists the following commutative diagram:
$$
\xymatrix{
\widetilde{Y}\ar@{->}[d]_{\eta}\ar@{->}[rr]^\nu&& Y\ar@{-->}[d]^{\chi}\\
\p^1_\R\ar@{=}[rr]&&\p^1_\R}
$$
where $\nu$ is the blow up of the points $[0:0:1:i]$ and $[0:0:1:-i]$,
and $\chi$ is the rational map given by $[x:z:u:w]\to[x:z]$.
Then $\eta$ is a~conic bundle, which is defined over $\mathbb{R}$.
Observe that $\eta$ has exactly four geometrically singular fibres --- these are the preimages of the curves in $Y$
that are cut out on $Y$ by the~equations $x=0$, $z=0$, $x=z$ and $x=sz$.
These fibres are conics in $\p_\R^2$ that are irreducible over $\R$.
This implies that $\Pic(\widetilde{Y})\cong\mathbb{Z}^2$, so that $\Pic(Y)\cong\mathbb{Z}$.
Therefore, we conclude that $\Pic(S)^G\cong\mathbb{Z}$.
\end{remark}

If $S$ is a~smooth del Pezzo surface in $\p_\R(1,1,1,2)$ that is given by
$$
w^2=-(y^2+ax^2+bxz+cz^2)^2\pm xz(x-z)(x-sz),
$$
for some real numbers $a,b,c,s$ such that $a>0$, $s>1$, the~curve $C_4$ is smooth,
and either $C_4(\mathbb{R})$ consists of one oval, or the~locus $C_4(\mathbb{R})$ consists of two nested ovals in $\p_\R^2(\R)$,
then the~involution $\tau\in\mathrm{Aut}(S)$ given by
$[x:y:z:w]\mapsto[x:-y:z:w]$ and the~corresponding birational involution in $\Bir(\mathbb{P}^2_{\mathbb{R}})$
will be both called a~\emph{Kowalevskaya involution}.
The class of such involutions in $\Bir(\p_\R^2)$ will be denoted by $\mathfrak{K}_1$.

\begin{example}
\label{example:Kowalevskays}
If $S=S_+$ and $f_4=-(y^2+x^2+z^2)^2-xz(x-z)(x-2z)$, then $C_4(\R)$ consists of a~single oval, the surface $S$ is smooth and rational,
one has $S(\R)\approx \Sph^2$, and the Kowalevskaya involution is given by $[x:y:z:w]\mapsto[x:-y:z:w]$.
Similarly, if $S=S_+$ and $f_4=-(y^2+3(3x-z)(3x-2z))^2-xz(x-z)(x-2z)$, then $C_4(\R)$ consists of two nested ovals, the surface $S$ is smooth and rational, one has $S(\R)\approx \Sph^1\times \Sph^1$, and the Kowalevskaya involution is given by $[x:y:z:w]\mapsto[x:-y:z:w]$.
\end{example}

Let us summarize the results of this section

\begin{proposition}
\label{proposition:Geiser}
If $\Pic(S)^G\simeq\mathbb{Z}$, then $\tau$ is the Geiser involution or a~Kowalevskaya  involution.
\end{proposition}

\begin{proof}
This follows from Proposition~\ref{prop:Kowalevskays-involution} and the fact that $F(\tau)\neq\varnothing$ by \cite[\S 8.1]{Yasinsky}. 
\end{proof}

Kowalevskaya involutions and Geiser involutions are never conjugate because one fixes an elliptic curve and the other does not.
If $\Pic(S)^G\simeq\mathbb{Z}$, then the surface $S$ is $G$-birationally rigid by Theorem~\ref{theorem:Manin-Segre}.
Hence, the surface $S$ is uniquely determined by the conjugacy class of the birational involution in $\Bir(\p_\R^2)$ that is induced by $\tau$, and two such involutions are conjugate if and only if the corresponding surfaces are $G$-isomorphic. However, the following question remains open:
\begin{question}
	Is a Kowalevskaya involution uniquely determined by its fixed curve?
\end{question}

\section{De Jonqui\`{e}res involutions}
\label{section:exceptional-conic-bundles}

This and the next sections are both devoted to the study of involutions of real conic bundles. We have treated involutions on conic bundles inducing an involution on $\p^1_\R$ in Section~\ref{subsec: Trepalin classification}, so now we are interested in involutions inducing the identity on the base. 
Following the~classical pattern \cite{DolgachevIskovskikh}, we distinguish two cases: $G$-exceptional and non-$G$-exceptional conic bundles.
The~corresponding birational involutions in $\Bir(\p_\R^2)$ will be called \emph{de Jonqui\`{e}res involutions}
and (\emph{$0$-twisted}, \emph{$1$-twisted}, \emph{$2$-twisted}) \emph{Iskovskikh involutions}, respectively.
In this section, all surfaces are assumed to be $\R$-rational.

\subsection{Explicit models}
Let $S$ be a~real smooth projective surface that is rational over~$\R$,
let $\tau$ be an involution in $\Aut(S)$, and let $G=\langle\tau\rangle$.
Suppose that there exists a~$G$-minimal conic bundle $\pi\colon S\to\p^1_\R$ such that $G$ acts trivially on the base of $\pi$. 

The following definition is an~adaptation of a~well-known definition of an  \emph{exceptional conic bundle} to our setting,  cf. \cite[\S~5.2]{DolgachevIskovskikh} and \cite[Definition 13]{BlancMangolte}.

\begin{definition}
\label{definition:exceptional-conic-bundle}
The $G$-conic bundle $\pi\colon S\to\p^1_\R$ is said to be \emph{$G$-exceptional}
if $\pi$ admits a real section.
\end{definition}

We now give several equivalent characterizations of $G$-exceptional  real conic bundles.

\begin{proposition}
	\label{proposition:exceptional-conic-bundles}
	Let $\pi\colon S\to\p_\R^1$ be a~real $G$-conic bundle. Assume that $S$ is $\R$-rational.
	Then the following conditions are equivalent:
	\begin{enumerate}
		\item $S$ is $G$-exceptional;
		\item $\pi$ admits two real sections $Z_1$ and $Z_2$ such that $Z_1+Z_2$ is $G$-invariant.
		\item the relatively minimal model of the conic bundle $\pi$ is a~Hirzerbruch surface $\FFF_n$;
		\item $\pi(S(\R))=\p_\R^1(\R)$.
	\end{enumerate}
\end{proposition}

\begin{proof}
	First, (1) and (2) are obviously equivalent, since having a real section $Z_1$ one can take $Z_2=\tau(Z_1)$.
	The implication (1)$\Rightarrow$(3) follows from the existence of a~real section of the conic bundle~$\pi$.
	Conversely, if there exists some~birational morphism $S\to\F_n$ over $\p^1_\R$ for some $n\geqslant 0$,
	then letting $Z_1$ be the strict transform of the special section (or any $0$-section if $n=0$), we get a section defined over $\R$. By Remark \ref{remark:conic-bundle-intervals}, also (3) and (4) are equivalent. 
\end{proof}

Let us show that every $G$-exceptional conic bundle is $G$-birational to a~hypersurface in the weighted projective space $\p(n,n,1,1)$ of degree $2n$.
In fact, we obtain even more general result, namely explicit equations for any $G$-exceptional $G$-minimal conic bundle.

\begin{lemma}
\label{lemma:exceptional}
Suppose that the conic bundle $\pi$ admits two real sections $Z_1$ and $Z_2$ such that $Z_1+Z_2$ is $G$-invariant.
Assume  that $\tau$ acts trivially on the base of $\pi$.
Then there exists a~$G$-equivariant birational map $\chi\colon S\dasharrow X$ that fits into the following commutative $G$-equivariant diagram:
\begin{equation}
\label{equation:exceptional}
\xymatrix{
S\ar@{->}[d]_{\pi}\ar@{-->}[rr]^\chi&& X\ar@{->}[d]^{\eta}\ar@{->}[rr]^{\rho}&&Y\ar@{-->}[dll]\\
\p^1_\R\ar@{=}[rr]&&\p^1_\R&&}
\end{equation}
where $X$ is a~smooth real surface, $\eta$~is a~$G$-minimal conic bundle, and
$Y$ is a~hypersurface in the weighted projective space $\mathbb{P}(n,n,1,1)$ of degree $2n=8-K_S^2$ that is given by
$$
xy=f(z,t)
$$
for some real (homogeneous) polynomial $f(z,t)$ of degree $2n$ that has no multiple roots.
The curves $Z_1$ and $Z_2$ are $\rho\circ\chi$-exceptional, $G$ acts on $Y$ as $[x:y:z:t]\mapsto [y:x:z:t]$,
the~map $Y\dasharrow\p^1_\R$ is given by $[x:y:z:t]\mapsto [z:t]$,
the morphism $\rho$ is the~minimal resolution,
and $x$, $y$, $z$, $t$ are coordinates on $\p(n,n,1,1)$ of weights $n$, $n$, $1$, $1$, respectively.
\end{lemma}

\begin{proof}
The~proof is the~same as the~proof of \cite[Lemma-Definition 13, Lemma~15]{BlancMangolte}.
Indeed, if \mbox{$Z_1\cap Z_2=\varnothing$}, we can let $\chi=\mathrm{Id}_{S}$, since
$Z_1^2=Z_2^2=\frac{1}{2}(K_S^2-8)\leqslant 0$ \cite{BlancMangolte},
and the linear system $|Z_1+Z_2+nF|$ gives birational morphism from $S=X$ to the required hypersurface in $\mathbb{P}(n,n,1,1)$,
where $F$ is a~fibre of the conic bundle $\pi$, and $n=-Z_1^2$.
The proofs of the remaining assertions are similar to what is done in \cite[Section~5.2]{DolgachevIskovskikh}
and are left to the reader.
	
If \mbox{$Z_1\cap Z_2\ne\varnothing$}, then singular fibres of the conic bundle~$\pi$ do not contain any point~in the~intersection \mbox{$Z_1\cap Z_2$}.
In this case, there exists a~$G$-equivariant  commutative diagram
$$
\xymatrix{
&W\ar@{->}[dl]_{\alpha}\ar@{->}[dr]^{\beta}&\\
S\ar@{->}[d]_{\pi}&&\widehat{S}\ar@{->}[d]^{\widehat{\pi}}\\
\p^1_\R\ar@{=}[rr]&&\p^1_\R}
$$
such that $\alpha$ is the blow up of the set  $Z_1\cap Z_2$,
and $\beta$ is the contraction of the proper transforms of the fibres of $\pi$
that contain points in $Z_1\cap Z_2$. We have $\widehat{Z}_1\cdot\widehat{Z}_1<Z_1\cdot Z_1$,
where $\widehat{Z}_1$ and $\widehat{Z}_2$ are proper transforms on $\widehat{S}$ of the curves $Z_1$ and $Z_2$, respectively.
Iterating the described construction, we obtain the required map $S\dasharrow X$.
\end{proof}

\begin{remark}
	The equivalent conditions of Proposition \ref{proposition:exceptional-conic-bundles} are thus equivalent to 
	\begin{enumerate}
		\item[(5)] there exists a $G$-equivariant diagram (\ref{equation:exceptional}).
	\end{enumerate}
\end{remark}

\begin{remark}
\label{remark:exceptional-rational}
Observe that all surfaces in the diagram \eqref{equation:exceptional} are rational over the field $\R$,
and this does not impose any additional restriction on the polynomial $f(z,t)$.
\end{remark}

If the conic bundle $\pi\colon S\to\p^1_\R$ is $G$-exceptional and $K_S^2=8-2n\leqslant 4$,
the involution $\tau$ and the~corresponding birational involution in $\Bir(\mathbb{P}^2_{\mathbb{R}})$
will be both called \emph{de Jonqui\`{e}res involution} of genus $g=n-1\geqslant 1$, where $2n$ is the number of singular fibres of the conic bundle $\pi$.
The class of such birational  involutions in $\Bir(\mathbb{P}^2_{\mathbb{R}})$ will be denoted by $\mathfrak{dJ}_{g}$.
Note that it follows from Lemma~\ref{lemma:exceptional} that de Jonqui\`{e}res involutions in $\Bir(\mathbb{P}^2_{\mathbb{R}})$
are given by the same formulas as classical de Jonqui\`{e}res involutions in $\Bir(\mathbb{P}^2_{\mathbb{C}})$.

\subsection{Real hyperelliptic curves}
\label{ss:fixedcurve}
Let us use assumptions and notations of Lemma~\ref{lemma:exceptional}. The fixed locus of the~biregular involution $\tau\in\Aut(X)$ is the curve $C\simeq\rho(C)$,
where $\rho(C)$ is given by
$$
\left\{\aligned
&x=y,\\
&x^2=f(z,t).
\endaligned
\right.
$$
If $n\geqslant 3$, then $\rho(C)$ is a~real hyperelliptic curve of genus $g=n-1$ with hyperelliptic covering
$\nu\colon C\to\p^1_\R$ given by $[x:y:z:t]\mapsto [z:t]$.
If $n=2$, then $\rho(C)$ is an elliptic curve. Over $\C$, the~curve $C$ is determined by the~roots in $\p^1_{\C}$ of the~polynomial $f(z,t)$.
However, over $\R$, there are two forms of this (complex) curve: the~curve $C=C_{+}$ and the~curve $C_{-}$ that is given in $\mathbb{P}(n,1,1)$ by $x^2=-f(z,t)$.
These two real curves are isomorphic over $\C$, but they are not \emph{always} isomorphic over $\mathbb{R}$.
If $C_+\cong C_-$, the curve $C$ is called a~\emph{Gaussian curve} in \cite{HuismannLattarulo}.

\begin{lemma}[cf. {\cite[\S~2]{HuismannLattarulo}}]
\label{lemma:HuismanaLattarulo}
Suppose that $n\geqslant 3$. Then $C_{+}\simeq C_{-}$ if and only if there exists $\beta\in\mathrm{GL}_{2}(\mathbb{R})$ such that $\beta^*(f)=-f$.
\end{lemma}

\begin{proof}
The required assertion follows from \cite[Theorem 2.3]{HuismannLattarulo}.
\end{proof}

\begin{example}
Let $f(z,t)=zt(z-t)(z-2t)(z^2+t^2)$.
Then there is no non trivial elements $\beta\in\mathrm{GL}_{2}(\mathbb{R})$ such that $\beta^*(f)=-f$.
Thus, by Lemma~\ref{lemma:HuismanaLattarulo}, we see that $C_+$ is not isomorphic to $C_-$ over the reals.
\end{example}

For the surface $Y$ given by $xy=f(z,t)$, we set $Y_{+}=Y$, and we let $Y_{-}$ be $xy=-f(z,t)$.
Let $X_+=X$ and $X_-$ be their~minimal resolutions, respectively. The action of $G$ by permuting $x$ and $y$ gives these surfaces the structure of $G$-varieties.
Then Lemma~\ref{lemma:HuismanaLattarulo} gives: 

\begin{corollary}\label{corollary:HuismanaLattarulo}
In the notation of Lemma~\ref{lemma:exceptional}, suppose that $n\geqslant 3$.
Then one has a $G$-isomorphism $X_{+}\simeq X_{-}$ over $\mathbb{R}$ if and only if $C_{+}\simeq C_{-}$ over $\mathbb{R}$. 
\end{corollary}

\subsection{$G$-birational classification}
\label{subsection:de-Jonquires-involution-conjugation}

We are ready to prove the main result of this section.
Let $S$ and $S^\prime$ be two real smooth  $\mathbb{R}$-rational projective surfaces,
and let $\tau$ and $\tau^\prime$ involutions in $\Aut(S)$ and $\Aut(S^\prime)$, respectively.
Set $G=\langle\tau\rangle$ and $G^\prime=\langle\tau^\prime\rangle$.
Suppose, in addition,  there are $G$-minimal conic bundle $\pi\colon S\to\p^1_\R$
and $G^\prime$-minimal conic bundle \mbox{$\pi^\prime\colon S^\prime\to\p^1_\R$}.

\begin{theorem}
\label{theorem:exceptional-conic-bundles-n-large}
Suppose that the fixed curve $F(\tau)$ is a~smooth curve of genus $g\geqslant 2$,
the~$G$-conic bundle $\pi$ is~$G$-exceptional, and the~$G^\prime$-conic bundle $\pi^\prime$ is $G^\prime$-exceptional.
Then there is a~birational map $\vartheta\colon S\dasharrow S^\prime$ such that $\tau^\prime=\vartheta\circ\tau\circ\vartheta^{-1}$
if and only if $F(\tau)\simeq F(\tau^\prime)$.
\end{theorem}

\begin{proof}
Suppose that $F(\tau)\simeq F(\tau^\prime)$.
To complete the proof, it is enough to show that there exists a~birational map $\vartheta\colon S\dasharrow S^\prime$ such that $\tau^\prime=\vartheta\circ\tau\circ\vartheta^{-1}$. By Lemma~\ref{lemma:exceptional},
we may assume that $S$ is given by $xy=f(z,t)$ in $\mathbb{P}(n,n,1,1)$ for some homogeneous polynomial $f(z,t)$ of degree $2g+2$ that does not have multiple roots,
and~$\tau$~is given by $[x:y:z:t]\mapsto[y:x:z:t]$.
Likewise, we may assume that $S^\prime$ is given by $xy=f^\prime(z,t)$ in $\mathbb{P}(n,n,1,1)$ for some polynomial $f^\prime(z,t)$ of degree $2g+2$ that does not have multiple roots,
and $\tau^\prime$ is given by $[x:y:z:t]\mapsto[y:x:z:t]$.

If $F(\tau)\simeq F(\tau^\prime)$,
then it follows from  \cite[Lemma 2.1]{HuismannLattarulo} that there are real numbers $a$, $b$, $c$, $d$ such that $ad-bc\ne 0$ and $f(az+bt,cz+dt)=\pm f^\prime(z,t)$. Moreover, using Lemma \ref{lemma:HuismanaLattarulo}, we may choose the sign to be $+$.
Now $f(az+bt,cz+dt)=f^\prime(z,t)$ and we let $\vartheta$ to be the map given by $[x:y:z:t]\mapsto [x:y:az+bt:cz+dt]$.
\end{proof}

\section{Iskovskikh involutions}
\label{section:non-exceptional}

Now, we present another three constructions of birational involutions of the real projective plane that fix hyperelliptic curves.
Over $\mathbb{C}$, all these involutions are conjugate to de Jonqui\`{e}res involutions described in Section~\ref{section:exceptional-conic-bundles},
which is not the case over $\mathbb{R}$.
The involutions described in this section are given by the involutions of non-$G$-exceptional real conic bundles (see Definition~\ref{definition:exceptional-conic-bundle}).
Since Vasily Iskovskikh pioneered the study of conic bundles over algebraically non-closed fields
\cite{Iskovskikh1967,Iskovskikh1970,Iskovskikh80},
we decided to call the constructed birational involutions (\mbox{\emph{$0$-twisted}}, \emph{$1$-twisted}, \emph{$2$-twisted}) \emph{Iskovskikh involutions}.

First, we give an explicit birational model of a~non-$G$-exceptional real conic bundle,
and present a tool that can be used to check whether two such $G$-conic bundles are $G$-birational or not.
This is done in Theorems~\ref{theorem:main-non-exceptional} and~\ref{theorem:main-non-exceptional-equivalence}, respectively.

\subsection{Good birational model}
\label{subsection:good-model-non-exceptional}
Let $\pi\colon S\to\p^1_\R$ be a~smooth real $G$-minimal conic bundle,
where $G=\langle\tau\rangle$ for an involution $\tau\in\mathrm{Aut}(S)$ that acts by identity on $\p^1_\R$.
By Proposition~\ref{proposition:exceptional-conic-bundles}, this conic bundle  is not $G$-exceptional
if and only if $\pi(S(\R))$ is a~union of intervals in $\p^1_\R(\R)$.
The main result of this section is the following theorem, which we will prove in Section~\ref{subsection:birational-model-proof}.

\begin{theorem}
\label{theorem:main-non-exceptional}
Suppose that the $G$-minimal conic bundle $\pi\colon S\to\p^1_\R$ is not $G$-exceptional.
Then there exists a $G$-equivariant commutative diagram
\begin{equation}
\label{equation:main-non-exceptional}
\xymatrix{
S\ar@{->}[d]_{\pi}\ar@{-->}[rr]^{\chi}&&X\ar@{->}[d]^{\eta}\\
\p^1_\R\ar@{->}[rr]_\phi&&\mathbb{P}^1_\R}
\end{equation}
where $\chi$ is a birational map, $\phi\in\mathrm{PGL}_2(\mathbb{R})$, $X$ is a smooth surface, $\eta$ is a $G$-minimal conic bundle,
the fibre $\eta^{-1}([1:0])$ is smooth and does not have real points,
the quasi-projective surface $Y=X\setminus \eta^{-1}([1:0])$ is given in $\mathbb{P}^2_\mathbb{R}\times\mathbb{A}^1_{\mathbb{R}}$ by
\begin{equation}
\label{equation:non-exceptional-model}
A(t)x^2+B(t)xy+C(t)y^2=H(t)z^2
\end{equation}
for some polynomials $A,B,C,H\in\mathbb{R}[t]$ such that $(B^2-4AC)H$ does not have multiple roots and $\deg(B^2-4AC)$ is even,
the involution $\tau$ acts on the surface $Y$ by
$$
([x:y:z],t)\mapsto([x:y:-z],t),
$$
and the restriction map $\eta\vert_{Y}\colon Y\to\mathbb{P}^1_\R\setminus [1:0]=\mathbb{A}^1_{\R}$ is the map given by $([x:y:z],t)\mapsto t$,
where $([x:y:z],t)$ are coordinates on $\mathbb{P}^2_\mathbb{R}\times\mathbb{A}^1_{\mathbb{R}}$.
Moreover, the following holds:
\begin{enumerate}
\item the polynomial $H(t)$ has only real roots and its leading coefficient is negative,
\item fibres of $\eta$ over roots of the polynomial $H(t)$ are singular irreducible conics.
\end{enumerate}
\end{theorem}

\begin{remark*}
	Note that $A\ne 0$ and $C\ne 0$ in \eqref{equation:non-exceptional-model} since our conic bundle $\pi\colon S\to\p_\R^1$ is assumed to be non-$G$-exceptional. 
\end{remark*}

\begin{corollary}
\label{corollary:fixed-curve-non-exceptional}
In the assumptions and notations of Theorem~\ref{theorem:main-non-exceptional},
either $F(\tau)=\varnothing$, or~the $G$-fixed curve $F(\tau)$ is birational to the real algebraic curve given in $\mathbb{A}^2_\R$ by $w^2=B(t)^2-4A(t)C(t)$,
which is elliptic if $\deg(B^2-4AC)=4$ and hyperelliptic if $\deg(B^2-4AC)\geqslant 6$. 
\end{corollary}

\begin{proof}
The $G$-fixed curve $F(\tau)$ is birational to the curve $\{A(t)x^2+B(t)xy+C(t)y^2=0\}\subset\mathbb{P}^1_\mathbb{R}\times\mathbb{A}^1_{\mathbb{R}}$,
where $([x:y],t)$ are coordinates on $\mathbb{P}^1_\mathbb{R}\times\mathbb{A}^1_{\mathbb{R}}$.
Setting $w=2\left(A\frac{x}{y}+\frac{B}{2}\right)$, we get the result.
\end{proof}

\begin{corollary}
\label{corollary:main-non-exceptional}
In the assumptions and notations of Theorem~\ref{theorem:main-non-exceptional},
there exists a $G$-equivariant commutative diagram
$$
\xymatrix{
Y\ar@{->}[d]_{\pi_Y}\ar@{-->}[rr]&&\widehat{Y}\ar@{->}[d]^{\pi_{\widehat{Y}}}\\
\mathbb{A}^1_\R\ar@{=}[rr]&&\mathbb{A}^1_\R}
$$
where $\widehat{Y}$ is a (possibly singular) surface in $\mathbb{P}^2_\mathbb{R}\times\mathbb{A}^1_{\mathbb{R}}$ given by $\widehat{A}(t)x^2+\widehat{C}(t)y^2=H(t)z^2$
for some non-zero polynomials $\widehat{A},\widehat{C}\in\R[t]$ such that $\widehat{C}=\widehat{A}(4AC-B^2)$,
both $\pi_Y$ and $\pi_{\widehat{Y}}$ are given by $([x:y:z],t)\mapsto t$,
the $G$-action on $\widehat{Y}$ is given by $([x:y:z],t)\mapsto([x:y:-z],t)$,
and $Y\dasharrow\widehat{Y}$ is a~birational map that is biregular along singular fibres of the conic bundle $\pi_Y$.
Moreover, we have:
\begin{itemize}
\item $\widehat{A}$ and $\widehat{C}$ have no multiple roots,
\item $\widehat{A}$ and $H$ are co-prime,
\item $\widehat{C}$ and $H$ are co-prime.
\end{itemize}
In particular, $H$ has only real roots,
the fibres of the~conic bundle $\pi_{\widehat{Y}}$ over roots of  $H$ are singular irreducible conics,
$\pi_{\widehat{Y}}(\widehat{Y}(\R))$ is a~union of closed bounded intervals in $\mathbb{A}^1_\R(\R)$,
and the fibres of the~conic bundle $\pi_{\widehat{Y}}$ over the boundary points of the intervals in $\pi_{\widehat{Y}}(\widehat{Y}(\R))$ are reduced.
\end{corollary}

\begin{proof}
Changing variables $[x:y:z]\mapsto[ax+by:cx+dy:z]$ for general real numbers $a,b,c,d$ such that $ad-bc\ne 0$,
we may assume that $A(t)$ does not vanish at the roots of $(4AC-B^2)H$. Introducing new coordinates $\widehat{x}=Ax+\frac{B}{2}y$, $\widehat{y}=\frac{y}{2}$, $\widehat{z}=Az$,
we get a $G$-equivariant birational map $Y\dasharrow\widehat{Y}$, where
$\widehat{Y}$ is a surface in $\mathbb{P}^1_\mathbb{R}\times\mathbb{A}^1_{\mathbb{R}}$ given by
$A\widehat{x}^2+A(4AC-B^2)\widehat{y}^2=H\widehat{z}^2$,
where $([\widehat{x}:\widehat{y}:\widehat{z}],t)$ are coordinates on $\mathbb{P}^2_\mathbb{R}\times\mathbb{A}^1_{\mathbb{R}}$.
This implies the required assertion.
\end{proof}

Note that the equation \eqref{equation:non-exceptional-model} is not canonically defined by the conic bundle \mbox{$\pi\colon S\to\mathbb{P}^1_{\mathbb{R}}$}.
To check whether two equations like \eqref{equation:non-exceptional-model} determine $G$-birationally equivalent $G$-conic bundles or not,
one can use the following result, which we will prove in Section~\ref{subsection:non-exceptional-birational-classification}.

\begin{theorem}
\label{theorem:main-non-exceptional-equivalence}
Let $\eta_1\colon X_1\to\mathbb{P}^1_\R$ and $\eta_1\colon X_2\to\mathbb{P}^1_\R$ be two smooth $G$-minimal $G$-conic bundles as on the right-hand side of Diagram  (\ref{equation:main-non-exceptional}). 
Suppose that $Y_1=X_1\setminus \eta_1^{-1}([1:0])$ and $Y_2=X_2\setminus \eta_2^{-1}([1:0])$ are given in $\mathbb{P}^2_\mathbb{R}\times\mathbb{A}^1_{\mathbb{R}}$ by
\begin{align*}
Y_1\colon A_1x^2+B_1xy+C_1y^2&=H_1z^2,\\
Y_2\colon A_2x^2+B_2xy+C_2y^2&=H_2z^2,
\end{align*}
for  $A_1,B_1,C_1,H_1,A_2,B_2,C_2,H_2\in\mathbb{R}[t]$
such that $(B_1^2-4A_1C_1)H_1$ and $(B_2^2-4A_2C_2)H_1$ do not have multiple roots,
and the degrees of   $B_1^2-4A_1C_1$ and $B_2^2-4A_2C_2$ are even,
where both restrictions $\eta_{1}\vert_{Y_1}\colon Y_1\to\mathbb{A}^1_{\mathbb{R}}$ and $\eta_{2}\vert_{Y_2}\colon Y_2\to\mathbb{A}^2_{\mathbb{R}}$
are given by $([x:y:z],t)\mapsto t$.
Suppose further that the following conditions are satisfied:
\begin{enumerate}
\item  $H_1$ and $H_2$ have only real roots and their leading coefficients are negative,
\item fibres of the $G$-conic bundles $\eta_1$ and $\eta_2$ over all roots of  $H_1$ and $H_2$ are singular irreducible conics, respectively.
\end{enumerate}
Suppose that we have the $G$-action on $X_1$ and $X_2$ such that $G$ acts on $Y_1$ and $Y_2$ as follows:
$
([x:y:z],t)\mapsto([x:y:-z],t).
$
Then there is a $G$-equivariant birational map $\rho\colon X_1\dasharrow X_2$ that fits commutative diagram
\begin{equation}
\label{equation:main-non-exceptional-equivalence}
\xymatrix{
X_1\ar@{->}[d]_{\eta_1}\ar@{-->}[rr]^{\rho}&&X_2\ar@{->}[d]^{\eta_2}\\
\p^1_\R\ar@{=}[rr]&&\mathbb{P}^1_\R}
\end{equation}
if and only if $\eta_1(X_1(\R))=\eta_2(X_2(\R))$, $B_1^2-4A_1C_1=\lambda(B_2^2-4A_2C_2)$ and $H_1=\mu H_2$
for some positive real numbers $\lambda$ and $\mu$. 
\end{theorem}

The conditions $B_1^2-4A_1C_1=\lambda(B_2^2-4A_2C_2)$ and $H_1=\mu H_2$ in Theorem~\ref{theorem:main-non-exceptional-equivalence} are essential
and cannot be omitted, which follows from Corollary~\ref{corollary:fixed-curve-non-exceptional} and Section~\ref{subsection:special-fibres} below.
In fact, we cannot omit the condition $\eta_1(X_1(\R))=\eta_2(X_2(\R))$ either.

\begin{example}
\label{remark:conic-bundles-rational-non-rational}
In the notations of Theorem~\ref{theorem:main-non-exceptional-equivalence},
suppose that $Y_1$ is given in $\mathbb{P}^2_\mathbb{R}\times\mathbb{A}^1_{\mathbb{R}}$ by
$$
(t-1)(t-2)(t^2+1)x^2+(t-3)(t-4)(t^2+2)y^2=tz^2,
$$
and $Y_2$ is is given in $\mathbb{P}^2_\mathbb{R}\times\mathbb{A}^1_{\mathbb{R}}$ by 
$$(t-1)(t-3)(t^2+1)x^2+(t-2)(t-4)(t^2+2)y^2=tz^2$$
Thus, we have $H_1=H_2=t$ and
$$
B_1^2-4A_1C_1=B_2^2-4A_2C_2=-4(t-1)(t-2)(t-3)(t-4)(t^2+1)(t^2+2).
$$
But $X_1$ is $\R$-rational and $X_2$ is not $\R$-rational, since $X_1(\R)$ is connected and $X_2(\R)$~is~not. Note that the fixed curves here are hyperelliptic curves of genus $2$.
\end{example}

\subsection{Special fibres}
\label{subsection:special-fibres}

Let $\pi\colon S\to\p^1_\R$ be a~smooth real $G$-conic bundle that is $G$-minimal,
where $G=\langle\tau\rangle$ for an involution $\tau\in\mathrm{Aut}(S)$ that acts trivially on $\p^1_\R$. 
If $F(\tau)\neq\varnothing$, then 
$\tau$ pointwise fixes a~smooth irreducible curve $C\subset S$ that is a two-section of the conic bundle $\pi$, see Corollary~\ref{corollary:fixed-curve-non-exceptional}. In particular, $K_S^2\leqslant 4$. The case $K_S^2\geqslant5$ (and therefore $F(\tau)=\varnothing$) will be dealt with in Section~\ref{section:classification}. 

Let $g$ be the genus of the curve~$C$, let $F$ be a~(complex) singular fibre of the conic bundle $\pi$, let $F_1$ and $F_2$ be its irreducible components.
Then $F_1\cdot C=F_2\cdot C=1$, and the intersection $F_1\cap F_2$ consists of one~point.
Moreover, it follows from the $G$-minimality of the conic bundle $\pi\colon S\to\p^1_\R$ that
\begin{itemize}
\item either $F\cap C=F_1\cap F_2$,
\item or $F_1\cap F_2\not\in C$ and $F$ is real.
\end{itemize}
In the latter case, the curves $F_1$ and $F_2$ are both $G$-invariant and thus must be swapped by the action of the Galois group $\mathrm{Gal}(\mathbb{C}/\mathbb{R})$.
In this case (when $F_1\cap F_2\not\in C$), we will say that the singular fibre $F=F_1+F_2$ is \emph{special}. We let
$$
\SF_S=\text{the number of special fibres of the $G$-minimal conic bundle $\pi\colon S\to\p^1_\R$}.
$$

We have $K_S^2=6-2g-\SF_S$, because $\pi\colon S\to\p^1$ has $2g+2+\delta_S$ singular fibres. Moreover, if $\pi\colon S\to\p^1_\R$ is $G$-exceptional,~then~$\SF_S=0$.
Furthermore, if $S$ is $\R$-rational, then $\SF_S\in\{0,1,2\}$, see Remark~\ref{remark:conic-bundle-intervals}.

\begin{lemma}
	\label{lemma:special-fibres}
	Suppose that $g\geqslant 1$. Let $\chi\colon S\dasharrow X$ be a $G$-birational map such that $X$ is a smooth surface,
	and there exist $G$-minimal conic bundle $\eta\colon X\to\p^1_\R$. Then $\SF_S=\SF_{X}$.
\end{lemma}

\begin{proof}
	Follows from Lemma~\ref{rem:K2=4-and-rkPic=2} and the equality $K_S^2=6-2g-\SF_S$.
\end{proof}

In particular, if $g\geqslant 1$ and $\delta_S>0$, then  $S$ is not $G$-birational to a $G$-exceptional $G$-minimal conic bundle.
In fact, we can say more:

\begin{lemma}
	\label{lemma:non-exceptional-exceptional}
	Suppose that $g\geqslant 2$. If $\pi\colon S\to\p^1_\R$ is not $G$-exceptional, then $S$ is not $G$-birational to a $G$-exceptional $G$-minimal conic bundle.
\end{lemma}

\begin{proof}
Lemma~\ref{rem:K2=4-and-rkPic=2} implies that if $K_S^2\leqslant0$, then $S$ is birationally rigid, and if $K_S^2\in\{1,2\}$, then any $G$-equivariant birational map from $S$ to another $G$-minimal surface is a composition of $G$-equivariant elementary transformations of $G$-conic bundles and the exchange of two fibrations by a Geiser or Bertini involution. These involutions preserve the number of special fibres.
\end{proof}

On the other hand, if $g=1$ and $\delta_S=0$ and the conic bundle  $\pi\colon S\to\p^1_\R$ is not $G$-exceptional,
then $S$ can be $G$-birational to a $G$-exceptional $G$-minimal conic bundle, see Example~\ref{example:non-exceptional-exceptional}. 

It is very easy to locate special fibres of the $G$-conic bundle provided by Theorem~\ref{theorem:main-non-exceptional}.
Namely, in the assumptions and notations of Theorem~\ref{theorem:main-non-exceptional},
the special fibres of the conic bundle $\eta\colon X\to\p^1_\R$ are the fibres of the~morphism $\eta$ over the roots of the polynomial~$H$.
Then $\SF_X=\mathrm{deg}(H)$.

\subsection{Iskovskikh involutions}
\label{subsection:Iskovskikh-involutions}

As in Section~\ref{subsection:special-fibres},
let $S$ be a real smooth projective surface, let $\tau$ be an involution in $\Aut(S)$,
let $G=\langle\tau\rangle$, and let $\pi\colon S\to\p^1_\R$ be a~$G$-minimal conic bundle such that $\tau$ acts trivially on $\p^1_\R$.
The involution $\tau$ pointwise fixes a~smooth irreducible curve $C\subset S$ that is a two-section of the conic bundle $\pi$.
Let $g$ be the genus of the curve~$C$, and let $\SF_S$ be the number of special fibres of the conic bundle $\pi$.
Then
$$
K_S^2=6-2g-\SF_S.
$$
Suppose, in addition, that the surface $S$ is $\R$-rational and $g\geqslant 1$. Then $\SF_S\in\{0,1,2\}$,
and $\tau$ induces a birational involution $\iota\in\Bir(\p_\R^2)$. If the conic bundle $\pi\colon S\to\p^1_\R$ is not $G$-exceptional, we say that both involutions $\tau$ and $\iota$ are
\begin{itemize}
\item \emph{$0$-twisted Iskovskikh involution} if $\SF_S=0$,
\item \emph{$1$-twisted Iskovskikh involution} if $\SF_S=1$,
\item \emph{$2$-twisted Iskovskikh involution} if $\SF_S=2$.
\end{itemize}
We will denote by $\mathfrak{I}_{g}$ the class of all $0$-twisted Iskovskikh  birational involutions in $\Bir(\p_\R^2)$ whose fixed curves have genus $g\geqslant 1$.
Similarly, we will denote by $\mathfrak{I}_{g}^\prime$ and  $\mathfrak{I}_{g}^{\prime\prime}$
the classes consisting of all $1$-twisted Iskovskikh involutions and $2$-twisted Iskovskikh involutions whose fixed curves have genus $g\geqslant 1$,
respectively. We deliberately do not define the classes $\mathfrak{I}_0,\mathfrak{I}_0'$ and $\mathfrak{I}_0''$, as we shall see in the proof of Main Theorem (Section \ref{section:classification}) that they are contained in $\mathfrak{L}\cup\mathfrak{Q}\cup\mathfrak{T}_4$.

\begin{lemma}\label{lemma:Iskovskikh-involutions-conjugation}
Let $\jmath$, $\iota$, $\iota^\prime$, $\iota^{\prime\prime}$ be birational involutions in the classes $\mathfrak{dJ}_{g}$, $\mathfrak{I}_{g}$,
$\mathfrak{I}_{g}^\prime$, $\mathfrak{I}_{g}^{\prime\prime}$, respectively, where $g\geqslant 1$.
Then they are pairwise non-conjugate in  $\Bir(\p_\R^2)$ with only possible exception of involutions $\jmath$ and $\iota$
being conjugate in the case when $g=1$ and $F(\jmath)\simeq F(\iota)$.
\end{lemma}

\begin{proof}
The statement follows from Lemma~\ref{lemma:special-fibres} and Lemma \ref{lemma:non-exceptional-exceptional}.
\end{proof}

Thus, we see that the classes $\mathfrak{dJ}_{g}$, $\mathfrak{I}_{g}$,
$\mathfrak{I}_{g}^\prime$, $\mathfrak{I}_{g}^{\prime\prime}$ are pairwise disjoint for $g\geqslant 2$.
Similarly, the classes $\mathfrak{I}_{1}$, $\mathfrak{I}_{1}^\prime$, $\mathfrak{I}_{1}^{\prime\prime}$ are pairwise disjoint,
and the classes $\mathfrak{dJ}_{1}$, $\mathfrak{I}_{1}^\prime$, $\mathfrak{I}_{1}^{\prime\prime}$ are also pairwise disjoint.
On the other hand, we will see in Example~\ref{example:non-exceptional-exceptional} that $\mathfrak{dJ}_{1}\cap\mathfrak{I}_{1}\ne\varnothing$.
We do not know whether $\mathfrak{dJ}_{1}=\mathfrak{I}_{1}$ or not, nor whether $\mathfrak{dJ}_{1}\subseteq\mathfrak{I}_{1}$ or $\mathfrak{dJ}_{1}\supseteq\mathfrak{I}_{1}$.

\subsection{Proof of Theorem~\ref{theorem:main-non-exceptional}}
\label{subsection:birational-model-proof}
In this section, we prove Theorem~\ref{theorem:main-non-exceptional}.
During the proof, we~will slightly abuse our original conventions and consider quasi-projective conic bundles.
As usual, all surfaces are assumed to be real and  geometrically rational.

Let $S$ be a smooth projective surface, let $\tau$ be an~involution in~$\mathrm{Aut}(S)$, and let $G=\langle\tau\rangle$.
Suppose that there is a $G$-minimal conic bundle $\pi\colon S\to\p^1_\R$ such that $G$ acts trivially on $\p^1_\R$,
and $\pi$ is not $G$-exceptional. Then $\pi(S(\R))\ne\p^1_\R(\R)$ by Proposition~\ref{proposition:exceptional-conic-bundles}.
To prove Theorem~\ref{theorem:main-non-exceptional}, we must construct a $G$-equivariant commutative diagram \eqref{equation:main-non-exceptional}.

To start with, let us choose $\phi$ in \eqref{equation:main-non-exceptional}.
Let $P$ be any point in $\p^1_\R$ such that $P\not\in\pi(S(\R))$.
Then we take any $\phi\in\mathrm{PGL}_2(\mathbb{R})$ such that $\phi(P)=[1:0]$.
Now, we swap $\pi$ with $\phi\circ\pi$. Then the fibre $\pi^{-1}([1:0])$ is smooth and has no real points.
Hence, to prove Theorem~\ref{theorem:main-non-exceptional}, we may assume that the map $\phi$ in \eqref{equation:main-non-exceptional} is identity.

We set $U=S\setminus\pi^{-1}([1:0])$. Then $U$ is $G$-invariant, and $\pi$ induces $G$-equivariant conic bundle
$\pi_U\colon U\to\mathbb{A}^1_{\mathbb{R}}$.

\begin{lemma}
\label{lemma:first-form}
There exists a $G$-equivariant commutative diagram
$$
\xymatrix{
U\ar@{->}[d]_{\pi_U}\ar@{-->}[rr]&&W\ar@{->}[d]^{\pi_W}\\
\mathbb{A}^1_\R\ar@{=}[rr]&&\mathbb{A}^1_\R}
$$
such that $U\dashrightarrow W$ is a birational map, $W$ is a smooth surface in $\p^{2}_\R\times\mathbb{A}^1_\R$ given by
\begin{equation}
\label{equation:non-exceptional-model-1}
A(t)x^2+C(t)y^2=H(t)z^2
\end{equation}
for a suitable choice of coordinates $([x:y:z],t)$ on $\p^{2}_\R\times\mathbb{A}^1_\R$,
where $A,C,H\in \mathbb{R}[t]$,
the~morphism $\pi_W$ is induced by the natural projection $\p^{2}_\R\times\mathbb{A}^1_\R\to \mathbb{A}^1_\R$,
and the~$G$-action on the~surface $W$ is given by $([x:y:z],t)\mapsto([x:y:-z],t)$.
\end{lemma}

\begin{proof}
Observe that $U$ is isomorphic to a closed subset in $\p^2_\R\times\mathbb{A}^1_\R$,
and the conic bundle $\pi_U$ is induced by the natural projection $\p^{2}_\R\times\mathbb{A}^1_\R\to \mathbb{A}^1_\R$. 
The surface $U$ is given in $\p^{2}_\R\times\mathbb{A}^1_\R$ by
$$
A_0(t)x^2+B_0(t)xy+C_0(t)y^2+H_0(t)z^2+E_0(t)xz+F_0(t)yz=0
$$
for some polynomials $A_0,B_0,C_0,H_0,E_0,F_0\in\R[t]$.
This equation defines a conic $\mathcal{C}\subset\p^2_{\R(t)}$,
which is the scheme generic fibre of the~conic bundle $\pi_U\colon U\to\mathbb{A}^1_\R$.
By Proposition~\ref{proposition:exceptional-conic-bundles}, this conic $\mathcal{C}$ does not have points in $\R(t)$,
because the conic bundle $\pi$ is not $G$-exceptional.

Since $G$ acts trivially on $\mathbb{A}^1_\R$, its action on $U$ gives a geometric $G$-action on the conic~$\mathcal{C}$,
which lifts to a linear action on $\p^2_{\R(t)}$.
Therefore, linearly changing coordinates on $\p^2_{\R(t)}$,
we may assume that the~group $G$ acts on $\p^2_{\R(t)}$ by the involution $[x:y:z]\mapsto[x:y:-z]$.
Then the conic $\mathcal{C}$ is given in $\p^2_{\R(t)}$ by the equation
$$
A_1(t)x^2+B_1(t)xy+C_1(t)y^2=H_1(t)z^2
$$
for some $A_1$ $B_1$, $C_1$, $H_1$ in $\R(t)$, where we consider $[x:y:z]$ as coordinates on~$\p^2_{\R(t)}$.

Note that $A_1\ne 0$, because $\mathcal{C}$ has no points in $\R(t)$.
Thus, completing the squares and clearing the denominators, we may assume that $\mathcal{C}$ is given by
$$
A_2(t)x^2+C_2(t)y^2=H_2(t)z^2
$$
for some  polynomials $A_2,C_2,H_2\in\R[t]$.
We may further assume that $\gcd(A_2,C_2,H_2)=1$.
Scaling $x$, $y$, $z$, we may also assume that $A_2$, $C_2$, $H_2$ have no multiple roots.

If~$R=\gcd(A_2,C_2)\ne 1$, then $A_2=RA_3$ and $C_2=RC_3$, so multiplying the~equation of the conic $\mathcal{C}$ by $R$ and scaling the coordinates,
we see that $\mathcal{C}$ is now given by
$$
A_3(t)x^2+C_3(t)y^2=H_3(t)z^2,
$$
where $A_3$, $C_3$, $H_3=RH_2$ have no multiple roots, while $A_3$ and $C_3$ are co-prime. If $T=\gcd(A_2,H_2)\ne 1$, $A_2=TA_3$, $H_2=TH_3$ we can again multiply the equation by $T$, and after a change of coordinates get 
		\[
		A_3(t)x^2+C_3(t)y^2=H_3(t)z^2,
		\]
		where $C_3=TC_2$, $\gcd(A_3, H_3)=1$. Note that $C_3$ is also coprime to $A_3$. Indeed, if $\gcd(C_3,A_3)=P\ne 1$ then $P$ divides $A_2$. Since $\gcd(A_2,C_2)=1$, we have $\gcd(T,C_2)=1$ and thus $P$ divides either $T$ or $C_2$. The latter case is impossible, as otherwise $P$ divides both $A_2$ and $C_2$, in the former case $P$ divides both $T$ and $A_3$, so $A_2=TA_3$ has a multiple root, a contradiction. The same argument shows that we may reach $\gcd(C_3,H_3)=1$. 

Finally, let $W$ be the~closed subset in $\p^{2}_\R\times\mathbb{A}^1_\R$ that is given by $A_3(t)x^2+C_3(t)y^2=H_3(t)z^2$.
Then $W$ is smooth, since $A$, $C$, $H$ do not have multiple roots and are pairwise co-prime.
This completes the proof of the lemma, because the coordinate changes we did give the~required $G$-equivariant birational map $U\dasharrow W$, since we did not change~$t$.
\end{proof}

\begin{remark}
One of the anonymous referees suggested the following shorter and more elegant proof of Lemma~\ref{lemma:first-form}.
Namely, the surface $S$ embeds over $\p^1_\R$ into the projectivisation of the $G$-equivariant vector bundle $\pi_*\omega_{S/\p^1_\R}$ as relative conic fibration. 
All vector bundles over $\mathbb{A}^1_\R$ are trivial and all $G$-equivariant coherent sheaves split as direct sum of the $+1$ and $-1$ eigensheaves for the $G$-action. 
Consequently, the projective bundle has homogeneous coordinates $[x:y:z]$ such that the $G$-action is given by $[x:y:z]\mapsto[x:y:-z]$. Then $S$ has an equation of the claimed form. The assumption that $S$ is $G$-minimal imply the remaining claims of Lemma~\ref{lemma:first-form}. 
We thank the referee for their suggestion. 
\end{remark}

We just replaced our conic bundle $\pi_U\colon U\to \mathbb{A}^1_\R$ with the conic bundle \mbox{$\pi_W\colon W\to \mathbb{A}^1_\R$},
which is given in $\p^{2}_\R\times\mathbb{A}^1_\R$ by a very simple equation.
However, this comes with a price:
\begin{itemize}
\item the conic bundle $\pi_U\colon U\to\AAA_\R^1$ is relatively $G$-minimal over $\mathbb{A}^1_\R$ by our assumption,
\item the conic bundle $\pi_W\colon W\to \mathbb{A}^1_\R$ is not necessarily relatively $G$-minimal over $\mathbb{A}^1_\R$,
\end{itemize}
because the constructed fibrewise $G$-birational map $U\dasharrow W$ can destroy $G$-minimality.
Observe that the conic bundle $\pi_W\colon W\to\mathbb{A}^1_\R$ is relatively $G$-minimal over $\mathbb{A}^1_\R$
if and only~if the polynomial $H(t)$ in \eqref{equation:non-exceptional-model-1} has only real roots,
and the fibre of the~morphism $\pi_W$ over every root of the polynomial $H(t)$ is a singular irreducible conic.

Now, we are going to explicitly apply $G$-equivariant relative Minimal Model Program
to the conic bundle $\pi_W\colon W\to\mathbb{A}^1_\R$ from Lemma~\ref{lemma:first-form}
to birationally transform it into a~$G$-minimal conic bundle which is also explicitly given in $\p^{2}_\R\times\mathbb{A}^1_\R$.

\begin{lemma}
\label{lemma:G-MMP-form}
Let $Z$ be a smooth surface in $\p^{2}_\R\times\mathbb{A}^1_\R$ that is given by
$$
A(t)x^2+B(t)xy+C(t)y^2=H(t)z^2
$$
for some polynomials $A,B,C,H\in\R[t]$,
let $\pi_Z\colon Z\to \mathbb{A}^1_\R$ be the morphism induced by the~natural projection $\p^{2}_\R\times\mathbb{A}^1_\R\to \mathbb{A}^1_\R$,
where $([x:y:z],t)$ are coordinates on $\p^{2}_\R\times\mathbb{A}^1_\R$.
Let us fix~a~$G$-action on the~surface $Z$ that is given by $([x:y:z],t)\mapsto([x:y:-z],t)$.
Then there exists  a~$G$-equivariant commutative diagram
$$
\xymatrix{
Z\ar@{->}[d]_{\pi_Z}\ar@{-->}[rr]&&\widetilde{Z}\ar@{->}[d]^{\pi_{\widetilde{Z}}}\\
\mathbb{A}^1_\R\ar@{=}[rr]&&\mathbb{A}^1_\R}
$$
such that $Z\dashrightarrow\widetilde{Z}$ is a birational map,
and $\widetilde{Z}$ is a smooth surface in $\p^{2}_\R\times\mathbb{A}^1_\R$ given by
$$
\widetilde{A}(t)x^2+\widetilde{B}(t)xy+\widetilde{C}(t)y^2=\widetilde{H}(t)z^2
$$
for some  $\widetilde{A},\widetilde{B},\widetilde{C},\widetilde{H}\in\R[t]$
such that the~polynomial $\widetilde{H}(t)$ has only real roots,
the fibre of the~morphism $\pi_{\widetilde{Z}}$ over every root of $\widetilde{H}(t)$ is a singular irreducible conic,
and the~$G$-action on the~surface $\widetilde{Z}$ is given by $([x:y:z],t)\mapsto([x:y:-z],t)$.
\end{lemma}

\begin{proof}
First, we observe that the smoothness of the surface $Z$ is equivalent to the condition that the~polynomial $(4AC-B^2)H$ does not have multiple roots.

Second, we suppose that $H(t)$ has a real root such that the fibre of $\pi_Z$ over this root is a~conic with two real components.
For simplicity, we may assume that this root is $t=0$.
Write  $A(t)=A_0+tA_1(t)$, $B(t)=B_0+tB_1(t)$, $C(t)=C_0+tC_1(t)$,
where $A_0,B_0,C_0\in\R$. Then $\pi_Z^{-1}(0)$ is given by $A_0x^2+B_0xy+C_0y^2=t=0$.
By our assumption, we have
$$
A_0x^2+B_0xy+C_0y^2=(cx-ay)(dx+by)
$$
for some real numbers $a,b,c,d$ such that the~polynomials $cx-ay$ and $dx+by$ are co-prime.
Let $u=t(cx-ay)$, $v=dx+by$, $\lambda=ad+bc$, so that
$$
\left\{\aligned
&x=\frac{1}{\lambda t}\left(bu+tav\right),\\
&y=\frac{1}{\lambda t}\left(ctv-du\right).
\endaligned
\right.
$$
Substituting this into the defining equation of $Z$, we see get the equation
\begin{multline*}
t^2\overline{H}(t)z^2=t\left(A_0x^2+B_0xy+C_0y^2\right)+t^2\left(A_1(t)x^2+B_1(t)xy+C_1(t)y^2\right)=\\
=uv\left(1+\frac{t}{\lambda^2}\left( 2abA_1(t)-2cdC_1(t)+(bc-ad)B_1(t)\right)\right)+\\
+u^2\frac{1}{\lambda^2}\left(b^2A_1(t)+d^2C_1(t)-bdB_1(t)\right)+v^2\frac{t^2}{\lambda^2}\left(a^2A_1(t)+c^2C_1(t)+acB_1(t)\right),
\end{multline*}
where $\overline{H}\in\R[t]$ such that $H=t\overline{H}(t)$.
Let $\alpha\colon\p^2_\R\times\mathbb{A}^1_\R\dashrightarrow\p^2_\R\times\mathbb{A}^1_\R$  be the map given by
$$
([x:y:z],t)\mapsto([t(cx-ay):dx+by:tz],t).
$$
Then the map $\alpha$ induces a $G$-equivariant birational map $Z\dasharrow\overline{Z}$
such that $\overline{Z}$ is a surface in $\p^2_\R\times\mathbb{A}^1_\R$ that is given by
$$
\overline{A}(t)x^2+\overline{B}(t)xy+\overline{C}(t)y^2=\overline{H}(t)z^2
$$
for some $\overline{A},\overline{B},\overline{C}\in\R[t]$ such that
the polynomial $(4\overline{A}(t)\overline{C}(t)-\overline{B}(t)^2)\overline{H}(t)$ is a positive scalar multiple of the~polynomial $(4A(t)C(t)-B(t)^2)\overline{H}(t)$.
So, we see that $\overline{Z}$ is~smooth.
Therefore, iterating this process, we may assume that the fibres of our original conic bundle $\pi_Z\colon Z\to\mathbb{A}^1_{\mathbb{R}}$ over every real root of the polynomial $H(t)$ is an irreducible conic.

Finally, suppose that $H(t)$ has a non-real root $t=\varepsilon\in\mathbb{C}$.
Then the~fibre of the~conic bundle $\pi_Z$ over this root is given by $t-\varepsilon=A(\varepsilon)x^2+B(\varepsilon)xy+C(\varepsilon)y^2=0$,
and
$$
A(\varepsilon)x^2+B(\varepsilon)xy+C(\varepsilon)y^2=(cx-ay)(dx+by)
$$
for $a,b,c,d\in\C$ such that $cx-ay$ and $dx+by$ are co-prime, since $\pi_Z$ has reduced fibres.
Let $\alpha_{\varepsilon}\colon\p^2_\R\times\mathbb{A}^1_\R\dashrightarrow\p^2_\R\times\mathbb{A}^1_\R$
and  $\alpha_{\overline{\varepsilon}}\colon\p^2_\R\times\mathbb{A}^1_\R\dashrightarrow\p^2_\R\times\mathbb{A}^1_\R$
be rational maps given by
$$
([x:y:z],t)\mapsto([(t-\varepsilon)(cx-ay):dx+by:(t-\varepsilon)z],t)
$$
and
$$
([x:y:z],t)\mapsto([(t-\bar{\varepsilon})(\bar{c}x-\bar{a}y):\bar{d}x+\bar{b}y:(t-\bar{\varepsilon})z],t),
$$
respectively. Then the composition $\alpha_{\bar{\epsilon}}\circ\alpha_\epsilon$ is a rational map defined over $\mathbb{R}$
that induces a $G$-equivariant birational map $Z\dasharrow\widehat{Z}$
such that $\widehat{Z}$ is a surface in $\p^2_\R\times\mathbb{A}^1_\R$  given by
$$
\widehat{A}(t)x^2+\widehat{B}(t)xy+\widehat{C}(t)y^2=\widehat{H}(t)z^2
$$
for polynomials $\widehat{A}(t),\widehat{B}(t),\widehat{C}(t),\widehat{H}(t)\in\R[t]$ such that
$$
\widehat{H}(t)=\frac{H(t)}{(t-\epsilon)(t-\bar{\epsilon})},
$$
and $(4\widehat{A}(t)\widehat{C}(t)-\widehat{B}(t)^2)\widehat{H}(t)=\xi (4A(t)C(t)-B(t)^2)\widehat{H}(t)$ for some positive real number ~$\xi$.
In particular, we see that the~surface $\widehat{Z}$ is~smooth.
Now, iterating this process, we obtain the required $G$-equivariant birational transformation $Z\dasharrow\widetilde{Z}$.
\end{proof}

Now, we are ready to finish the proof of Theorem~\ref{theorem:main-non-exceptional}.

\begin{proof}[Proof of Theorem~\ref{theorem:main-non-exceptional}]
Recall that $U=S\setminus\pi^{-1}([1:0])$, and the fibre $\pi^{-1}([1:0])$ is smooth.
Applying Lemmas~\ref{lemma:first-form} and \ref{lemma:G-MMP-form}, we get  $G$-equivariant commutative diagram
$$
\xymatrix{
U\ar@{->}[d]_{\pi_U}\ar@{-->}[rr]&&Y\ar@{->}[d]^{\pi_Y}\\
\mathbb{A}^1_\R\ar@{=}[rr]&&\mathbb{A}^1_\R}
$$
such that $Y$ is a smooth surface $\{A(t)x^2+B(t)xy+C(t)y^2-H(t)z^2=0\}\subset\mathbb{P}^2_\mathbb{R}\times\mathbb{A}^1_{\mathbb{R}}$,
where $A,B,C,H\in\mathbb{R}[t]$, the group $G$ acts on the surface $Y$ by
$$
([x:y:z],t)\mapsto([x:y:-z],t),
$$
and $\pi_Y$ is given by $([x:y:z],t)\mapsto t$, where $([x:y:z],t)$ are coordinates on $\mathbb{P}^2_\mathbb{R}\times\mathbb{A}^1_{\mathbb{R}}$.

Since $Y$ is smooth, we see that $(B^2-4AC)H$ does not have multiple roots. It follows from Lemmas~\ref{lemma:first-form} and \ref{lemma:G-MMP-form} that we may further assume that
$H(t)$ has only real roots, and the fibre of $\pi_Y$ over every root of the polynomial $H(t)$ is a singular irreducible conic.
Then the conic bundle $\pi_Y\colon Y\to\mathbb{A}^1_\R$ is relatively $G$-minimal over $\mathbb{A}^1_\R$.

Multiplying the defining equation of the surface $Y$ by $\pm 1$, we may further assume that the leading coefficient of the polynomial $H(t)$ is negative.

Now, there exists a completion $Y\subset X$ such that $X$ is a projective smooth surface,
the~action of the group $G=\langle\tau\rangle$ on the surface $Y$ extends to its regular action on $X$,
the~morphism $\pi_Y\colon Y\to\mathbb{A}^1_{\R}$ induces a $G$-equivariant morphism $\eta\colon X\to\mathbb{P}^1$ that
maps the~complement $X\setminus Y$ to the point $[1:0]$ and fits the $G$-equivariant commutative diagram:
$$
\xymatrix{
S\ar@{->}[d]_{\pi}\ar@{-->}[rr]^{\chi}&&X\ar@{->}[d]^{\eta}\\
\p^1_\R\ar@{=}[rr]&&\mathbb{P}^1_\R}
$$
where $\chi$ is a composition of the constructed $G$-equivariant birational map $U\dasharrow Y$ with the $G$-equivariant embedding $Y\subset X$.

Moreover, applying relative $G$-Minimal Model Program to $X$ over $\mathbb{P}^1_\R$, we may further assume that $\eta\colon X\to\mathbb{P}^1_\R$ is a $G$-minimal conic bundle. 
Since $\pi^{-1}([1:0])$ is smooth and has empty real locus by assumption, we conclude that the fibre $\eta^{-1}([1:0])=X\setminus Y$ is also smooth and has empty real locus.
Note that $\eta(X(\R))=\pi(S(\R))$, so that $\eta(X(\R))$ is a union of intervals in $\mathbb{P}^1_\R(\R)$,
because the conic bundle $\pi\colon S\to\mathbb{P}^1_\R$ is assumed to be non-exceptional.

To complete the proof of Theorem~\ref{theorem:main-non-exceptional}, it remains to explain why $\deg(B^2-4AC)$ is~even.
Let $\mathscr{C}$ be the curve in $X$ that is fixed by $\tau$.
Then $\eta$ induces the double cover  $\mathscr{C}\to\mathbb{P}^1_\R$, which is not ramified over the~point $[1:0]$,
because the fibre $\eta^{-1}([1:0])$ is smooth.
On~the~other hand, the induced double cover $\mathscr{C}\setminus(\mathscr{C}\cap\eta^{-1}([1:0]))\to\mathbb{A}^1_\R$
is branched exactly at the roots of the polynomial $B^2-4AC$, which does not have multiple roots.
Therefore, since~the~number of ramification points~of the~double cover $\mathscr{C}\to\mathbb{P}^1_\R$ is even,
we conclude that $\deg(B^2-4AC)$ is also even. Theorem~\ref{theorem:main-non-exceptional} is proved.
\end{proof}

\subsection{Proof of Theorem~\ref{theorem:main-non-exceptional-equivalence}}
\label{subsection:non-exceptional-birational-classification}
Let us use all assumptions and notations of Theorem~\ref{theorem:main-non-exceptional-equivalence}.
Set $\pi_{Y_1}=\eta_{1}\vert_{Y_1}$ and $\pi_{Y_2}=\eta_{1}\vert_{Y_2}$.
Let $\Delta_1=4A_1C_1-B_1^2$ and $\Delta_2=4A_2C_2-B_2^2$.
Then, applying Corollary~\ref{corollary:main-non-exceptional} twice,
we obtain two  $G$-equivariant commutative diagrams
$$
\xymatrix{
Y_1\ar@{->}[d]_{\pi_{Y_1}}\ar@{-->}[rr]^{\chi_1}&&\widehat{Y}_1\ar@{->}[d]^{\pi_{\widehat{Y}_1}}\\
\mathbb{A}^1_\R\ar@{=}[rr]&&\mathbb{A}^1_\R}
$$
and
$$
\xymatrix{
Y_2\ar@{->}[d]_{\pi_{Y_2}}\ar@{-->}[rr]^{\chi_2}&&\widehat{Y}_2\ar@{->}[d]^{\pi_{\widehat{Y}_2}}\\
\mathbb{A}^1_\R\ar@{=}[rr]&&\mathbb{A}^1_\R}
$$
such that $\widehat{Y}_1$ and $\widehat{Y}_2$ are (possibly singular) surfaces in $\mathbb{P}^{2}_\mathbb{R}\times\mathbb{A}^{1}_{\mathbb{R}}$
given by $\widehat{A}_1x^2+\widehat{A}_1\Delta_1y^2=H_1z^2$ and $\widehat{A}_2x^2+\widehat{A}_2\Delta_2y^2=H_2z^2$, respectively,
where $\widehat{A}_1$ and $\widehat{A}_2$ are polynomials in $\R[t]$ such that
$\widehat{A}_1$, $\widehat{A}_2$, $\widehat{A}_1\Delta_1$, $\widehat{A}_2\Delta_2$ do not have multiple roots,
$\widehat{A}_1$ and $H_1$ are co-prime,
$\widehat{A}_2$ and $H_2$ are co-prime,
$\widehat{A}_1\Delta_1$ and $H_1$ are co-prime,
$\widehat{A}_2\Delta_2$ and $H_2$ are co-prime, both conic bundles $\pi_{\widehat{Y}_1}$ and $\pi_{\widehat{Y}_2}$ are given by $([x:y:z],t)\mapsto t$,
the~$G$-actions on both surfaces $\widehat{Y}_1$ and $\widehat{Y}_2$ are given by $([x:y:z],t)\mapsto([x:y:-z],t)$,
both $\chi_1$ and $\chi_2$ are birational maps
that are biregular along singular fibres of the conic bundles $\pi_{Y_1}$ and $\pi_{Y_2}$, respectively.

First, we suppose that there exists a $G$-equivariant birational map $\rho\colon X_1\dasharrow X_2$
such that $\rho$ fits the~commutative diagram~\eqref{equation:main-non-exceptional-equivalence}.
Then $\eta_1(X_1(\R))=\eta_2(X_2(\R))$, since $\rho$ is an isomorphism away from finitely many fibres of $\eta_1$ and $\eta_2$.
Moreover, the map $\rho$ induces an isomorphism between the $G$-fixed curves in $X_1$ and $X_2$.
Hence, using  Corollary~\ref{corollary:fixed-curve-non-exceptional}, we get
$
\Delta_1=\lambda\Delta_2
$
for some $\lambda\in\R^\ast$.
Furthermore, the map $\rho$ is an isomorphism along singular fibres of  the conic bundles $\eta_1$ and $\eta_2$, because it is a composition of elementary transformations. 
Thus, in particular, $\rho$ maps special singular fibres of $\eta_1$ into special singular fibres of $\eta_2$,
which implies that
$
H_1=\mu H_2
$
for some $\mu\in\R^\ast$.
Since, by assumption, the leading coefficients of $H_1$ and $H_2$ are negative, we see that~$\mu>0$.
This gives $\lambda>0$.
Indeed, let $t_0$ be a general point in $\mathbb{A}^1_\R(\R)$ such that $t_0\not\in\eta_1(X_1(\R))$.
Since $\rho_1$ is an isomorphism along the~fibre $\pi_{Y_1}^{-1}(t_0)$,
we get $\Delta_1(t_0)>0$.
Similarly, we  get $\Delta_2(t_0)>0$, which implies that $\lambda>0$.
This proves one direction of Theorem~\ref{theorem:main-non-exceptional-equivalence}.

To prove the other direction of Theorem~\ref{theorem:main-non-exceptional-equivalence}, we suppose that  $\eta_1(X_1(\R))=\eta_2(X_2(\R))$,
and there are positive real numbers $\lambda$ and $\mu$ such that
$\Delta_1=\lambda\Delta_2$ and $H_1=\mu H_2$.
Let us show that there exists a $G$-equivariant birational map $\rho\colon X_1\dasharrow X_2$
that fits~\eqref{equation:main-non-exceptional-equivalence}.
By Lemma~\ref{lemma:equivalence-of-forms}, to prove this,
it is enough to show that the forms $\langle\widehat{A}_1, \widehat{A}_1\Delta_1\rangle$ and $\langle\widehat{A}_2,\widehat{A}_2\Delta_2\rangle$
are equivalent over $\R(t)$,
which can be shown by using Proposition~\ref{proposition:equivalence-of-forms}.

Let us check that all conditions of Proposition~\ref{proposition:equivalence-of-forms} applied to
the forms $\langle\widehat{A}_1, \widehat{A}_1\Delta_1\rangle$ and $\langle\widehat{A}_2,\widehat{A}_2\Delta_2\rangle$ are satisfied.
First, we observe that the condition (1) of Proposition~\ref{proposition:equivalence-of-forms} is satisfied, as $\Delta_1=\lambda\Delta_2$ and $\lambda>0$.

Now the option (2.a) is clearly not possible,
so assume that $\widehat{A}_1(\varepsilon)=(\widehat{A}_1\Delta_1)(\varepsilon)=0$.
Then $\Delta_1(\varepsilon)\ne 0$, because $\widehat{A}_1\Delta_1$ have no multiple roots.
Thus, we see that
$$
(\widehat{A}_1)_{\varepsilon}(\varepsilon)(\widehat{A}_1\Delta_1)_{\varepsilon}(\varepsilon)=\Big((\widehat{A}_1)_\varepsilon(\varepsilon)\Big)^2\Delta_1(\varepsilon)
$$
which is negative, since $\Delta_1(\varepsilon)=-B_1(\varepsilon)^2$. This shows that (2.b) holds.

To check (2.c), we use exactly the same argument to show that
$(\widehat{A}_1)_{\varepsilon}(\varepsilon)(\widehat{A}_1\Delta_1)_{\varepsilon}(\varepsilon)<0$
and $(\widehat{A}_2)_{\varepsilon}(\varepsilon)(\widehat{A}_2\Delta_2)_{\varepsilon}(\varepsilon)<0$.

We now check (3).
Let us denote the leading coefficients of $\widehat{A}_i$, $\widehat{A}_i\Delta_i$ and $H_i$ by $\alpha_i$, $\beta_i$ and $\gamma_i$, respectively. Assume $\alpha_i\beta_i<0$. Recall that $\pi(Y_i(\R))$ is a finite union of closed intervals bounded by the roots of $\Delta_1H_1$; let $\varepsilon_{\min},\varepsilon_{\max}\in\R$ be the smallest and the largest roots of $\Delta_1H_1$, respectively. But then for $t\gg\varepsilon_{\max}$ the fibre over $t$ is a conic with non-empty real locus (since the coefficients at $x^2$ and $y^2$ are of different signs), which is not possible. Therefore, both $\alpha_i$ and $\beta_i$ must be positive (recall that we assume $\gamma_i<0$). 
We now claim that $\deg \widehat{A}_1$ and $\deg \widehat{A}_2$ have the same parity, which will finish the proof, because the parities of $\deg \widehat{A}_i$ and $\deg \widehat{A}_i\Delta_i$ are the same (as $\deg \Delta_i$ is even). Assume the contrary, i.e. $\deg \widehat{A}_1=2n$, $\deg \widehat{A}_2=2m+1$. For $t\ll\varepsilon_{\min}$ the real loci of the fibres $\pi_1^{-1}(t)$ and $\pi_2^{-1}(t)$ are diffeomorphic to the real loci of the conics
$$
\alpha_1t^{2n}x^2+\beta_1 t^{2n+\deg\Delta_1}y^2=\gamma_1t^{\SF_{S_1}}z^2\ \ \text{and}\ \ \ \alpha_2t^{2m+1}x^2+\beta_2 t^{2m+1+\deg\Delta_2}y^2=\gamma_2t^{\SF_{S_2}}z^2,
$$
respectively. But since $\alpha_i$, $\beta_i$ are positive, $\gamma_i$ are negative, $\SF_{S_1}=\SF_{S_2}=k$ and $\deg\Delta_i$ is even, these conics are respectively isomorphic for any $t\ll\varepsilon_{\min}$ over $\R$ to
$$
x^2+y^2=(-1)^kz^2\ \ \text{and}\ \ -x^2-y^2=(-1)^kz^2,
$$
respectively.
But for any $k$ exactly one of these conics has non-empty real locus, which contradicts the condition $\eta_1(X_1(\R))=\eta_2(X_2(\R))$.
We conclude that $\deg \widehat{A}_1$ and $\deg \widehat{A}_2$ have the same parity.

Therefore, we can apply Proposition~\ref{proposition:equivalence-of-forms}  to
show that the forms $\langle\widehat{A}_1, \widehat{A}_1\Delta_1\rangle$
and $\langle\widehat{A}_2,\widehat{A}_2\Delta_2\rangle$ are equivalent over $\R(t)$.
This completes the proof of  Theorem~\ref{theorem:main-non-exceptional-equivalence}.

\section{Classification}
\label{section:classification}

\subsection{Proof of Main Theorem: classification}
\label{subsection:main-theorem}

Let $\iota$ be an involution in the group $\Bir(\p_\R^2)$. First, regularising the action of $\langle\iota\rangle$, we may assume that (after conjugation with a birational map $\p^2_\R\dashrightarrow S$) the involution $\iota$ is given by a biregular involution $\tau\in\Aut(S)$ of a smooth projective surface $S$, see Section~\ref{subsection:regularisation}. Let $G=\langle\tau\rangle$.	Then, applying $G$-equivariant Minimal Model Program, we may further assume that

\begin{enumerate}
	\item either $\Pic(S)^G\simeq\Z$, and $S$ is a real $\R$-rational~del Pezzo surface;
	\item or $\Pic(S)^G\simeq\Z^2$, and there exists a~$G$-equivariant conic bundle $\pi\colon S\to\p^1_\R$, and $S$ is again $\R$-rational.
\end{enumerate}

\begin{lemma}
\label{lemma:dP-Pic-Z-1-2-4-8-9}
Suppose that  $\Pic(S)^G\simeq\mathbb{Z}$. Then $K_S^2\in\{1,2,4,8,9\}$
\end{lemma}

\begin{proof}
Let us consider $S$ as complex surface. Denote by $\sigma$ the antiholomorphic involution that generates $\Gal(\C/\R)$.
Assume that $K_S^2\notin\{8,9\}$. Let $L$ be a~$(-1)$-curve in $S$ defined over $\mathbb{C}$, let $C$ be the $\langle\sigma,\tau\rangle$-orbit of $L$, and let $k$ be the number of irreducible components of the curve $C$.
Then $C\sim n(-K_S)$  for some $n\in\mathbb{Z}_{>0}$. Then $k=-K_S\cdot C=nK_S^2$.

In particular, if $K_S^2=7$, then $k=7n$; but $k\leqslant 3$, because $S$ contains three $(-1)$-curves defined over $\mathbb{C}$.
Similarly, if $K_S^2=6$, then $S$ contains six $(-1)$-curves, so $6\geqslant k=6n$ gives $k=6$,
but $C$ cannot consists of $6$ irreducible components, because the group $\langle \sigma,\tau\rangle\simeq(\mathbb{Z}/2)^2$
cannot transitively permute $6$ objects.
Likewise, if $K_S^2=5$, then $S$ contains ten $(-1)$-curves defined over $\mathbb{C}$, so that $10\geqslant k=5n$ gives $k\in\{5,10\}$,
which leads to a~contradiction as above. 
If $K_S^2=3$, then $S$ is a smooth cubic surface in $\mathbb{P}^3_{\R}$,
and $S$ contains twenty seven $(-1)$-curves defined over $\mathbb{C}$.
which implies that one of them must be $\langle\sigma,\tau\rangle$-invariant, so $\rk\Pic(S)^G>1$.
\end{proof}

Recall that case $K_S^2=8$ was already discussed in Section \ref{section: dP8}. Hence we may continue with smaller values of $K_S^2$.

\begin{lemma}\label{lem:K-2-larger-4}
If $K_S^2\geqslant 5$ and $K_S^2\ne 8$, then $\iota$ is conjugate to a linear involution of $\p_\R^2$. 
\end{lemma}

\begin{proof}
If $\Pic(S)^G\simeq\Z$, then $K_S^2=9$ by Lemma~\ref{lemma:dP-Pic-Z-1-2-4-8-9} and we are done. Thus, we may assume that $\Pic(S)^G\simeq\Z^2$,
and  there exists a~$G$-equivariant conic bundle $\pi\colon S\to\p^1_\R$. Furthermore, we may assume that $S$ is $G$-minimal (i.e. any $G$-equivariant birational morphism $S\to S'$ to a smooth projective $G$-surface is an isomorphism), since otherwise we reduce to the case $\Pic(S)^G\simeq\Z$. Then $K_S^2\ne 7$ by \cite[Theorem~5]{Iskovskikh80}.
If $K_S^2\in\{5,6\}$, then it follows from the proof of \cite[Theorem~5]{Iskovskikh80}
that $S$ is a del Pezzo surface. 
If $K_S^2=6$, then $\pi$ has a pair of non-real conjugate disjoint $(-1)$-sections, so $S$ is not $G$-minimal. If $K_S^2=5$, the $S$ is the blow-up of $\p^2_\R$ in four real points or in two real points and a pair of non-real conjugate points or of two pairs of non-real conjugate points. In each case $S$ is not $G$-minimal.
\end{proof}

Now we are going to discuss the case $K_S^2=4$. 

\begin{remark}\label{rmk:dP4}
Suppose that $S$ is a del Pezzo surface of degree $K_S^2=4$ and consider its complexification $S_{\C}$. 
Then $S_{\C}$ is a blow-up of $\p^2_\C$ in five points in general position. 
So, $\Pic(S_{\C})\simeq\Z^{6}$ has a basis $e_0,\ e_1,\ldots,e_5$, where $e_0$ is the pull-back of the class of a line on $\p_\C^2$, and $e_i$ are the classes of exceptional curves. Then
$
\{s\in\Pic(S_{\C}):\ s^2=-2,\ s\cdot K_{S_{\C}}=0 \}
$
is a root system in the orthogonal complement to $K_{S_{\C}}^{\bot}\subset\Pic(S_{\C})\otimes\R$ and one can associate it with the Weyl group $\mathcal{W}(\mathrm{D}_5)\simeq(\Z/2)^4\rtimes \Sym_5$. Furthermore, there are natural homomorphisms 
\[
\rho\colon\Aut(S_{\C}) \hookrightarrow \mathcal{W}(\mathrm{D}_5),\ \ {\rm Gal}(\C/\R)\rightarrow \mathcal{W}(\mathrm{D}_5),
\]
where $\rho$ is an injection. We denote by $\alpha^*$ the image of $\alpha\in\Gal(\C/\R)\times \Aut(S_{\C})$ in the corresponding Weyl group.
The conjugacy classes in the Weyl groups are indexed by Carter graphs, named e.g. $A_1$, $A_1^2$, etc. A Carter graph determines the characteristic polynomial of an element from a given class, its eigenvalues and trace on $K_{S_{\C}}^{\bot}$. We refer to \cite[6.1]{DolgachevIskovskikh} for more details.
Moreover, the surface $S_{\C}$ is isomorphic to an intersection of two quadrics
\[
\sum_{i=1}^{5}x_i^2=\sum_{i=1}^{5}\lambda_ix_i^2=0
\]
in $\p_\C^4$ for some $\lambda_1,\dots,\lambda_5\in\C$. The group $\Aut(S_{\C})$ injects into the Weyl group $\mathcal{W}(\mathrm{D}_5)$ and $\Aut(S_{\C})$ always contains the subgroup $(\Z/2)^4$, which acts as a diagonal subgroup of $\PGL_5(\C)$. There are two types of involutions in this group, $\iota_{ij}$ and $\iota_{ijkl}$, which switch the signs of $x_i,x_j$ and $x_i,x_j,x_k,x_l$, respectively. 
\end{remark}

\begin{lemma}\label{lem: KS2=4}
Let $S$ be a del Pezzo surface.
	Suppose that $K_S^2=4$ and $\Pic(S)^G\simeq\Z$. Then one of the following two possibilities holds:
	\begin{enumerate}
		\item\label{dp4 elliptic}  $S^\tau=F(\tau)$ is a smooth genus 1 curve with non-empty real locus, $S/G\simeq Q_{3,1}$ and
		\begin{enumerate}
			\item either $S(\R)\approx\Sph^2$;
			\item or  $S(\R)\approx\Sph^1\times\Sph^1$;
		\end{enumerate}	
		\item\label{dp4 empty}  $F(\tau)=\varnothing$, $S^\tau$ consists of four points, and at least two of them are real. 
	\end{enumerate}
\end{lemma}
\begin{proof}
	The map given by the~anticanonical linear system $|-K_S|$ embeds the surface $S$ into $\p^4_\R$ as a~complete intersection of two real quadrics. The involution $\tau$ extends linearly to $\p^4_\R$ and hence $S^{\tau}$ is a linear section, which is moreover smooth (Lemma~\ref{lemma:fixed-curves-smooth}). Thus $S^\tau$ is either a smooth genus 1 curve in $|-K_S|$, or four points (possibly complex conjugate), see e.g. \cite[\S~9]{BlancLinearisation}. Let $\eta: S\to Z=S/\tau$ be the quotient map. Then  $\Pic(Z)\simeq\Z$. 
		
	(\ref{dp4 elliptic}) Assume $S^\tau$ is a smooth genus 1 curve. Then $Z$ is smooth and by the Hurwitz formula one has $K_Z^2=\frac{1}{2}(2K_S)^2$, so	$Z\simeq Q_{3,1}$ since $\rho(Z)=1$ and the quotient map $S\to S/G$ is branched over a~divisor $E$ of bidegree $(2,2)$ in $Z_\C\simeq\p_\C^1\times\p_\C^1$, which is a~complete intersection of the quadric~$Z$ with another quadric surface in $\p^3_{\R}$. Note that $E$ is a real genus 1 curve, so $E(\R)\in\{ \varnothing,\Sph^1,\Sph^1\sqcup\Sph^1\}$. In fact, the case $E(\R)=\varnothing$ is impossible, as $S(\R)$ is then disconnected, contradicting the rationality of $S$. The other two possibilities give $S(\R)\approx\Sph^2$ and $S(\R)\approx\Sph^1\times\Sph^1$, respectively.
	
	(\ref{dp4 empty}) To complete the proof, we have to exclude the case when $S^\tau$ consists of two pairs of complex conjugate points. Assume this is the case. Then $Z$ is a singular del Pezzo surface of degree $K_Z^2=2$, which has two pairs of complex conjugate $A_1$ points. Let $\widetilde{Z}$ be the minimal resolution of $Z$. Let us consider two cases: when $\widetilde{Z}$ is $\R$-rational and when it is not.
		
	Assume that $\widetilde{Z}$ is $\R$-rational. Run the minimal model program over $\R$ on $\widetilde{Z}$. We get a birational morphism $\widetilde{Z}\to\overline{Z}$ such that $\overline{Z}$ is a smooth $\R$-rational and $\R$-minimal surface. By Theorem \ref{thm: IskovskikhCrit}, one has $K_{\overline{Z}}^2\geqslant 5$. Since $\rk\Pic(\widetilde{Z})=3$ and $K_{\widetilde{Z}}^2=2$, we get $\rk\Pic(\overline{Z})=1$ and $K_{\overline{Z}}^2\in\{ 5,6\}$. In particular, $\overline{Z}$ is del Pezzo. However, there is no $\R$-rational $\R$-minimal del Pezzo surface of degree $5$ or $6$ \cite[Corollary 2.3]{Russo}. 
		
	Thus, we see that $\widetilde{Z}$ is not $\R$-rational. We will show that this is not possible. Indeed, assume the opposite holds. Then by \cite[Proposition 5.1]{Tre18}, our involution $\tau$ is conjugate over $\C$ to the involution $\iota_{12}$ in the notation of Remark~\ref{rmk:dP4}. 
	The centraliser of this involution in the group $\mathcal{W}(\mathrm{D}_5)$ is $(\Z/2)^4 \rtimes \langle (12), (345), (45)\rangle$, and the generator $\sigma$ of $\Gal(\C/\R)$ maps to this group. Conjugacy classes of all elements in $\mathcal{W}(\mathrm{D}_5)$ are listed in \cite[Table 4]{Tre20} or \cite[Table 3]{DolgachevIskovskikh}; we use this information below.
	
	Recall that $K_S^\bot\subset\Pic(S_\C)$ is a 5-dimensional lattice. The induced involution $\tau^*$ is of type $A_1^2$, with two eigenvalues $-1$. Note that $\sigma^*\ne{\rm id}$, as we cannot have $\Pic(S_\C)^{\langle\tau\rangle}\simeq\Z$, e.g. by \cite[Theorem 6.9]{DolgachevIskovskikh}. The involution $\sigma^*$ cannot be of type $A_1^4$, because $S$ would be not $\R$-rational (see e.g. \cite[\S 4]{Wa87}). So, at most three eigenvalues of $\sigma^*$ are $-1$. Moreover, the product of $\tau^*$ and $\sigma^*$ has at least one eigenvalue equal to 1, as there are no elements of type $A_1^5$ in $\mathcal{W}(\mathrm{D}_5)$. Then $K_S^\bot$ contains a vector that is an eigenvector for both $\sigma^*$ and $\tau^*$, which implies $\rk\Pic(S)^G>1$, a contradiction.
\end{proof}

\noindent{\bf Summary on del Pezzo case:} If $\Pic(S)^G\simeq\Z$ and $K_S^2\in\{1,2\}$, it follows from Proposition~\ref{proposition:Bertini} and Proposition~\ref{proposition:Geiser} that  $\tau$ belongs to the classes $\mathfrak{B}_4$, $\mathfrak{G}_3$ or $\mathfrak{K}_1$. If $K_S^2=8$ then it follows from Proposition~\ref{prop: K2=8} that $\tau$ belongs either to the class $\mathfrak{L}$ or $\mathfrak{Q}$. If $K_S^2\geqslant 5$ and $K_S^2\ne 8$ then $\tau$ belongs to $\mathfrak{L}$ by Lemma~\ref{lem:K-2-larger-4}. Moreover, if $\Pic(S)^G\simeq\Z$ and $K_S^2=4$, there is a $G$-fixed real point by Lemma \ref{lem: KS2=4}. Blowing up this point, we obtain a $G$-equivariant birational map from $S$ to a $G$-minimal conic bundle. 

\medskip

Now, using Lemma \ref{lemma:dP-Pic-Z-1-2-4-8-9}, to complete the proof of the classification part of the main theorem, we may assume that $K_S^2\leqslant 4$, $\Pic(S)^G\simeq\Z^2$ and there exists a~$G$-equivariant conic bundle $\pi\colon S\to\p^1_\R$. In this case, we have

\begin{lemma}\label{lemma: Trepalin}
Suppose that $F(\tau)=\varnothing$. Then $\tau$ is contained in one of the following classes: $\mathfrak{T}_{4n}$, $\mathfrak{T}_{4n+2}'$, or $\mathfrak{T}_{4n}''$ for some $n\geqslant 1$.
\end{lemma}
\begin{proof}
If $\tau$ acts non-trivially on the base of the conic bundle $\pi$, then $\tau$ is conjugate to a Trepalin involution by Theorem \ref{thm: Trepalin involutions}. Let $\tau$ act trivially on the base of $\pi$. 
Since $K_S^2\leqslant4$ and $F(\tau)=\varnothing$, Lemma~\ref{lemma:exceptional} implies that $\pi\colon S\to \p^1$ is not exceptional. 
By Section~\ref{subsection:Iskovskikh-involutions}, we have $4\geqslant K_S^2=6-\delta_S$ and $\delta_S\in\{0,1,2\}$ since $S$ is $\R$-rational. 
Then $K_S^2=4$ and 
there exist $2$ special fibres of $\pi_S$. By Theorem~\ref{theorem:main-non-exceptional}, $S^{\tau}$ is a smooth rational curve that is the union of a double section of $\pi_S$ and two points (the singular points of the special fibres). 
Let us show that $\tau\in\mathfrak{T}_4$. 
 
Using Theorem~\ref{theorem:main-non-exceptional}, we may assume that the fibre $\pi_S^{-1}([1:0])$ is smooth and does not have real points,
the quasi-projective surface $Y=S\setminus \pi_S^{-1}([1:0])$ is given in $\mathbb{P}^2_\mathbb{R}\times\mathbb{A}^1_{\mathbb{R}}$ by
\[
A(t)x^2+B(t)xy+C(t)y^2+(t-a)(t-b)z^2=0
\]
for some polynomials $A,B,C\in\mathbb{R}[t]$ such that $(B^2-4AC)(t-a)(t-b)$ does not have multiple roots and $\deg(B^2-4AC)$ is two, the fibres of $\pi_S$ over $t=a$ and $t=b$ are singular irreducible conics and the involution $\tau$ acts on $Y$ by
$
([x:y:z],t)\mapsto([x:y:-z],t),
$
and the restriction map $\pi_S\vert_{Y}\colon Y\to\mathbb{P}^1_\R\setminus [1:0]=\mathbb{A}^1_{\R}$ is the map given by $([x:y:z],t)\mapsto t$,
where $([x:y:z],t)$ are coordinates on $\mathbb{P}^2_\mathbb{R}\times\mathbb{A}^1_{\mathbb{R}}$.
After an affine change of the coordinate $t$, we may assume that $\Delta=4AC-B^2$ is one of the polynomials
$t^2+1$, $-(t^2+1)$, $t^2-1$, or $-(t^2-1)$.

Since the fibres of $\pi_S$ over $t=a$ and $t=b$ are singular irreducible conics, we see that $\Delta(a)>0$ and $\Delta(b)>0$. This rules out the case $\Delta(t)=-(t^2+1)$. 
Moreover, if $\Delta(t)=-(t^2-1)$, then the real locus of  $\pi_S^{-1}([1:0])$ is non-empty, which contradicts the assumptions above. 
Hence $\Delta(t)=t^2-1$ or $\Delta(t)=t^2+1$ and in the first case, we have $a,b\notin[-1,1]$.
Recall that $\pi_S(S(\R))$ is a single interval, so $\pi_S(S(\R))=[a,b]$.
If $\Delta(t)=t^2-1$, then $(-1,1)\subset[a,b]$. This implies that $[-1,1]\subset[a,b]$.  

Applying Theorem~\ref{theorem:main-non-exceptional-equivalence}, we obtain a $G$-equivariant commutative diagram
\[
\xymatrix{
S\ar@{->}[d]_{\pi_S}\ar@{-->}[rr]&&S'\ar@{->}[d]^{\pi_{S'}}\\
\p^1_\R\ar@{=}[rr]&&\mathbb{P}^1_\R}
\]
such that $S\dashrightarrow S'$ is birational and $S'$ is a smooth projective surface, $\pi_{S'}$ is a $G$-minimal conic bundle and the quasi-projective surface $Y'=S'\setminus \pi_{S'}^{-1}([1:0])$ is given in $\mathbb{P}^2_\mathbb{R}\times\mathbb{A}^1_{\mathbb{R}}$ by
$
x^2+\Delta(t) y^2+(t-a)(t-b)z^2=0.
$
Observe that $(\frac{x}{z},\frac{y}{z})\mapsto \frac{y}{z}$ is rational and $G$-equivariant,  its generic fibre is a conic and the induced action on the base is non-trivial. Theorem~\ref{thm: Trepalin involutions} and Lemma~\ref{rem:K2=4-and-rkPic=2} and $K_{S'}^2=4$ imply that $\tau$ is a twisted Trepalin involution.

Indeed, the involution $\tau$ acts on the surface $Y'$ by
$
([x:y:z],t)\mapsto([x:y:-z],t),
$
and $\pi_{S'}$ induces the map $Y'\to\AAA^1_\R$, $([x:y:z],t)\to t$. 
Let $U$ in $Y'$ be the $G$-invariant subset given by $y\neq0$. Then $U$ is given in $\AAA^3_\R$ by
$
x^2+\Delta(t)+(t-a)(t-b)z^2=0,
$
where we consider $x,z,t$ as affine coordinates on $\AAA^3_\R$. Introducing new the affine coordinate $y=(t-a)z$, we see that $U$ is a complete intersection in $\AAA^4_\R$ given by
\[
\begin{cases}
	x^2+\Delta(t)+y(y+z(b-a))=0,\\
	y=(t-a)z.
\end{cases}
\]
Observe that $U$ admits a $G$-equivariant compactification $X\subset\p^4_\R$ given by 
\[
X\colon \begin{cases}
	x^2+\widetilde\Delta(t,w)+y(y+z(b-a))=0,\\
	yw=(t-aw)z.
\end{cases}
\]
where $\widetilde\Delta(t,w)=t^2+w^2$ if $\Delta(t)=t^2+1$ and $\widetilde\Delta(t,w)=t^2-w^2$ if $\Delta(t)=t^2-1$, and $\tau$ acts on $X$ as follows
$
[x:y:z:t:w]\mapsto [x:-y:-z:t:w],
$
where $[x:y:z:t:w]$ are homogeneous coordinates on $\p^4_\R$. 

Observe that $\Pic(X)^G\simeq\Z^2$, and $X$ admits two conic bundle structures, namely $\pi_1,\pi_2\colon X\to\p^1_\R$ given by
$\pi_1\colon [x:y:z:t:w]\mapsto [y:z]$ and	$\pi_2\colon  [x:y:z:t:w]\mapsto [z:w]$. Then $\pi_1$ and $\pi_2$ are both $G$-minimal conic bundles. 
In fact, we have constructed the following $G$-equivariant commutative diagram:
\[
\xymatrix{
S\ar@{->}[d]_{\pi_S}\ar@{-->}[rr]&&S'\ar@{->}[d]^{\pi_{S'}}\ar@{-->}[rr]&& X\ar@{->}[d]^{\pi_1}\ar[rr]^{\pi_2}&&\p^1_\R\\
\p^1_\R\ar@{=}[rr]&&\mathbb{P}^1_\R\ar@{=}[rr]&&\mathbb{P}^1_\R&&
}
\]
Furthermore, $\tau$ fixes pointwisely the conic
$
C=\{
	y=z=x^2+\widetilde\Delta(t,w)=0
\}
$
and the two points $p_1=[0:0:1:0:0]$ and $p_2=[0:a-b:1:0:0]$. 
Note that $\tau$ acts non-trivially on the base of $\pi_2$, 
and $C$ is the fibre of $\pi_2$ over $[0:1]$, and the points $p_1$ and $p_2$ are contained in the fibre $\pi_2^{-1}([1:0])$. Therefore, we are in the setting of Section~\ref{subsec: Trepalin classification}.
It follows from Theorem~\ref{thm: Trepalin involutions} that $\tau$ is conjugate to a $l$-twisted Trepalin involution, $l\in\{0,1,2\}$. 
We have 
$\tau\in\mathfrak{T}_4$, since $\tau$ fixes no non-real points. 
\end{proof}

\noindent{\bf Summary on conic bundle case:} Thus in what follows we assume that $F(\tau)$ is a smooth curve of genus $g\geqslant 1$. In particular, $\tau$ acts trivially on the base $\p_\R^1$. If $S$ is $G$-exceptional then $\tau$ is a de Jonqui\`{e}res involution of genus $g$, studied in Section \ref{section:exceptional-conic-bundles}, i.e. it belongs to the class $\mathfrak{dJ}_g$. Now assume that $\pi: S\to\p_\R^1$ is not $G$-exceptional. 
Then we are in the setting of Section~\ref{section:non-exceptional}. Therefore, the involution $\tau$ is a $\delta_S$-twisted Iskovskikh involution (see Section~\ref{subsection:Iskovskikh-involutions}), where $\delta_S\in\{0,1,2\}$ is the number of special fibres of $\pi$ (see Section \ref{subsection:special-fibres}). So, $\iota$ is contained in one of the classes $\mathfrak{I}_{g}$, $\mathfrak{I}_{g}^\prime$, $\mathfrak{I}_{g}^{\prime\prime}$. This completes the proof of the classification part of Main Theorem.

\subsection{Proof of Main Theorem: conjugation}
\label{subsection:conjugation}

Now, let $\iota$ and $\iota^\prime$ be involutions in $\Bir(\p^2_\R)$.
We already know from Section~\ref{subsection:main-theorem} 
that $\iota$ and $\iota^\prime$ are  contained in one of the classes 
$\mathfrak{L}$, $\mathfrak{Q}$, $\mathfrak{T}_{4n}$, $\mathfrak{T}_{4n+2}^\prime$, $\mathfrak{T}_{4n}^{\prime\prime}$, 
$\mathfrak{B}_4$, $\mathfrak{G}_3$, $\mathfrak{K}_1$, $\mathfrak{dJ}_{g}$,
$\mathfrak{I}_{g}$, $\mathfrak{I}_{g}^\prime$, $\mathfrak{I}_{g}^{\prime\prime}$, where $n\geqslant 1$ and $g\geqslant 1$. We now show that these classes are actually disjoint, with one possible exception: involutions in $\mathfrak{dJ}_{1}$ and $\mathfrak{I}_{1}$  may be conjugate (see Example~\ref{example:non-exceptional-exceptional}).

We first explore possible conjugations between involutions fixing an irrational curve, i.e. for which $F(\tau)\ne\varnothing$. The classes $\mathfrak{B}_4$, $\mathfrak{G}_3$ and $\mathfrak{K}_1$ are pairwise different, because they fix non-isomorphic curves. 
Moreover, $\mathfrak{B}_4$, $\mathfrak{G}_3$ are disjoint from the classes of de Jonqui\`{e}res and Iskovskikh involutions with $g\geqslant 1$, as the latter two fix a (hyper)elliptic curve and the former two fix non-hyperelliptic curves (the Geiser involution fixes a plane quartic curve and the Bertini involution fixes the intersection of a quadric cone and a cubic in $\p^3_\R$). The class $\mathfrak{K}_1$ is disjoint from $\mathfrak{dJ}_1$, $\mathfrak{I}_1$, $\mathfrak{I}_1^{\prime}$ and $\mathfrak{I}_1^{\prime\prime}$ by Theorem \ref{theorem:Manin-Segre}. By Lemma \ref{lemma:Iskovskikh-involutions-conjugation}, the classes $\mathfrak{dJ}_{g}$, $\mathfrak{I}_{g}$,
$\mathfrak{I}_{g}^\prime$, $\mathfrak{I}_{g}^{\prime\prime}$ with $g\geqslant 1$ are non-intersecting, with the only possible exception $\mathfrak{dJ}_1\cap\mathfrak{I}_1\ne\varnothing$, mentioned in the statement (see Example \ref{example:non-exceptional-exceptional} below).

It remains to investigate the conjugacy between involutions, for which $F(\tau)=\varnothing$; those include the classes $\mathfrak{L}, \mathfrak{Q}$ and Trepalin involutions $\mathfrak{T}_{4n}$, $\mathfrak{T}_{4n+2}^{\prime}$, $\mathfrak{T}_{4n}^{\prime\prime}$, $n\geqslant 1$. By Proposition~\ref{prop: K2=8}, the classes $\mathfrak{L}$ and $\mathfrak{Q}$ are disjoint. Lemma~\ref{rem:K2=4-and-rkPic=2} implies that Trepalin involutions are neither linearizable, nor conjugate to an involution in $\mathfrak{Q}$. The same lemma implies that $\mathfrak{T}_{4n}\cap\mathfrak{T}_{4m+2}^{\prime}=\varnothing$ and $\mathfrak{T}_{4m+2}^{\prime}\cap\mathfrak{T}_{4n}^{\prime\prime}=\varnothing$ for all $n,m\geqslant 1$. Similarly, Theorem \ref{theorem:Iskovskikh} and Lemma~\ref{rem:K2=4-and-rkPic=2}(\ref{K2=4-and-rkPic=2:2}) imply that $\mathfrak{T}_{4n}\cap \mathfrak{T}_{4m}^{\prime\prime}=\varnothing$ for all $n,m\geqslant 1$ with only possible exception $n=m=1$.

Finally, let us show that $\mathfrak{T}_4\cap\mathfrak{T}_4''=\varnothing$. Indeed, let $\iota\in\mathfrak{T}_4''$. By Theorem~\ref{thm: Trepalin involutions}, we can assume that it is regularised as $\tau$ on a conic fibration $\pi_{\widehat S''}\colon \widehat S''\to\p_\R^1$ as in (\ref{diag: Trepalin 3}).
Recall that $(\widehat S'')^\tau$ consists of two pairs of complex conjugate points and each of these pairs is contained in a $\tau$-invariant fibre of $\pi_{\widehat S''}$. This property is preserved by all possible $G$-equivariant Sarkisov links starting from $\widehat S''$, see Lemma~\ref{rem:K2=4-and-rkPic=2}(\ref{K2=4-and-rkPic=2:4}). It implies that we cannot conjugate $\tau$ to an involution from $\mathfrak{T}_4$ (as the $G$-fixed locus of elements of $\mathfrak T_4$ are $\R$-rational curves, see Section~\ref{subsection:Trepalin-involution-1}).

\begin{remark}
Suppose $\tau$ is a Trepalin involution acting regularly on a $\langle\tau\rangle$-conic bundle $S\to\p^1_\R$ with $S(\R)\approx \mathbb{S}^2$. Then, if $S\dashrightarrow S'$ is any  $\langle\tau\rangle$-birational map to a $\langle\tau\rangle$-conic bundle $S'\to\p^1$, then $S'(\R)\approx \mathbb{S}^2$ which again follows from Lemma~\ref{rem:K2=4-and-rkPic=2} and $K_S^2\leqslant4$ even.
\end{remark}

Let us show that $\mathfrak{dJ}_1\cap\mathfrak{I}_1\neq\varnothing$.
 
\begin{example}\label{example:non-exceptional-exceptional}
Let $\omega\colon S\to\p^1_\R\times\p^1_\R$ be a double cover branched over the smooth genus 1 curve $R$ that is given by
$(x_0^2+x_1^2)(y_0^2-y_1^2)=x_1^2y_1^2$, let $C$ be the preimage of the curve $R$ via $\omega$,
let $\tau$ be the~involutions of the double cover $\omega$, let $G=\langle\tau\rangle$,
where $([x_0:x_1],[y_0:y_1])$ are coordinates on $\p^1_\R\times\p^1_\R$.
Then $\mathrm{Pic}(S)^G\simeq\mathbb{Z}^2$, $F(\tau)=C\simeq R$, and $R(\R)\approx\mathbb{S}^1\sqcup\mathbb{S}^1$. Thus $S(\R)\approx\mathbb{T}^2$ and $S$ is $\mathbb{R}$-rational by Theorem \ref{theorem:rational-real}.
Let $\mathrm{pr}_i\colon \p^1_\R\times\p^1_\R\to\p^1_\R$ be the~projection to the $i$-th factor.
Set $\pi_1=\mathrm{pr}_1\circ\omega$ and $\pi_2=\mathrm{pr}_2\circ\omega$.
Then $\pi_1$ and $\pi_2$ are $G$-minimal conic bundles without special fibres.
But $\pi_1(S(\mathbb{R}))=\p^1_\R$, while $\pi_2(S(\mathbb{R}))$ is an interval in $\p^1_\R$,
which implies that the conic bundle $\pi_1$ is $G$-exceptional, but $\pi_2$ is not $G$-exceptional. 
In particular, $\tau\in\mathfrak{dJ}_1\cap\mathfrak{I}_1$. 
\end{example}

Let us discuss conjugation between involution of a given class. 
If $\tau,\tau'$ are two involution in the class $\mathfrak{B}_4$, they are regularised on smooth del Pezzo surfaces of degree $1$ and they are conjugate if and only if the surfaces are equivariantly isomorphic, see Corollary~\ref{corollary:Bertini-curve}. 
This allows to distinguish them with linear algebra. 
Similarly, two involutions in classes $\mathfrak{G}_3$ and $\mathfrak{K}_1$ are regularised on smooth del Pezzo surfaces of degree $2$ and they are conjugate if and only if the surfaces are equivariantly  isomorphic, see Section~\ref{section:Geiser}. 

We can also distinguish involutions up to conjugacy in $\mathfrak{T}_{4n}$, $n\geqslant 2$. Namely, let $\tau$ and $\tau'$ be two $0$-twisted Trepalin involutions. They can be regularised on the affine surfaces $Y$ and $Y'$, respectively, given by
\[
Y\colon\begin{cases}
&x^2+y^2+\prod_{i=1}^{n}(t-\epsilon_i)=0,\\
&w^2+(t-\lambda_1)(t-\lambda_2)=0,
\end{cases}
\quad\text{and}\quad
Y'\colon\begin{cases}
&x^2+y^2+\prod_{i=1}^{n}(t-\epsilon_i')=0,\\
&w^2+(t-\lambda_1')(t-\lambda_2')=0,
\end{cases}
\]
where $\tau$ and $\tau'$ act by $(x,y,z,w)\mapsto(x,y,z,-w)$. If $n\geq2$, then $\tau$ and $\tau'$ are conjugate if and only if there exist $\varphi\in\PGL_2(\R)$ that maps $\{\lambda_1,\lambda_2\}$ to $\{\lambda_1',\lambda_2'\}$ and $\{\epsilon_1,\dots,\epsilon_n\}$ onto $\{\epsilon_1',\dots,\epsilon_n'\}$, as follows from Theorem~\ref{theorem:Iskovskikh}. 
The same argument allows us to distinguish involutions in $\mathfrak{T}_{4n+2}'$ if $n\geqslant 1$ and in $\mathfrak{T}_{4n}''$ if $n\geqslant2$. 

Involutions in $\mathfrak{dJ}_g$ are uniquely determined by their fixed curve if $g\geqslant 2$, see Theorem~\ref{theorem:exceptional-conic-bundles-n-large}. Finally, recall that Theorem~\ref{theorem:main-non-exceptional-equivalence} gives an effective way to distinguish conjugacy classes inside $\mathfrak{I}_g$, $g\geqslant 2$, inside $\mathfrak{I}_g'$, $g\geqslant 2$ and inside $\mathfrak{I}_g''$, $g\geqslant 1$. For the remaining classes $\mathfrak{T}_4''$, $\mathfrak{dJ}_1$, $\mathfrak{I}_1$ and $\mathfrak{I}_1'$ we do not know how to decide whether two involutions inside a class are conjugate. 

\section{Proof of the Main Corollary}\label{subsection:main-corollary}

Let $C$ be a~real hyperelliptic curve of genus $g\geqslant 2$ such that
the~real locus $C(\R)$ consists of at least $2$ connected components.
Let us show that $\Bir(\p_\R^2)$ contains uncountably many non-conjugate $2$-twisted Iskovskikh involutions in the class $\mathfrak{I}_g^{\prime\prime}$
that fix a~curve isomorphic to~$C$.
Observe that $C$ is birational to the curve in $\mathbb{A}^2_\R$ which is given by
$
w^2=-4f(t)
$
for
$$
f(t)=\prod_{i=1}^{2r}\big(t-\epsilon_i\big)\prod_{i=1}^{s}\big(t-\beta_i\big)\big(t-\bar{\beta}_i\big)
$$
where $\epsilon_1<\epsilon_2<\cdots<\epsilon_{2r}$ are real numbers,
$\beta_1,\ldots,\beta_s$ are complex non-real numbers, $r\geqslant 2$, $s\geqslant 0$,
and $(w,t)$ are coordinates on $\mathbb{A}^2_\R$. Note that $g=r+s-1$. Let $a$ and $b$ be real numbers such that $a<\epsilon_1$ and $b>\epsilon_{2r}$,
and let $Y_{a,b}$ be the surface in $\mathbb{P}^2_\mathbb{R}\times\mathbb{A}^1_{\mathbb{R}}$ that is given by
$
x^2+f(t)y^2+(t-a)(t-b)z^2=0,
$
where $([x:y:z],t)$ are coordinates on $\mathbb{P}^2_\mathbb{R}\times\mathbb{A}^1_{\mathbb{R}}$.
Let $\tau_{a,b}$ be the involution in $\mathrm{Aut}(Y_{a,b})$ given by
$
([x:y:z],t)\mapsto([x:y:-z],t),
$
let $G=\langle\tau_{a,b}\rangle$,
and let $\pi_{Y_{a,b}}\colon Y_{a,b}\to\mathbb{A}^1_{\mathbb{R}}$ be $G$-equivariant morphism   given by $([x:y:z],t)\mapsto t$.
Then $Y_{a,b}$ is smooth, and there exists $G$-equivariant commutative diagram
$$
\xymatrix{
Y_{a,b}\ar@{->}[d]_{\pi_{a,b}}\ar@{^{(}->}[rr]&&X_{a,b}\ar@{->}[d]^{\eta_{a,b}}\\
\mathbb{A}^1_\R\ar@{^{(}->}[rr]&&\mathbb{P}^1_\R}
$$
where $X_{a,b}$ is a real smooth projective surface, $\mathbb{A}^1_\R\hookrightarrow\mathbb{P}^1_\R$ and $Y_{a,b}\hookrightarrow X_{a,b}$ are open immersions
such that the embedding $\mathbb{A}^1_\R\hookrightarrow\mathbb{P}^1_\R$ is given by $t\mapsto[t:1]$,
and $\eta_{a,b}$ is a $G$-minimal conic bundle such that the fibre $\eta_{a,b}^{-1}([1:0])=X_{a,b}\setminus Y_{a,b}$ is a smooth conic without real points.

\begin{remark}\label{remark:scrolls}
We can construct $X_{a,b}$ by taking any smooth projective completion of $Y_{a,b}$
such that the $G$-action extends to a biregular action,
and the rational map to $\mathbb{P}^1_\R$ induced by $\pi_{a,b}$ is morphism, and then applying relative $G$-equivariant Minimal Model program over $\mathbb{P}^1_\R$.
Alternatively, we can construct $X_{a,b}$ as a hypersurface in the scroll $\mathbb{F}(0,r+s,1)$.
Namely, recall from \cite{Reid} that the~real scroll $\mathbb{F}(0,r+s,1)$~can be defined as the~quotient
$(\mathbb{A}^3_\R\setminus 0)\times(\mathbb{A}^2_\R\setminus 0)/\mathbb{G}_m^2$
for the following $\mathbb{G}_m^2$-action:
$$
\Big(\big(x,y,z\big),\big(t_0,t_1\big)\Big)\mapsto\Bigg(\Big(\mu x,\frac{\mu y}{\lambda^{r+s}},\frac{\mu z}{\lambda}\Big),\big(\lambda t_0,\lambda t_1\big)\Bigg)
$$
where $(\lambda,\mu)\in\mathbb{G}_m^2$, and $((x,y,z),(t_0,t_1))$ are coordinates on $\mathbb{A}^2_\R\times\mathbb{A}^3_\R$.
Therefore, we can define the~surface $X_{a,b}$ as a hypersurface in the scroll $\mathbb{F}(0,r+s,1)$ that is given by
$$
x^2+F(t_0,t_1)y^2+(t_0-at_1)(t_0-bt_1)z^2=0,
$$
where $F(t_0,t_1)=t_1^df(\frac{t_0}{t_1})$,
and $([x:y:z],[t_0:t_1])$ are bihomogeneous coordinates on $\mathbb{F}(0,r+s,1)$.
Then $Y_{a,b}\hookrightarrow X_{a,b}$ is given by $t_1\ne 0$,
the $G$-action on $X_{a,b}$ is given by
$$
\big([t_0:t_1],[x:y:z]\big)\mapsto\big([t_0:t_1],[x:y:-z]\big),
$$
and the~conic bundle $\eta_{a,b}\colon X_{a,b}\to \mathbb{P}^{1}_{\mathbb{R}}$ is given by $([t_0:t_1],[x:y:z])\mapsto[t_0:t_1]$.
\end{remark}

Note  that $X_{a,b}(\R)=\eta_{a,b}^{-1}([a,b])$ and  $X_{a,b}(\R)$ is connected, which implies that $X_{a,b}$ is $\R$-rational.
Let $\iota_{a,b}$ be a birational involution in $\Bir(\p^2_\R)$ given by $\tau_{a,b}$ and some birational map $X_{a,b}\dasharrow\p^2_\R$.
Then $F(\iota_{a,b})\cong C$ by Corollary~\ref{corollary:fixed-curve-non-exceptional}, so that
$\iota_{a,b}$ is a $2$-twisted Iskovskikh involution in the class $\mathfrak{I}_{g}^{\prime\prime}$,
because $\eta_{a,b}$ has two special singular fibres: the fibres
$\eta_{a,b}^{-1}([a:1])$ and $\eta_{a,b}^{-1}([b:1])$.

The surface $X_{a,b}$ has no real point over $[1:0]$ thus its real locus is identical to the real locus of $Y_{a,b}$ and then diffeomorphic to the connected sum of $2r$ real projective planes. Indeed, forgetting the action of $\tau_{a,b}$, we can contract over $\R$ one real $(-1)$-component of each of the $2r$ real singular fibres and obtain a surface whose real locus is diffeomorphic to the sphere, then  $X_{a,b}(\R)=\#^{2r}\R\p^2$.

\begin{proposition}
\label{proposition:invinitely-many-involutions}
Let $a_1,b_1,a_2,b_2$ be general real numbers such that $a_1,a_2<\epsilon_1$ and $b_1,b_2>\epsilon_{2r}$. Suppose that $g(C)\geqslant 2$. 
Then the involutions $\tau_{a_1,b_1}$ and $\tau_{a_2,b_2}$ are not conjugate in $\Bir(\p^2_\R)$.
\end{proposition}

\begin{proof}
The involutions $\tau_{a_1,b_1}$ and $\tau_{a_2,b_2}$ are conjugate  if and only if
there is a $G$-equivariant birational map $X_{a_1,b_1}\dasharrow X_{a_2,b_2}$.
Suppose such birational map exists. Since \mbox{$K_{X_{a_1,b_1}}^2=K_{X_{a_2,b_2}}^2=4-2g\leqslant 0$},
it follows from Theorem~\ref{theorem:Iskovskikh}
that there is a $G$-equivariant commutative diagram
$$
\xymatrix{
X_{a_1,b_1}\ar@{->}[d]_{\eta_{a_1,b_1}}\ar@{-->}[rr]&&X_{a_2,b_2}\ar@{->}[d]^{\eta_{a_2,b_2}}\\
\mathbb{P}^1_\R\ar@{->}[rr]_{\phi}&&\mathbb{P}^1_\R}
$$
for some  $\phi\in\mathrm{PGL}_2(\R)$. Let
\begin{gather*}
\Sigma=\big\{[\epsilon_1:1],[\epsilon_2:1],\ldots,[\epsilon_{2r}:1],[\beta_1:1],[\bar{\beta}_1:1],\ldots,[\beta_s:1],[\bar{\beta}_s:1]\big\},\\
\Sigma_{1}=\big\{[a_1,1],[b_1,1]\big\},\ \ \
\Sigma_{2}=\big\{[a_2,1],[b_2,1]\big\}.
\end{gather*}
It follows from the proof of Theorem~\ref{theorem:Iskovskikh}
or from the descriptions of two-dimensional Sarkisov links that
the birational map $X_{a_1,b_1}\dasharrow X_{a_2,b_2}$ is an isomorphism along singular fibres of the conic bundles $\eta_{a_1,b_1}$ and $\eta_{a_2,b_2}$.
Therefore, we have $\phi(\Sigma)=\Sigma$ and $\phi(\Sigma_1)=\Sigma_2$. Since the set $\Sigma$ consists of at least $6$ points in $\mathbb{P}^1_\C$,
it follows from $\mathrm{Aut}(\mathbb{P}^1_\R,\Sigma)$ is a finite group.
So, since $\phi\in\mathrm{Aut}(\mathbb{P}^1_\R,\Sigma)$,
we have finitely many possibilities for the subset $\phi(\Sigma_1)$,
and we may assume that $\Sigma_2$ is not among these finitely many possibilities,
since $a_2$ and $b_2$ are general by assumption.
Then $\phi(\Sigma_1)\ne\Sigma_2$, which is a contradiction.
\end{proof}

\vskip\baselineskip

Now let us construct an uncountable number of non-conjugate involutions in $\Bir(\p^2_\R)$ that fix no geometrically irrational curve. Let  $r\geqslant 2$ and fix real numbers $a<\epsilon_1<b<\epsilon_2<\cdots<\epsilon_{2r}$. Consider the affine surface $U\subset\AAA^4_\R$ given by  
$$
\left\{\aligned
&x^2+y^2+(t-\epsilon_1)(t-\epsilon_2)\cdots(t-\epsilon_{2r})=0,\\
&w^2+(t-a)(t-b)=0,
\endaligned
\right.
$$
and the projectivisation $S$ of $U$ described in Section~\ref{subsection:Trepalin-involution-1}. As explained in the construction \ref{subsection:Trepalin-involution-1}, $S$ is $\R$-rational and there is a conic fibration $S\to\p^1_\R$ on $S$ that has $4r\geqslant 8$ singular fibres and smooth fibres above $a$ and $b$. The involution 
$[w:x:y:z]\mapsto[-w:x:y:z]$
of $U$ extends to a regular involution $\tau_{a,b}$ on $S$ such that $\Pic(S)^{\langle\tau\rangle}\simeq\Z^2$ and $F(\tau_{a,b})=\varnothing$. In fact, the set of 
$\tau_{a,b}$-fixed points are the fibres above $a$ and $b$, and $\tau_{a,b}$ is a $0$-twisted Trepalin involution.

\begin{proposition}\label{proposition:invinitely-many-involutions-rational-curve}
Let $a_1,a_2,b_1,b_2$ be general real numbers such that $a_1,a_2<\epsilon_1<b_1,b_2<\epsilon_2$. Then the involutions $\tau_{a_1,b_1}$ and $\tau_{a_2,b_2}$ are not conjugate in $\Bir(\p^2_\R)$. 
\end{proposition}

\begin{proof}
We can follow the proof of Proposition~\ref{proposition:invinitely-many-involutions} almost word by word. 
\end{proof}

Similar uncountable families can be constructed with $1$-twisted and $2$-twisted Trepalin involutions. 
\bigskip

\appendix

\section{Quadratic forms}
\label{section:Witt-ring}

In this appendix, we present some techniques from the theory of quadratic forms.

\subsection{Witt's ring of a~field}\label{ss:Witt}

Let $\kk$ be a~base field of characteristic not equal to two and $V$ be a~vector space of dimension $n$ over $\kk$. A {\it quadratic space} is a~pair $(V,q)$ where $q: V\to\kk$ is a~quadratic form. Given two quadratic spaces $(V_1,q_1)$ and $(V_2,q_2)$, one can form their {\it orthogonal sum} $V_1\perp V_2$ which is a~quadratic space $(V_1\oplus V_2,q_1\perp q_2)$ where $(q_1\perp q_2)(v_1,v_2)=q_{1}(v_1)+q_{2}(v_2)$. Further, we have an operation of tensor product $(V_1\otimes V_2,q_1\otimes q_2)$ of these two quadratic spaces, where $q_1\otimes q_2$ is defined on elementary tensors by $(q_1\otimes q_2)(v_1\otimes v_2)=q_1(v_1)q_2(v_2)$. Since every quadratic form can be brought to the diagonal form $a_1x_1^2+a_2x_2^2\ldots +a_nx_n^2$, $a_i\in\kk^*$, we shall denote quadratic forms $\langle a_1,a_2,\ldots,a_n\rangle$; the latter is   equal to $\langle a_1\rangle\perp\langle a_2\rangle\perp\ldots\perp\langle a_n\rangle$.

Let $\QForms(\kk)$ be the set of all isometry classes of non-singular quadratic forms over $\kk$. The operations of orthogonal sum $\perp$ and tensor product $\otimes$ make $\QForms(\kk)$ into commutative semiring: note that no non-zero element of $\QForms(\kk)$ has an additive inverse. However, Witt's Cancellation Theorem \cite[I.4.2]{Lam} states that $\QForms(\kk)$ is a~cancellation monoid, i.e. $q\perp q_1\cong q\perp q_2$ always implies $q_1\cong q_2$. This allows to embed $\QForms(\kk)$ into a~group via the standard Grothendieck construction. Namely, define a~relation on $\QForms(\kk)\times\QForms(\kk)$ by setting
\[
(q_1,q_2)\sim(q_1',q_2')\ \ \ \text{if and only if}\ \ \ q_1\perp q_2'=q_1'\perp q_2.
\]
The cancellation law implies that $\sim$ is an equivalence relation and there is a~well-defined addition on the set of equivalence classes
\[
[(q_1,q_2)]+[(q_1',q_2')]=[(q_1\perp q_1',q_2\perp q_2')]
\]
which makes $(\QForms(\kk)\times\QForms(\kk)/\sim)$ into a~commutative ring $\widehat{\Witt}(\kk)$ called the {\it Grothendieck-Witt} ring of quadratic forms over $\kk$. The map $\iota: \QForms(\kk)\to\widehat{\Witt}(\kk)$, $\iota(q)=(q,0)$, is an injection, so we can view $\QForms(\kk)$ as a~subset of $\widehat{\Witt}(\kk)$. We can write $(q_1,q_2)=\iota(q_1)-\iota(q_2)=q_1-q_2$, hence saying that $q_1$ and $q_2$ are equal as elements of $\widehat{\Witt}(\kk)$ is equivalent to saying that $q_1$ and $q_2$ are equivalent as quadratic forms.

The quadratic form $\mathbf{h}=\langle 1,-1\rangle$ is called the {\it hyperbolic plane}, and an orthogonal sum of hyperbolic planes is called a~{\it hyperbolic space}. It is not difficult to see that the set of all hyperbolic spaces and their additive inverses constitutes an ideal in $\widehat{\Witt}(\kk)$ denoted $\mathbb{Z}\mathbf{h}$. The quotient ring $\Witt(\kk)=\widehat{\Witt}(\kk)/\mathbb{Z}\mathbf{h}$ is called the {\it Witt ring} of $\kk$.

\begin{remark}\label{rem: Witt of R}
	It is easy to see that the elements of $\Witt(\kk)$ are in one-to-one correspondence with the isometry classes of all anisotropic forms over $\kk$. For example, over $\kk=\R$ there are only two classes of anisotropic forms of a~given dimension $n>1$, namely $n\langle 1\rangle$ and $n\langle -1\rangle$. In particular, the map
	\begin{equation}\label{eq: Witt ring of R}
		\psi: \Witt(\R)\to\mathbb{Z},\ \ \ \ q\mapsto\begin{cases}
			\ \ n,\ \ \text{if}\ q\cong n\langle 1\rangle,\\
			-n,\ \ \text{if}\ q\cong n\langle -1\rangle
		\end{cases}
	\end{equation}
	is an isomorphism.
\end{remark}

Two $n$-ary quadratic forms $q_1$ and $q_2$ on $V$ are called {\it equivalent} if there exists $A\in\mathrm{GL}(V)$ such that $q_1(v)=q_2(Av)$ for all $v\in V$. Now we can formulate the following classical result which is due to Witt (see e.g. \cite[II.1.4]{Lam} for a~modern exposition).

\begin{theorem}\label{thm: Witt equivalence}
	Two non-degenerate quadratic forms over a~field $\kk$ are equivalent if and only if they have the same dimension and equal classes in $\Witt(\kk)$.
\end{theorem}

Let $(\FFF,\nu)$ be a~discretely valuated field and $A=\{x\in \FFF: \nu(x)\geqslant 0 \}$ be the corresponding valuation ring (one sets $\nu(0)=\infty$, so $0\in A$). This is a~local ring with maximal ideal $\mathfrak{m}=\{x\in\FFF: \nu(x)\geqslant 1\}$ generated by any element $\pi$ with $\nu(\pi)=1$; such $\pi$ is called a~{\it uniformizer} of $A$. The field $\overline{\FFF}=A/\mathfrak{m}$ is called the {\it residue class field} of $A$. Recall that every $y\in\FFF^*$ can be written uniquely in the form $y=u\pi^{\nu(y)}$ where $\pi$ is a~fixed uniformizer and $u$ belongs to the ring of units $\Units(A)=\{x\in\FFF^*: \nu(x)=0\}$. This implies that every quadratic form over $\FFF$ can be written as
\begin{equation}\label{eq: quadratic forms over dvf}
	q=q_1\perp \langle\pi\rangle q_2, \ \ \ \text{where}\ \ q_1=\langle u_1,\ldots,u_r\rangle,\ q_2=\langle u_{r+1},\ldots,u_n\rangle,\ \ u_i\in\Units(A).
\end{equation}
Now Springer's theorem \cite[VI.1.4]{Lam} implies that the classes of $\overline{q}_1=\langle \overline{u}_1,\ldots,\overline{u}_r\rangle$ and $\overline{q}_2=\langle \overline{u}_{r+1},\ldots,\overline{u}_n\rangle$ are uniquely determined in $\Witt(\overline{\FFF})$ for a~given form $q$. They are called the {\it first and second residue forms} of $q$, respectively.

\subsection{Quadratic forms over $\R(t)$}
\label{ss:quadratic forms equivalence}

In section \ref{section:non-exceptional}, we will need to find out when two quadratic forms in three variables over $\R(t)$ are $G$-equivariantly equivalent. So, let us start with the following easy

\begin{lemma}\label{lemma:equivalence-of-forms}
	Suppose that $G$ act on $\R(t)^3$ by $(x,y,z)\mapsto(x,y,-z)$.
	Let $A,B,C,D,E,F\in\R[t]$ be polynomials with no multiple roots, such that $A,B,E$ have no common divisors and $C,D,F$ have no common divisors. Then the forms $\langle A,B,E\rangle$ and $\langle C,D,F\rangle$ are $G$-equivariantly equivalent if and only if  $\langle A,B\rangle$ and $\langle C,D\rangle$ are equivalent over $\R(t)$ and $F=\mu E$ for some $\mu>0$.
\end{lemma}
\begin{proof}
	If the forms $\langle A,B,E\rangle$ and $\langle C,D,F\rangle$ are $G$-equivariantly equivalent, then the equivalence is given by an automorphism of ${\R(t)}^3$ of the form $\phi\colon (x,y,z)\mapsto (ax+by, cx+dy,\lambda z)$ with $a,b,c,d,\lambda\in\R(t)$.
	Then $\langle A,B\rangle$ and $\langle C,D\rangle$ are equivalent and $ F=\lambda^2 E$. Since $E$ and $F$ are both polynomials without multiple roots, it follows that $\lambda\in\R^*$.
	
	Suppose that $\langle A,B\rangle$ and $\langle C,D\rangle$ are equivalent over $\R(t)$ and $F=\mu E$ for some $\mu>0$. If the map $(x,y)\mapsto(ax+by,cx+dy)$, $a,b,c,d\in\R(t)$, induces the equivalence between $\langle A,B\rangle$ and $\langle C,D\rangle$, then the $G$-equivariant isomorphism $(x,y,z)\mapsto(ax+by,cx+dy,\sqrt{\mu}z)$ induces an equivalence between $\langle A,B,E\rangle$ and $\langle C,D,F\rangle$.
\end{proof}

We now provide necessary and sufficient conditions for $\langle A,B\rangle$ and $\langle C,D\rangle$ to be equivalent over $\R(t)$ by applying the above described construction to the particular case $\FFF=\R(t)$. Let $\{\pi\}$ be the set of monic irreducible polynomials in $\R[t]$. Recall that each maximal ideal $(\pi)$ of $\R[t]$ defines a~valuation on the field $\R(t)$ given by $|f|_\pi=c^{-\nu_\pi(f)}$ with $c\in\R_{>1}$ and $\nu_\pi(f)$ being the power of $(\pi)$ dividing $(f)$. Note that different choice of $c$ gives an equivalent valuation. Denote by $\R(t)_\pi$ the $(\pi)$-adic completion of $\R(t)$ with respect to $|\cdot|_\pi$. Define $\partial_\pi$ to be the composition
\[
\Witt(\R(t))\to\Witt(\R(t)_\pi)\to \Witt(\overline{\R(t)_\pi}),
\]
where the second map is the second residue homomorphism with $\pi$ as the uniformizer. Then there is a~{\it Milnor's exact sequence} \cite[IX.3.1]{Lam}
\begin{equation}
	 0\overset{}{\longrightarrow}\Witt(\R)\overset{\iota}{\longrightarrow}\Witt(\R(t))\overset{\oplus\partial_\pi\hspace{0.3cm}}{\longrightarrow}\bigoplus_{\{\pi\}}\Witt\left (\overline{\R(t)_\pi}\right )\overset{}{\longrightarrow} 0
\end{equation}
\begin{lemma}
	One has $\overline{\R(t)_\pi}\cong\R[t]/(\pi)$.
\end{lemma}
\begin{proof}
The valuation ring of $\R(t)_\pi$ is $\R[t]_\pi$, the $\pi$-adic completion of $\R[t]$. Its unique maximal ideal is   $(\pi)_\pi$, the $\pi$-adic completion of $(\pi)$. Then we have
$$
\overline{\R(t)_\pi}\cong\R[t]_\pi/{(\pi)_\pi}\cong\R[t]/(\pi),
$$
where the last isomorphism follows from e.g. \cite[10.15]{Atiyah}.
\end{proof}
In what follows let $\overline{u}$ be the class of $u\in\R[t]$ in $\R[t]/(\pi)$. We can rewrite Milnor's sequence in the form
\begin{equation}\label{eq: Milnor sequence}
	 0\overset{}{\longrightarrow}\Witt(\R)\overset{\iota}{\longrightarrow}\Witt(\R(t))\overset{\oplus\partial_\pi\hspace{0.3cm}}{\longrightarrow}\bigoplus_{\{\pi\}}\Witt\left (\R[t]/(\pi)\right )\overset{}{\longrightarrow} 0
\end{equation}
and the homomorphism $\partial_\pi$ is explicitly given on $1$-forms by
\begin{equation}\label{eq: second residue}
	\partial_\pi\langle v\rangle=\begin{cases}
		0,\ \ \ \ \ \ \ \text{if}\ \pi\ \text{does not divide}\ v,\\
		\langle \overline{v/\pi} \rangle,\  \text{otherwise}.
	\end{cases}
\end{equation}

If $t-\varepsilon$ divides $Q\in\R[t]$ we set $Q_\varepsilon=Q/(t-\varepsilon)$ for brevity. 

\begin{theorem}[{\cite[Satz 23]{Witt}}]\label{thm:Witt}
Two quadratic forms over $\R(t)$ are isomorphic if and only if they have same rank, same determinant and same signature for almost all real values of $t$. 
\end{theorem}

As explained to us by A. Merkurjev, Theorem~\ref{thm:Witt} implies the following technical proposition, whose more algorithmic form makes it easier to use in some applications.

\begin{proposition}
\label{proposition:equivalence-of-forms}
Let $A, B, C, D\in\R[t]$ be polynomials with no multiple roots. Two quadratic forms $\langle A, B\rangle$ and $\langle C, D\rangle$ are equivalent (with respect to $\mathrm{GL}_2(\R(t))$-action) if and only if the following three conditions hold:
\begin{enumerate}
\item[(1)] $ABCD$ is a~square in $\R(t)$ (equivalently, $AB=CD$ in $\R(t)/(\R(t)^*)^2$).
\item[(2)] For every real root $\varepsilon$ of $A$ the following conditions hold
\begin{enumerate}
\item[(2.a)] if $A(\varepsilon)=0$, $B(\varepsilon)\ne 0$  then $A_\varepsilon(\varepsilon)C_\varepsilon(\varepsilon)>0$ when $C(\varepsilon)=0$ and $A_\varepsilon(\varepsilon)D_\varepsilon(\varepsilon)>0$ when $D(\varepsilon)=0$.
\item[(2.b)] if $A(\varepsilon)=B(\varepsilon)=0$ and $C(\varepsilon),D(\varepsilon)\ne 0$ then $A_\varepsilon(\varepsilon)B_\varepsilon(\varepsilon)<0$.
\item[(2.c)] if $A(\varepsilon)=B(\varepsilon)=C(\varepsilon)=D(\varepsilon)=0$ then either $A_\varepsilon(\varepsilon)B_\varepsilon(\varepsilon)<0$ and $C_\varepsilon(\varepsilon)D_\varepsilon(\varepsilon)<0$, or all numbers $A_\varepsilon(\varepsilon), B_\varepsilon(\varepsilon), C_\varepsilon(\varepsilon), D_\varepsilon(\varepsilon)$ have the same sign.
\end{enumerate}
\item[(3)]\label{item3}
If $\lambda_A,\lambda_B,\lambda_C,\lambda_D$ are the leading coefficients of $A,B,C,D$, respectively, then one of the following holds:
\begin{itemize}
		\item Exactly one of $A,B$ and one of $C,D$ has odd degree and
		the leading coefficients of even-degree polynomials have the same sign;
		\item $\deg A$, $\deg B$, $\deg C$, $\deg D$ are all odd;
		\item or $\deg A$, $\deg B$ are odd, $\deg C$, $\deg D$ are even, and $\lambda_C\lambda_D<0$; 
		\item or $\deg A$, $\deg B$ are even, $\deg C$, $\deg D$ are odd, and $\lambda_A\lambda_B<0$;
		\item $\deg A$, $\deg B$, $\deg C$, $\deg D$ are all even
		and either $\lambda_A\lambda_B<0$ and $\lambda_C\lambda_D<0$, or all $\lambda_A$, $\lambda_B$, $\lambda_C$, $\lambda_D$ are of the same sign.
	\end{itemize}
\end{enumerate}
\end{proposition}
\begin{proof}
	By Witt's theorem \ref{thm: Witt equivalence}, we need to check that our two forms have the same class in the Witt ring $\Witt(\R(t))$. For this we are going to use Milnor's exact sequence (\ref{eq: Milnor sequence}). First let us compare the images of our forms under each $\partial_\pi$. If $\deg\pi=2$ then $\R[t]/(\pi)\cong\C$ and $\Witt(\C)$ is isomorphic to $\mathbb{Z}/2\mathbb{Z}$ via the ``parity'' homomorphism which takes a~form $q$ to $0$ if $\dim q$ is even, and to $1$ if $\dim q$ is odd. Since $ABCD$ is a~square in $\R(t)$, the images of $\langle A,B\rangle$ and $\langle C,D\rangle$ under $\partial_\pi$ are the same.
	
	Now let $\pi=t-\varepsilon$, where $\varepsilon\in\R$. Denote by $\vartheta$ the isomorphism \[\vartheta: \R[t]/(t-\varepsilon)\overset{\raisebox{0.25ex}{$\sim\hspace{0.2ex}$}}{\smash{\longrightarrow}}\R,\ \ \overline{u}\mapsto u(\varepsilon).\]
	We have the following cases.
	\begin{enumerate}
		\item $A(\varepsilon)=0$, $B(\varepsilon)\ne 0$. Then by (\ref{eq: second residue}) one has $\partial_\pi\langle B\rangle= 0$. Since $ABCD$ is a~square and $A,B,C,D$ have no multiple factors, exactly one of $C$ and $D$ is divisible by $t-\varepsilon$. Without loss of generality, assume $C(\varepsilon)=0$, $D(\varepsilon)\ne 0$, so $\partial_\pi\langle D\rangle= 0$. We see that the class of $\partial_\pi\langle A\rangle=\langle\overline{A_\varepsilon} \rangle$ in $\Witt(\R[t]/\pi)\cong\Witt(\R)$ is equal to the class of $\partial_\pi\langle C\rangle=\langle\overline{C_\varepsilon} \rangle$ if and only if $\vartheta(\overline{A_\varepsilon})=A_\varepsilon(\varepsilon)$ and $\vartheta(\overline{C_\varepsilon})=C_\varepsilon(\varepsilon)$ have the same sign (see Remark \ref{rem: Witt of R}).
		\item $A(\varepsilon)=B(\varepsilon)= 0$, but $C(\varepsilon)$ and $D(\varepsilon)$ do not equal $0$. Then $\partial_\pi\langle C\rangle=\partial_\pi\langle D\rangle=0$ and we must have $\langle \overline{A_\varepsilon}\rangle=-\langle \overline{B_\varepsilon}\rangle$ in $\Witt(\R[t]/(\pi))$, which exactly means that $A_\varepsilon(\varepsilon)$ and $B_\varepsilon(\varepsilon)$ are of different signs.
		\item $A(\varepsilon)=B(\varepsilon)=C(\varepsilon)=D(\varepsilon)=0$.
		We must have
		$
		\partial_\pi\langle A\rangle+\partial_\pi\langle B\rangle=\partial_\pi\langle C\rangle+\partial_\pi\langle D\rangle
		$
		which, by definition of $\partial_\pi$, is equivalent to the following equality between quadratic 1-forms over $\R[t]/(\pi)$:
		$
		\langle \overline{A_\varepsilon}\rangle+\langle \overline{B_\varepsilon}\rangle=\langle \overline{C_\varepsilon}\rangle+\langle \overline{D_\varepsilon}\rangle.
		$
		This is equivalent, after applying the isomorphism $\vartheta: \R[t]/(\pi)\to\R$, to the equality
		\[
		\langle {A_\varepsilon(\varepsilon)}\rangle+\langle {B_\varepsilon(\varepsilon)}\rangle=\langle {C_\varepsilon(\varepsilon)}\rangle+\langle {D_\varepsilon(\varepsilon)}\rangle.
		\]
		Let $\psi$ be the isomorphism (\ref{eq: Witt ring of R}) $\Witt(\R)\to\mathbb{Z}$. Now for any $z\in\{A_\varepsilon(\varepsilon), B_\varepsilon(\varepsilon),C_\varepsilon(\varepsilon),D_\varepsilon(\varepsilon)\}$ we have $\psi(z)=1$ if $z>0$ and $\psi(z)=-1$ if $z<0$.
		This translates  in the condition that either $A_\varepsilon(\varepsilon)B_\varepsilon(\varepsilon)<0$ and $C_\varepsilon(\varepsilon)D_\varepsilon(\varepsilon)<0$, or all the numbers $A_\varepsilon(\varepsilon), B_\varepsilon(\varepsilon), C_\varepsilon(\varepsilon), D_\varepsilon(\varepsilon)$ have the same sign.
	\end{enumerate}
	Now the map $\iota$ in the exact sequence (\ref{eq: Milnor sequence}) admits a~splitting $\jmath_\infty$ which is defined on $1$-forms as follows: if $q=\langle h(t)\rangle$ and $h$ has degree $d$ and the leading coefficient $\lambda_q$ then $\jmath_\infty(q)=0$ for $d$ odd and $\jmath_\infty(q)=\langle\lambda_q\rangle$ for $d$ even, see \cite[Theorem 21.1, Lemma 19.10]{ElmanKarpenkoMerkurjev}. Note that $AB=CD$ in $\R(t)/(\R(t)^*)^2$ implies that $\deg(AB)$ and $\deg(CD)$ have the same parity. If they are both odd then exactly one element of each pair $A,B$ and $C,D$ has odd degree, and then $\jmath_\infty\langle A,B\rangle=\langle \pm1\rangle=\jmath_\infty\langle C,D\rangle$ if and only if the leading coefficients of the even-degree polynomials have the same sign. If they are both even, then one of the remaining conditions in (3) must be satisfied.
\end{proof}

\begin{remark}
\label{rem: Vishik}
	A slightly different approach to Proposition \ref{proposition:equivalence-of-forms} was suggested to us by A. Vishik. Namely, we have the following statement (again, probably due to Witt): $\langle A,B\rangle\cong \langle C,D\rangle$ if and only if the following two conditions hold:
	\begin{enumerate}
		\item The classes of discriminants $AB$ and $CD$ in $\kk^*/(\kk^*)^2$ coincide;
		\item $\langle A,B\rangle$ represents $C$.
	\end{enumerate}
	The necessity is easy, so let us prove the sufficiency. The diagonalization theorem for quadratic forms implies that, given a~non-degenerate quadratic space $(V,q)$ and any anisotropic vector $v\in V$ there exists an orthogonal basis $(v,e_2,\ldots,e_n)$. Thus if $q$ represents $\theta\in\kk^*$ then $q\cong \theta x_1^2+\alpha_2x_2^2+\ldots + \alpha_nx_n^2$ (this observation implies in particular that if $n=2$ then $q\cong\langle \theta,\theta\cdot\disc{q}\rangle$, a~fact that we shall use below). In our situation we have $\langle A,B\rangle\cong\langle C,E\rangle$, but then condition (1) implies $E=D$.
	
	Now $\langle A,B\rangle$ represents $C$ if and only if $q=\langle A,B,-C \rangle$ is isotropic. Indeed, the latter condition is equivalent to $\mathbf{h}$ being a~subform of $q$. But $\mathbf{h}\cong\langle C,-C\rangle$ (by the observation from above) and therefore $q$ is isotropic if and only if $\langle C,-C\rangle$ is its subform, which is equivalent to $\langle C\rangle$ being a~subform of $\langle A,B\rangle$ by Witt's Cancellation Theorem \cite[I.4.2]{Lam}.
	
	Set $E=A/C$, $F=B/C$. The form $q$ is then proportional to the form $\langle 1,-E,-F\rangle $. Consider Milnor's K-theory $\Milnor_*^M(\kk)=\bigoplus_{n\geqslant 0}\left (\Milnor_1(\kk) \right )^{\otimes n}/J$, where $\Milnor_1(\kk)$ is the group $\kk^*$ written additively (with elements denoted by $\{ a\}$, $a\in\kk^*$) and $J$ is the ideal generated by all elements $\{a \}\otimes\{1-a\}$. Let $a_1,\ldots,a_n\in\kk^*$. The class of $a_1\otimes\ldots\otimes a_n$ in $\Milnor_*^M$ is called a~{\it symbol} $\{a_1,\ldots,a_n\}$. It is well known that the following statements are equivalent:
	\begin{enumerate}
		\item The Pfister form $\langle\langle a_1,\ldots,a_n\rangle\rangle=\langle 1,-a_1\rangle\otimes\ldots\otimes\langle 1,-a_n\rangle $ is isotropic.
		\item The quadratic form $\langle\langle a_1,\ldots,a_{n-1}\rangle\rangle+\langle -a_n\rangle$ is isotropic.
		\item The symbol $\{a_1,\ldots,a_n\}$ is zero in $\Milnor_n^M(\kk)$.
	\end{enumerate}
	We conclude that $q$ is isotropic if and only if the symbol $\{E,F\}$ is zero in $\Milnor_2^M(\kk)$. Now let $\kk=\R(t)$. There is a~short exact sequence similar to (\ref{eq: Milnor sequence})
	\[
	 0\overset{}{\longrightarrow}\Milnor^M_n(\R)\overset{}{\longrightarrow}\Milnor_n^M(\R(t))\overset{}{\longrightarrow}\bigoplus_{\{\pi\}}\Milnor_{n-1}^M(\R[t]/(\pi))\overset{}{\longrightarrow} 0
	\]
	and an analysis similar to the one of Proposition \ref{proposition:equivalence-of-forms} applies. We leave the details to an interested reader.
\end{remark}

\end{document}